\documentclass[final,reqno]{siamltex}
\usepackage{latexsym,amsmath,amssymb,amsfonts,mathrsfs,tikz,adjustbox}
\usepackage{epsf,graphicx,epsfig,color,cite,cases}
\usepackage{subfigure,graphics,multirow,marginnote,enumerate,bm}
\usetikzlibrary{bending,decorations.markings,arrows.meta,calc,decorations.pathmorphing,math}
\newcommand{\Div}{{\rm Div\,}}

\newcommand{\om}{\omega}

\newcommand{\pa}{\partial}

\newcommand{\ov}{\overline}
\newcommand{\I}{{\rm Im}}
\newcommand{\Rt}{{\rm Re}}
\newcommand{\curl}{{\rm curl\,}}
\newcommand{\dive}{{\rm div\,}}

\newcommand{\wid}{\widetilde}

\newcommand{\na}{\nabla}
\newcommand{\mat}{\mathbb}

\newcommand{\R}{{\mat R}}

\newcommand{\C}{{\mat C}}
\newcommand{\Sp}{{\mat S}}

\newcommand{\be}{\begin{eqnarray}}
\newcommand{\ben}{\begin{eqnarray*}}
\newcommand{\en}{\end{eqnarray}}
\newcommand{\enn}{\end{eqnarray*}}

\newtheorem{remark}[theorem]{Remark}

\begin{document}
\renewcommand{\theequation}{\arabic{section}.\arabic{equation}}

\title{\bf elastic scattering by locally rough interfaces}
\author{
Chengyu Wu\thanks{School of Mathematics and Statistics, Xi'an Jiaotong University,
	Xi'an 710049, Shaanxi, China ({\tt wucy99@stu.xjtu.edu.cn})}
\and
Yushan Xue\thanks{School of Statistics and Mathematics, Central University of Finance and Economics, China ({\tt cnxueyushan@163.com})}
\and
Jiaqing Yang\thanks{School of Mathematics and Statistics, Xi'an Jiaotong University,
	Xi'an 710049, Shaanxi, China ({\tt jiaq.yang@mail.xjtu.edu.cn})}
 }
\date{}
\maketitle


\begin{abstract}
  In this paper, we present the first well-posedness result for elastic scattering by locally rough interfaces in both two and three dimensions. Inspired by the Helmholtz decomposition, we discover a fundamental identity for the stress vector, revealing an intrinsic relationship among the generalized stress vector, the Lamé constants and certain tangential differential operators. This identity leads to two key limits for surface integrals involving scattered solutions, from which we deduce the first uniqueness result of direct problem for all frequencies. Through a detailed analysis, applying the steepest descent method, subsequently we derive the existence and uniqueness of the corresponding two-layered Green's tensor along with its explicit expression when the transmission coefficient equals 1. Finally, by leveraging properties of the Green's tensor, we establish the existence of solutions via the variational method and the boundary integral equation, thereby achieving the first well-posedness result for elastic scatteing by rough interfaces. 
\end{abstract}



\pagestyle{myheadings}
\thispagestyle{plain}
\markboth{C. Wu, Y. Xue and J. Yang}{Elastic scattering by locally rough interfaces}

\section{Introduction}\label{sec1}
\setcounter{equation}{0}
The present paper concerns the time-harmonic elastic scattering by a penetrable locally rough surface in both two and three dimensions. Such problems can find substantial applications in various fields, such as geophysics, ocean acoustics and seismology. 

Let the scattered interface be denoted by $\Gamma:=\{(x',x_d)\in \R^d:x_d=f(x')\}$ $(d=2,3)$, where $f$ is assumed to be a Lipschitz continuous function with compact support. This means that $\Gamma$ is just a local perturbation of the plannar interface $\Gamma_0:=\{(x',x_d)\in \R^d:x_d=0\}$. The whole space $\R^d$ is then separated by $\Gamma$ into the upper half-space $D^+:=\{(x',x_d)\in \R^d:x_d>f(x')\}$ and the lower half-space $D^-:=\{(x',x_d)\in \R^d:x_d<f(x')\}$. Let $B_r(x)$ denote the open ball centered at $x\in\R^d$ with radius $r>0$. For balls centered at the origin, we abbreviate by $B_r$. 

Suppose that the whole space $\R^d$ is fullfilled with a piecewise isotropic and homogeneous elastic medium characterized by the Lamé constants $\lambda,\mu$ and the mass density $\rho$ satisfying that $\mu>0,d\lambda+2\mu>0$ and $\rho=\rho_\pm>0$ in $D^\pm$ with $\rho_+\neq\rho_-$. Given an incident field $\mathbf{u}^{in}$, consider the time-harmonic elastic scattering by the locally rough interface $\Gamma$, which is governed by the Navier equation 
\begin{align}\label{1.1}
  \Delta^*\mathbf{u}_\pm+\rho_\pm\om^2\mathbf{u}_\pm=0~~~\text{in}~D^\pm, 
\end{align}
where $\Delta^*:=\mu\Delta+(\lambda+\mu)\grad\dive$, $\om>0$ is the angular frequency, $\mathbf{u}_+$ is the scattered field in $D^+$ and $\mathbf{u}_-$ is the transmitted field in $D^-$. 

In order to state the boundary conditions for $\mathbf{u}_\pm$ at the interface $\Gamma$, we introduce the following notations. In two dimensions, for $\mathbf{y}=(y_1,y_2)^\top\in\R^2$, define $\mathbf{y}^\perp:=(y_2,-y_1)^\top$. Further, we will make use of the differential operators 
$$
  \grad^\perp u:=(\pa_2u,-\pa_1u)^\top~~~~\text{and}~~~~\dive^\perp\mathbf{u}:=\pa_2u_1-\pa_1u_2. 
$$
Denote by $\boldsymbol{\nu}$ the unit normal vector pointing out of $D^+$ and let $\boldsymbol{\tau}=\boldsymbol{\nu}^\perp$ be the unit tangential vector if $d=2$. The generalizd stress vector on $\Gamma$ is then defined by 
\begin{align*}
  \mathbf{P}_{\tilde\mu,\tilde\lambda}\mathbf{u}:=
  \left\{
  \begin{array}{ll}
  	(\mu+\tilde{\mu})\pa_{\boldsymbol{\nu}}\mathbf{u}+\tilde{\lambda}\boldsymbol{\nu}\dive\mathbf{u}-\tilde{\mu}\boldsymbol{\tau}\dive^\perp\mathbf{u},~~~&{\rm if}~d=2,\\
  	(\mu+\tilde{\mu})\pa_{\boldsymbol{\nu}}\mathbf{u}+\tilde{\lambda}\boldsymbol{\nu}\dive\mathbf{u}+\tilde{\mu}\boldsymbol{\nu}\times\curl\mathbf{u},~~~&{\rm if}~d=3, 
  \end{array}
  \right.
\end{align*}
where $\tilde{\mu},\tilde{\lambda}$ are real numbers satisfying $\tilde{\mu}+\tilde{\lambda}=\mu+\lambda$. Note that $\mathbf{P}_{\tilde\mu,\tilde\lambda}\mathbf{u}$ is equal to the physical stress vector $\mathbf{T}\mathbf{u}$ for the choice $\tilde{\mu}=\mu$ and $\tilde{\lambda}=\lambda$. The transmission boundary conditions for $\mathbf{u}_\pm$ are then given by 
\begin{align}\label{1.2}
  \mathbf{u}_+-\mathbf{u}_-=-\mathbf{u}^{in}~~~~\text{and}~~~~\mathbf{P}_{\tilde\mu,\tilde\lambda}\mathbf{u}_+-a_0\mathbf{P}_{\tilde\mu,\tilde\lambda}\mathbf{u}_-=-\mathbf{P}_{\tilde\mu,\tilde\lambda}\mathbf{u}^{in}~~~\text{on}~\Gamma, 
\end{align}
where the transmission coefficient $a_0>0$. 

To assure the uniqueness of the problem \eqref{1.1}--\eqref{1.2}, it is necessary to impose radiation conditions on the scattered field $\mathbf{u}_+$ and the transmitted field $\mathbf{u}_-$ at infinity. To this end, we note that the solutions $\mathbf{u}_\pm$ to the Navier equation admit the Helmholtz decomposition 
\begin{align*}
  \mathbf{u}_\pm=
  \left\{
  \begin{array}{ll}
  	\displaystyle-k_{p,\pm}^{-2}\grad\phi_{p,\pm}-k_{s,\pm}^{-2}\grad^\perp\phi_{s,\pm},~~~&{\rm if}~d=2,\\ [1mm]
  	\displaystyle-k_{p,\pm}^{-2}\grad\phi_{p,\pm}+k_{s,\pm}^{-2}\curl\boldsymbol{\phi}_{s,\pm},~~~&{\rm if}~d=3, 
  \end{array}
  \right.
\end{align*}
where the wave numbers $k_{p,\pm}$ and $k_{s,\pm}$ satisfy 
$$
  k_{p,\pm}^2=\frac{\rho_\pm\om^2}{2\mu+\lambda},~~~k_{s,\pm}^2=\frac{\rho_\pm\om^2}{\mu}, 
$$
and the compressional part $\phi_{p,\pm}:=\dive \mathbf{u}_\pm$, the shear part $\phi_{s,\pm}:=\dive^\perp\mathbf{u}_\pm, \boldsymbol{\phi}_{s,\pm}:=\curl\mathbf{u}_\pm$ satisfy 
$$
  \Delta\phi_{a,\pm}+k_{a,\pm}^2\phi_{a,\pm}=0,~a=p,s,~\text{and}~\curl\curl\boldsymbol{\phi}_{s,\pm}-k_{s,\pm}^2\boldsymbol{\phi}_{s,\pm}=0. 
$$
In view of the above equations, we require the following Kupradze radiation conditions on the problem \eqref{1.1}--\eqref{1.2}, 
\begin{subequations}\label{1.3}
	\begin{align}\label{1.3a}
	&d=2:~\pa_r\phi_{a,\pm}-ik_{a,\pm}\phi_{a,\pm}=o(r^{-1/2}),~a=p,s,~ \\
	\label{1.3b}
	&d=3:~\pa_r\phi_{p,\pm}-ik_{p,\pm}\phi_{p,\pm}=o(r^{-1}),~\curl\boldsymbol{\phi}_{s,\pm}\times\mathbf{\hat x}-ik_{s,\pm}\boldsymbol{\phi}_{s,\pm}=o(r^{-1}), ~
	\end{align}
\end{subequations}
uniformly in $\mathbf{\hat x}:=\mathbf{x}/r\in \Sp^{d-1}_\pm$ as $r:=|\mathbf{x}|\rightarrow\infty$, where
$\Sp^{d-1}_\pm:=\Sp^{d-1}\cap\R^d_\pm$ and $\R^d_\pm:=\{(x',x_d)\in \R^d:x_d\gtrless0\}$. Clearly, \eqref{1.3} is a combination of the classical Sommerfeld and Silver-M\"{u}ller radiation condition in the acoustic and electromagnetic scattering (c.f. \cite{FD06,DR13}). 

The elastic scattering problems involving rough surfaces have attracted considerable attention over the last twenty years. We refer to \cite{TA99a,TA99,MIJ11,JG10,GXY18} for periodic structures and locally rough surfaces, and \cite{TA01,TA02,JG12} for general global rough surfaces. However, the majority of the preceding work were only concerning the half-plane problem, while significantly less work existed for the whole space problem, which requires a substantially more elaborate analysis due to the presence of the transmitted wave and the complexity of the stress vector. Therefore, establishing a well-posedness theory for the elastic scattering by general rough interfaces continues to be an open problem. The primary difficulties lie in proving the uniqueness of the solution, a cornerstone of the scattering theory. For bounded scatterers, the Kupradze radiation condition was shown to be equivalent to an radiation condition about the surface integral between scattered solutions, which is easy to be applied for the uniqueness of solutions. Nevertheless, for locally rough surfaces, such equivalence has not been discovered yet, and thus it is hard to utilize the half-plane Kupradze radiation condition when considering the uniqueness of solution. Consequently, the uniqueness results for locally rough surfaces remain distinctly blank till now. 

Inspried by the Helmholtz decomposition, in this paper we discover a fundamental identity for the stress vector, which reveals an intrinsic relationship among the genearlized stress vector, the Lamé constants and certain tangential differential operators. This identity leads to two key limits for surface integrals concerning scattered solutions, which are instrumental in deducing both the uniqueness of solution and its Green's representation. Further, they in some extents indicate the equivalence between two aforementioned radiation conditions in the locally rough surface case. Utilizing these limits, and the uniqueness results in the acoustic and electromagnetic scattering, we prove the first uniqueness theorem of problem \eqref{1.1}--\eqref{1.3} for all frequencies. Furthermore, this framework enables the systematic extension of substantial uniqueness results from bounded elastic scatterers to locally rough surfaces.. 

Beyong its impressive application to rough surfaces, our utilization for the stress vector identity presents a fundamental contribution to the elastic scattering. Since in the definition of the genearlized stress vector, the dependence on the constants $\tilde\lambda$ and $\tilde\mu$ is rather vague, considerable existing works only hold under some constraints on the value of $\tilde\lambda$ and $\tilde\mu$. In contrast, this identity illustrate explicitly the relationship between the genearlized stress vector and $\tilde\lambda, \tilde\mu$. It appears that the term involving $\tilde\lambda, \tilde\mu$ corresponds to some tangential differential operators, and we show that they would vanish when considering the surface integral. This profound insight allows for the removal of these constraints in a wide range of previous results, significantly enhancing their generality and applicability. Moreover, utilizing this identity, it is possible to completely decouple the elastic waves to compressional and shear waves, thus simplifying the study of Lamé system to the investigation of two vectorial Helmholtz systems. 

The second part of this paper is about the two-layered Green's tensor for problem \eqref{1.1}--\eqref{1.3} with $a_0=1$. The two-layered Green's function in the acoustic scattering, which plays a pivotal role in both direct and inverse scattering in two-layered media, has been rigorously studied in \cite{HED05,NB84,LJ24,CP17}. Nevertheless, since the Navier equation has a more complex form than the Helmholtz equation and there is little work concerning the direct problem, to the best of the author's knowledge, there is no rigorous mathematical study about the two-layered elastic Green's tensor till now. In this paper, employing the aforementioned identity about the stress vector, we derive the existence and uniqueness of the corresponding two-layered Green's tensor. Using the Fourier transform and its inverse, we deduce the explicit expression of the Green's tensor in both two and three dimensions by solving several complicated systems of linear equations. Subsequently, following the arguments in \cite{HED05,NB84,LJ24,CP17}, utilizing the steepest descent method, we obtain the asymptotic behaviors of the corresponding expressions in two dimensions, while in three dimensions the analysis is reduced to the two-dimensional case after changing the Fourier transform to the Hankel transform and the asymptotic behaviors follows, which in turn implies the existence and uniqueness of the Green's tensor in both two and three dimensions. 

Finally, similarily as in \cite{GGT18,GXY18}, applying the properties of the two-layered Green's tensor, we propose an equivalent variational formulation to problem \eqref{1.1}--\eqref{1.3} with the Dirichlet-to-Neumann map derived from the integral representation of the scattered field. The variational formulation is then shown to be of Fredholm type, and thus the existence of solutions follows from the uniqueness, which achieves the first well-posedness result for elastic scattering by penetrable rough surfaces at all frequencies. 

The paper is organized as follows. In section \ref{sec2}, we first prove an identity for the stress vector, which implies two significant limits about the surface integral. Applying these limits, we obtain the uniqueness of solutions to problem \eqref{1.1}--\eqref{1.3} for all frequencies. In section \ref{sec3}, we derive the existence and uniqueness of the two-layered Green's tensor for problem \eqref{1.1}--\eqref{1.3} with $a_0=1$ by computing its explicit form and verifying that it satisfies the radiation condition \eqref{1.3}. In section \ref{sec4}, we prove the existence of solutions via the variational method and the boundary integral equaiton.

\section{Uniqueness of solution}\label{sec2}
\setcounter{equation}{0}
In this section, we shall derive the uniqueness result for problem \eqref{1.1}--\eqref{1.3}, which is based on an identity for the stress vector and two significant limits of the surface integral. 

\begin{lemma}\label{lem2.1}
	The following identities hold 
	\begin{align*}
	\mathbf{P}_{\tilde\mu,\tilde\lambda}\mathbf{u}=
	\left\{
	\begin{array}{ll}
		(\mu+\tilde{\mu})(\pa_{\boldsymbol{\tau}}\mathbf{u})^\perp+(2\mu+\lambda)\boldsymbol{\nu}\dive\mathbf{u}+\mu\boldsymbol{\tau}\dive^\perp\mathbf{u},~~~&{\rm if}~d=2,\\
		(\mu+\tilde{\mu})\mathcal{M}_{\boldsymbol{\nu}}(\mathbf{u})+(2\mu+\lambda)\boldsymbol{\nu}\dive\mathbf{u}-\mu\boldsymbol{\nu}\times\curl\mathbf{u},~~~&{\rm if}~d=3, 
	\end{array}
	\right.
	\end{align*}
	where $\mathcal{M}_{\boldsymbol{\nu}}$ is the differential operator defined by 
	\begin{align*}
	  \mathcal{M}_{\boldsymbol{\nu}}(\mathbf{u}):=
	  \left(
	  \begin{array}{l}
	  	\nu_2\pa_1u_2-\nu_1\pa_2u_2+\nu_3\pa_1u_3-\nu_1\pa_3u_3 \\
	  	\nu_1\pa_2u_1-\nu_2\pa_1u_1+\nu_3\pa_2u_3-\nu_2\pa_3u_3 \\
	  	\nu_1\pa_3u_1-\nu_3\pa_1u_1+\nu_2\pa_3u_2-\nu_3\pa_2u_2
	  \end{array}
	  \right)
	\end{align*}
	for $\mathbf{u}=(u_1,u_2,u_3)^\top$ and $\boldsymbol{\nu}=(\nu_1,\nu_2,\nu_3)^\top$. 
\end{lemma}
\begin{proof}
	From the definition of the generalized stress vector $	\mathbf{P}_{\tilde\mu,\tilde\lambda}\mathbf{u}$, it suffices to show that 
	\begin{align*}
	  \pa_{\boldsymbol{\nu}}\mathbf{u}=
	  \left\{
	  \begin{array}{ll}
	  	(\pa_{\boldsymbol{\tau}}\mathbf{u})^\perp+\boldsymbol{\nu}\dive\mathbf{u}+\boldsymbol{\tau}\dive^\perp\mathbf{u},~~~&{\rm if}~d=2,\\
	  	\mathcal{M}_{\boldsymbol{\nu}}(\mathbf{u})+\boldsymbol{\nu}\dive\mathbf{u}-\boldsymbol{\nu}\times\curl\mathbf{u},~~~&{\rm if}~d=3. 
	  \end{array}
	  \right.
	\end{align*}
	Direct calculation yields that 
	\begin{align*}
	  \begin{aligned}
	  &\;(\pa_{\boldsymbol{\tau}}\mathbf{u})^\perp+\boldsymbol{\nu}\dive\mathbf{u}+\boldsymbol{\tau}\dive^\perp\mathbf{u} \\
	   =&\left(
	  \begin{array}{l}
	  	\nu_2\pa_1u_2-\nu_1\pa_2u_2+\nu_1(\pa_1u_1+\pa_2u_2)+\nu_2(\pa_2u_1-\pa_1u_2) \\
	  	-\nu_2\pa_1u_1+\nu_1\pa_2u_1+\nu_2(\pa_1u_1+\pa_2u_2)-\nu_1(\pa_2u_1-\pa_1u_2)
	  \end{array}
	  \right) \\
	  =&\left(
	  \begin{array}{l}
	  	\nu_1\pa_1u_1+\nu_2\pa_2u_1 \\
	  	\nu_1\pa_1u_2+\nu_2\pa_2u_2
	  \end{array}
	  \right)
	  =\pa_{\boldsymbol{\nu}}\mathbf{u}, 
	  \end{aligned}
	\end{align*}
	and 
	\begin{align*}
	\begin{aligned}
	  &\;\mathcal{M}_{\boldsymbol{\nu}}(\mathbf{u})+\boldsymbol{\nu}\dive\mathbf{u}-\boldsymbol{\nu}\times\curl\mathbf{u} \\
	  =&\left(
	  \begin{array}{l}
	  	\nu_2\pa_1u_2-\nu_1\pa_2u_2+\nu_3\pa_1u_3-\nu_1\pa_3u_3 \\
	  	\nu_1\pa_2u_1-\nu_2\pa_1u_1+\nu_3\pa_2u_3-\nu_2\pa_3u_3 \\
	  	\nu_1\pa_3u_1-\nu_3\pa_1u_1+\nu_2\pa_3u_2-\nu_3\pa_2u_2
	  \end{array}
	  \right)+
	  \left(
	  \begin{array}{l}
	  	\nu_1(\pa_1u_1+\pa_2u_2+\pa_3u_3) \\
	  	\nu_2(\pa_1u_1+\pa_2u_2+\pa_3u_3) \\ 
	  	\nu_3(\pa_1u_1+\pa_2u_2+\pa_3u_3) 
	  \end{array}
	  \right) \\
	   &\;-\left(
	  \begin{array}{l}
	  	\nu_2(\pa_1u_2-\pa_2u_1)-\nu_3(\pa_3u_1-\pa_1u_3) \\
	  	\nu_3(\pa_2u_3-\pa_3u_2)-\nu_1(\pa_1u_2-\pa_2u_1) \\
	  	\nu_1(\pa_3u_1-\pa_1u_3)-\nu_2(\pa_2u_3-\pa_3u_2) 
	  \end{array}
	  \right) \\
	  =&\left(
	  \begin{array}{l}
	  	\nu_1\pa_1u_1+\nu_2\pa_2u_1+\nu_3\pa_3u_1 \\
	  	\nu_1\pa_1u_2+\nu_2\pa_2u_2+\nu_3\pa_3u_2  \\
	  	\nu_1\pa_1u_3+\nu_2\pa_2u_3+\nu_3\pa_3u_3 
	  \end{array}
	  \right) 
	  =\pa_{\boldsymbol{\nu}}\mathbf{u}, 
	  \end{aligned}
	\end{align*}
	which are the desired identities. 
\end{proof}
\begin{remark}\label{rem2.2}
	We note that the three dimensional case has been presented in \cite[Chapter V]{VD79}, while the two dimensional case is first proposed here. Further, it is known that 
	\begin{align*}
	  \boldsymbol{\nu}\times\grad\mathbf{u}=
	  \left(
	  \begin{array}{ccc}
	  	\nu_2\pa_3u_1-\nu_3\pa_2u_1 & \nu_3\pa_1u_1-\nu_1\pa_3u_1 & \nu_1\pa_2u_1-\nu_2\pa_1u_1 \\
	  	\nu_2\pa_3u_2-\nu_3\pa_2u_2 & \nu_3\pa_1u_2-\nu_1\pa_3u_2 & \nu_1\pa_2u_2-\nu_2\pa_1u_2 \\
	  	\nu_2\pa_3u_3-\nu_3\pa_2u_3 & \nu_3\pa_1u_3-\nu_1\pa_3u_3 & \nu_1\pa_2u_3-\nu_2\pa_1u_3 
	  \end{array}
	  \right), 
	\end{align*}
	which implies that $\mathcal{M}_{\boldsymbol{\nu}}(\mathbf{u})$ depends fully on $\boldsymbol{\nu}\times\grad\mathbf{u}$. Therefore, it is deduced that for arbitrarily fixed $\tilde{\mu}$ and $\tilde{\lambda}$, if $\mathbf{u}=\mathbf{P}_{\tilde{\mu},\tilde{\lambda}}\mathbf{u}=0$ on $\Gamma$, then $\pa_{\boldsymbol{\nu}}\mathbf{u}=0$ on $\Gamma$ and hence $\grad\mathbf{u}=0$ on $\Gamma$. That is, problem \eqref{1.1}--\eqref{1.3} with $a_0=1$ is independent of the choice of $\tilde{\mu}$ and $\tilde{\lambda}$. 
\end{remark}
\begin{corollary}\label{cor2.3}
	Suppose $\mathbf{u}^{in}=0$ and $a_0=1$, then we have 
	\begin{align*}
	  &d=2:\left\{
	  \begin{array}{l}
	  	\phi_{p,+}=\phi_{p,-},~\phi_{s,+}=\phi_{s,-}, \\ [1mm]
	  	-k_{p,+}^{-2}\pa_{\boldsymbol{\nu}}\phi_{p,+}+k_{s,+}^{-2}\pa_{\boldsymbol{\tau}}\phi_{s,+}=-k_{p,-}^{-2}\pa_{\boldsymbol{\nu}}\phi_{p,-}+k_{s,-}^{-2}\pa_{\boldsymbol{\tau}}\phi_{s,-}, \\ [1mm]
	   k_{p,+}^{-2}\pa_{\boldsymbol{\tau}}\phi_{p,+}+k_{s,+}^{-2}\pa_{\boldsymbol{\nu}}\phi_{s,+}=k_{p,-}^{-2}\pa_{\boldsymbol{\tau}}\phi_{p,-}+k_{s,-}^{-2}\pa_{\boldsymbol{\nu}}\phi_{s,-},
	  \end{array}
	  \right. \\
	&d=3:\left\{
	\begin{array}{l}
		\phi_{p,+}=\phi_{p,-},~\boldsymbol{\nu}\times\boldsymbol{\phi}_{s,+}=\boldsymbol{\nu}\times\boldsymbol{\phi}_{s,-}, \\ [1mm]
		-k_{p,+}^{-2}\pa_{\boldsymbol{\nu}}\phi_{p,+}+k_{s,+}^{-2}\boldsymbol{\nu}\cdot\curl\boldsymbol{\phi}_{s,+}=-k_{p,-}^{-2}\pa_{\boldsymbol{\nu}}\phi_{p,-}+k_{s,-}^{-2}\boldsymbol{\nu}\cdot\curl\boldsymbol{\phi}_{s,-}, \\ [1mm]
		-k_{p,+}^{-2}\boldsymbol{\nu}\times\grad\phi_{p,+}+k_{s,+}^{-2}\boldsymbol{\nu}\times\curl\boldsymbol{\phi}_{s,+} \\ [1mm]
		=-k_{p,-}^{-2}\boldsymbol{\nu}\times\grad\phi_{p,-}+k_{s,-}^{-2}\boldsymbol{\nu}\times\curl\boldsymbol{\phi}_{s,-}, 
	\end{array}
	\right.
	\end{align*}
	on the boundary $\Gamma$. 
\end{corollary}
\begin{proof}
	We only prove the two dimensional case, while the three dimensional case follows from a completely similar analog combining Lemma \ref{lem2.1} and Remark \ref{rem2.2}. 
	
	Since $\mathbf{u}^{in}=0$, we have $\mathbf{u}_+-\mathbf{u}_-=0$ on $\Gamma$, which further implies $\boldsymbol{\nu}\cdot(\mathbf{u}_+-\mathbf{u}_-)=\boldsymbol{\tau}\cdot(\mathbf{u}_+-\mathbf{u}_-)=0$ on $\Gamma$. By the Helmholtz decomposition it is seen that 
	\begin{align*}
	  \begin{aligned}
	  -k_{p,+}^{-2}\pa_{\boldsymbol{\nu}}\phi_{p,+}+k_{s,+}^{-2}\pa_{\boldsymbol{\tau}}\phi_{s,+}&=-k_{p,-}^{-2}\pa_{\boldsymbol{\nu}}\phi_{p,-}+k_{s,-}^{-2}\pa_{\boldsymbol{\tau}}\phi_{s,-}, \\
	  k_{p,+}^{-2}\pa_{\boldsymbol{\tau}}\phi_{p,+}+k_{s,+}^{-2}\pa_{\boldsymbol{\nu}}\phi_{s,+}&=k_{p,-}^{-2}\pa_{\boldsymbol{\tau}}\phi_{p,-}+k_{s,-}^{-2}\pa_{\boldsymbol{\nu}}\phi_{s,-}, 
	  \end{aligned}
	\end{align*}
	where we note that $\boldsymbol{\nu}\cdot\grad=\boldsymbol{\tau}\cdot\grad^\perp=\pa_{\boldsymbol{\nu}}$ and $\boldsymbol{\nu}\cdot\grad^\perp=-\boldsymbol{\tau}\cdot\grad=-\pa_{\boldsymbol{\tau}}$. Moreover, it follows from $\mathbf{u}_+-\mathbf{u}_-=0$ that $(\pa_{\boldsymbol{\tau}}\mathbf{u}_+)^\perp-(\pa_{\boldsymbol{\tau}}\mathbf{u}_-)^\perp=0$ on $\Gamma$, which indicates by Lemma \ref{lem2.1} that 
	$$
	  0=\mathbf{P}_{\tilde\mu,\tilde\lambda}\mathbf{u}_+-\mathbf{P}_{\tilde\mu,\tilde\lambda}\mathbf{u}_-=(2\mu+\lambda)\boldsymbol{\nu}(\dive\mathbf{u}_+-\dive\mathbf{u}_-)+\mu\boldsymbol{\tau}(\dive^\perp\mathbf{u}_+-\dive^\perp\mathbf{u}_-)
	$$
	on $\Gamma$. This clearly leads to $\phi_{p,+}=\phi_{p,-},~\phi_{s,+}=\phi_{s,-}$ on $\Gamma$, which ends the proof. 
\end{proof}
\begin{theorem}\label{thm2.4a}
	Assume $d=2$. Suppose $\mathbf{u}$ and $\mathbf{v}$ satisfy the Navier equaiton \eqref{1.1} and the radiation condition \eqref{1.3} in $D^+$ with $\phi_p,\varphi_p$ being the compressional part and $\phi_s,\varphi_s$ being the shear part, respectively. Then we have 
	\begin{align*}
	  &\lim\limits_{R\rightarrow\infty}\int_{\pa B_R^+}\left(\mathbf{P}_{\tilde{\mu},\tilde{\lambda}}\mathbf{u}\cdot\mathbf{v}-\mathbf{P}_{\tilde{\mu},\tilde{\lambda}}\mathbf{v}\cdot\mathbf{u}\right)ds=0, \\
	  &\lim\limits_{R\rightarrow\infty}\left(\I\int_{\pa B_R^+}\mathbf{P}_{\tilde{\mu},\tilde{\lambda}}\mathbf{u}\cdot\mathbf{\ov u}ds-(2\mu+\lambda)k_{p,+}^{-1}\int_{\pa B_R^+}|\phi_p|^2ds-\mu k_{s,+}^{-1}\int_{\pa B_R^+}|\phi_s|^2ds\right)=0, 
	\end{align*}
	where $\pa B_R^\pm:=\pa B_R\cap D^\pm$. 
\end{theorem}
\begin{proof}
	By Lemma \ref{lem2.1} and the Helmholtz decomposition, it is derived that 
	\begin{align}\label{2.1a} 
	  \begin{aligned} 
	  \mathbf{P}_{\tilde{\mu},\tilde{\lambda}}\mathbf{u}\cdot\mathbf{v}=\;&(\mu+\tilde{\mu})(\pa_{\boldsymbol{\tau}}\mathbf{u})^\perp\cdot\mathbf{v}+(2\mu+\lambda)\phi_p(-k_{p,+}^{-2}\pa_{\boldsymbol{\nu}}\varphi_p+k_{s,+}^{-2}\pa_{\boldsymbol{\tau}}\varphi_s) \\
	  &-\mu\phi_s(k_{p,+}^{-2}\pa_{\boldsymbol{\tau}}\varphi_p+k_{s,+}^{-2}\pa_{\boldsymbol{\nu}}\varphi_s), 
	  \end{aligned}
	\end{align}
	which implies that 
	\begin{align*}
	  \begin{aligned}
	  \mathbf{P}_{\tilde{\mu},\tilde{\lambda}}\mathbf{u}\cdot\mathbf{v}-\mathbf{P}_{\tilde{\mu},\tilde{\lambda}}\mathbf{v}\cdot\mathbf{u} 
	  =\;&(\mu+\tilde{\mu})\pa_{\boldsymbol{\tau}}(\mathbf{u}^\perp\cdot\mathbf{v})+\frac{(2\mu+\lambda)\mu}{\rho_+\om^2}\pa_{\boldsymbol{\tau}}(\phi_p\varphi_s-\phi_s\varphi_p) \\
	  &-(2\mu+\lambda)k_{p,+}^{-2}(\phi_p\pa_{\boldsymbol{\nu}}\varphi_p-\varphi_p\pa_{\boldsymbol{\nu}}\phi_p)-\mu k_{s,+}^{-2}(\phi_s\pa_{\boldsymbol{\nu}}\varphi_s-\varphi_s\pa_{\boldsymbol{\nu}}\phi_s). 
	  \end{aligned}
	\end{align*}
	Therefore, for $R>0$ sufficiently large, 
	\ben
	  \begin{aligned}
	  &\int_{\pa B_R^+}\left(\mathbf{P}_{\tilde{\mu},\tilde{\lambda}}\mathbf{u}\cdot\mathbf{v}-\mathbf{P}_{\tilde{\mu},\tilde{\lambda}}\mathbf{v}\cdot\mathbf{u}\right)ds \\
	  =\;&(\mu+\tilde{\mu})[(\mathbf{u}^\perp\cdot\mathbf{v})(R,0)-(\mathbf{u}^\perp\cdot\mathbf{v})(-R,0)]+\frac{(2\mu+\lambda)\mu}{\rho_+\om^2}[(\phi_p\varphi_s)(R,0)-(\phi_p\varphi_s)(-R,0)] \\
	  &-\frac{(2\mu+\lambda)\mu}{\rho_+\om^2}[(\phi_s\varphi_p)(R,0)-(\phi_s\varphi_p)(-R,0)]-(2\mu+\lambda)k_{p,+}^{-2}\int_{\pa B_R^+}(\phi_p\pa_{\boldsymbol{\nu}}\varphi_p-\varphi_p\pa_{\boldsymbol{\nu}}\phi_p)ds \\
	  &-\mu k_{s,+}^{-2}\int_{\pa B_R^+}(\phi_s\pa_{\boldsymbol{\nu}}\varphi_s-\varphi_s\pa_{\boldsymbol{\nu}}\phi_s)ds, 
	  \end{aligned}
	\enn
	which tends to zero as $R\rightarrow\infty$ by the radiation condition \eqref{1.3}. 
	
	Further, again by \eqref{2.1a} we see that 
	  \begin{align*}
	  	&\I(\mathbf{P}_{\tilde{\mu},\tilde{\lambda}}\mathbf{u}\cdot\mathbf{\ov u}) \\
	  =\;&\I(\pa_{\boldsymbol{\tau}}(\ov u_1u_2))+\frac{(2\mu+\lambda)\mu}{\rho_+\om^2}\I(\pa_{\boldsymbol{\tau}}(\phi_p\ov\phi_s))-(2\mu+\lambda)k_{p,+}^{-2}\I[\phi_p(\pa_{\boldsymbol{\nu}}\ov\phi_p+ik_{p,+}\ov\phi_p)] \\
	  &-\mu k_{s,+}^{-2}\I[\phi_s(\pa_{\boldsymbol{\nu}}\ov\phi_s+ik_{s,+}\ov\phi_s)]+(2\mu+\lambda)k_{p,+}^{-1}|\phi_p|^2+\mu k_{s,+}^{-1}|\phi_s|^2, 
	  \end{align*}
	which indicates the desired limit by the radiation condition \eqref{1.3}. 
\end{proof}
\begin{theorem}\label{thm2.5a}
	Assume $d=3$. Suppose $\mathbf{u}$ and $\mathbf{v}$ satisfy the Navier equaiton \eqref{1.1} and the radiation condition \eqref{1.3} in $D^+$ with $\phi_p,\varphi_p$ being the compressional part and $\boldsymbol{\phi}_s,\boldsymbol{\varphi}_s$ being the shear part, respectively. Then we have 
	\begin{align*}
	&\lim\limits_{R\rightarrow\infty}\int_{\pa B_R^+}\left(\mathbf{P}_{\tilde{\mu},\tilde{\lambda}}\mathbf{u}\cdot\mathbf{v}-\mathbf{P}_{\tilde{\mu},\tilde{\lambda}}\mathbf{v}\cdot\mathbf{u}\right)ds=0, \\
	&\lim\limits_{R\rightarrow\infty}\left(\I\int_{\pa B_R^+}\mathbf{P}_{\tilde{\mu},\tilde{\lambda}}\mathbf{u}\cdot\mathbf{\ov u}ds-(2\mu+\lambda)k_{p,+}^{-1}\int_{\pa B_R^+}|\phi_p|^2ds-\mu k_{s,+}^{-1}\int_{\pa B_R^+}|\boldsymbol{\phi}_s|^2ds\right)=0. 
	\end{align*}
\end{theorem}
\begin{proof}
	The proof is similar as Theorem \ref{thm2.4a} with some modifications. It follows from Lemma \ref{lem2.1} and the Helmholtz decomposition that 
	\begin{align*}
	  \begin{aligned}
	  	\mathbf{P}_{\tilde{\mu},\tilde{\lambda}}\mathbf{u}\cdot\mathbf{v}=\;&(\mu+\tilde{\mu})\mathcal{M}_{\boldsymbol{\nu}}(\mathbf{u})\cdot\mathbf{v}+(2\mu+\lambda)\phi_p(-k_{p,+}^{-2}\pa_{\boldsymbol{\nu}}\varphi_p+k_{s,+}^{-2}\boldsymbol{\nu}\cdot\curl\boldsymbol{\varphi}_s) \\
	  &-\mu\boldsymbol{\phi}_s\cdot(-k_{p,+}^{-2}\grad\varphi_p\times\boldsymbol{\nu}+k_{s,+}^{-2}\curl\boldsymbol{\varphi}_s\times\boldsymbol{\nu}). 
	  \end{aligned}
	\end{align*}
	Denote by $\Div$ the surface divergence. By direct calculation we derive that 
	$$
	  \mathcal{M}_{\boldsymbol{\nu}}(\mathbf{u})\cdot\mathbf{v}-\mathcal{M}_{\boldsymbol{\nu}}(\mathbf{v})\cdot\mathbf{u}=\boldsymbol{\nu}\cdot\curl(\mathbf{u}\times\mathbf{v})=-\Div(\boldsymbol{\nu}\times(\mathbf{u}\times\mathbf{v})), 
	$$
	and thus 
	$$
	  \I(\mathcal{M}_{\boldsymbol{\nu}}(\mathbf{u})\cdot\mathbf{\ov u})=\frac{1}{2}\I(\boldsymbol{\nu}\cdot\curl(\mathbf{u}\times\mathbf{\ov u}))=-\frac{1}{2}\I(\Div(\boldsymbol{\nu}\times(\mathbf{u}\times\mathbf{\ov u}))). 
	$$
	Therefore, it is deduced that 
	\ben
	  \begin{aligned}
	  	&\;\mathbf{P}_{\tilde{\mu},\tilde{\lambda}}\mathbf{u}\cdot\mathbf{v}-\mathbf{P}_{\tilde{\mu},\tilde{\lambda}}\mathbf{v}\cdot\mathbf{u} \\
	  =\;&-(\mu+\tilde{\mu})\Div(\boldsymbol{\nu}\times(\mathbf{u}\times\mathbf{v}))-\frac{(2\mu+\lambda)\mu}{\rho_+\om^2}\Div(\boldsymbol{\nu}\times(\phi_p\boldsymbol{\varphi}_s-\varphi_p\boldsymbol{\phi}_s)) \\
	  &-(2\mu+\lambda)k_{p,+}^{-2}(\phi_p\pa_{\boldsymbol{\nu}}\varphi_p-\varphi_p\pa_{\boldsymbol{\nu}}\phi_p)-\mu k_{s,+}^{-2}(\boldsymbol{\phi}_s\cdot\curl\boldsymbol{\varphi}_s\times\boldsymbol{\nu}-\boldsymbol{\varphi}_s\cdot\curl\boldsymbol{\phi}_s\times\boldsymbol{\nu}), 
	  \end{aligned}
	\enn
	and 
	\ben
	  \begin{aligned}
	  	&\;\I(\mathbf{P}_{\tilde{\mu},\tilde{\lambda}}\mathbf{u}\cdot\mathbf{\ov u}) \\
	  =\;&-\frac{(\mu+\tilde{\mu})}{2}\I(\Div(\boldsymbol{\nu}\times(\mathbf{u}\times\mathbf{\ov u})))-\frac{(2\mu+\lambda)\mu}{\rho_+\om^2}\I(\Div(\boldsymbol{\nu}\times(\phi_p\boldsymbol{\ov\phi}_s))) \\
	  &-(2\mu+\lambda)k_{p,+}^{-2}\I(\phi_p(\pa_{\boldsymbol{\nu}}\ov\phi_p+ik_{p,+}\ov\phi_p))-\mu k_{s,+}^{-2}\I(\boldsymbol{\phi}_s\cdot(\curl\boldsymbol{\ov\phi}_s\times\boldsymbol{\nu}+ik_{s,+}\boldsymbol{\ov\phi}_s)) \\
	  &+(2\mu+\lambda)k_{p,+}^{-1}|\phi_p|^2+\mu k_{s,+}^{-1}|\boldsymbol{\phi}_s|^2, 
	  \end{aligned}
	\enn
	which finishes the proof by the Stokes formula and the radiation condition \eqref{1.3}. 
\end{proof}

Clearly, Theorems \ref{thm2.4a} and \ref{thm2.5a} still hold if we replace $D^+$ by $D^-$. Utilizing these theorems, we can easily get the uniqueness result. 

\begin{theorem}\label{thm2.4}
	The problem \eqref{1.1}--\eqref{1.3} has at most one solution. 
\end{theorem}
\begin{proof}
	Suppose $\mathbf{u}^{in}=0$, we aim to show $\mathbf{u}_\pm=0$ in $D^\pm$. Integration by parts over $B_R\cap D^\pm$ with $R>0$ sufficiently large yields that 
	$$
	  \I\int_{\pa B_R^+}\mathbf{P}_{\tilde{\mu},\tilde{\lambda}}\mathbf{u_+}\cdot\mathbf{\ov u_+}ds+a_0\I\int_{\pa B_R^-}\mathbf{P}_{\tilde{\mu},\tilde{\lambda}}\mathbf{u_-}\cdot\mathbf{\ov u_-}ds=0. 
	$$
	Combining Theorems \ref{thm2.4a} and \ref{thm2.5a}, and the uniqueness result \cite[Proposition 2.1]{PC98}, we see that $\phi_{p,\pm}=\phi_{s,\pm}=0$ and $\boldsymbol{\phi}_{s,\pm}=0$ in $D^\pm$. Therefore, $\mathbf{u}_\pm=0$ in $D^\pm$. 
\end{proof}
\begin{remark}\label{rem2.5}
	Evidently, we can extends Theorem {\rm \ref{thm2.4}} to the case that there is a obstacle or an inhomogeneous media embedded in $D^\pm$, and further, $\lambda$ and $\mu$ are different in $D^\pm$. 
	Moreover, if $\Gamma$ is an impenetrable locally rough surface, the corresponding uniqueness results with respect to kinds of boundary conditions can be deduced analogously. 
\end{remark}
\begin{remark}\label{rem2.6}
	Lemma {\rm \ref{lem2.1}} and Theorems {\rm \ref{thm2.4a}} and {\rm \ref{thm2.5a}} actually bring new insights to the foundation of elastic scattering. Considerable existing results in the elastic scattering hold only under some assumptions on the value of $\tilde\lambda$ and $\tilde\mu$, while it can be seen that Theorems {\rm \ref{thm2.4a}} and {\rm \ref{thm2.5a}} are independent of the value of $\tilde\lambda$ and $\tilde\mu$. The reason is that Lemma {\rm \ref{lem2.1}} provides an explicit relation between the generalized stress vector and $\tilde\lambda, \tilde\mu$. The terms relating to $\tilde\lambda, \tilde\mu$ in $\mathbf{P}_{\tilde{\mu},\tilde{\lambda}}\mathbf{u}$ are only some tangential differential operators $(\pa_{\boldsymbol{\tau}}\mathbf{u})^\perp$ and $\mathcal{M}_{\boldsymbol{\nu}}(\mathbf{u})$, and we discover that they would vanish when considering the surface integral as in the proof of Theorems {\rm \ref{thm2.4a}} and {\rm \ref{thm2.5a}}. Therefore, proceeding similarily we can eliminate the constraints on $\tilde\lambda$ and $\tilde\mu$ of substantial previous work. Further, in {\rm \cite{HJ17}} the elastic waves are completely decoupled to the compressional and shear waves under some geometric conditions, which can be easily deduced without any assumptions utilizing Lemma {\rm \ref{lem2.1}}. 
\end{remark}

\section{Two-layered Green's tensor}\label{sec3}
\setcounter{equation}{0}
In the follwing, we focus on the case that the transmission coefficient $a_0=1$. In this section, we shall prove the existence and uniqueness of the two-layered Green's tensor for problem \eqref{1.1}--\eqref{1.3} with $a_0=1$ by computing its explicit form and verify that it satisfies the radiation condition \eqref{1.3} using the steepest descent method.

\subsection{Two-dimensional case}\label{sec3.1}
\setcounter{equation}{0}
Given any source point $\mathbf{y}\in\R^2_+\cup\R^2_-$, the two-layered Green's tensor $\mathbf{G}(\mathbf{x},\mathbf{y})$ in consideration is the solution to the following scattering problem: 
\be\label{3.1}
  \left\{
  \begin{array}{ll}
  	\Delta^*_{\mathbf{x}}\mathbf{G}(\mathbf{x},\mathbf{y})+\rho_\pm\om^2\mathbf{G}(\mathbf{x},\mathbf{y})=-\delta(\mathbf{x},\mathbf{y})\mathbf{I},~~~&{\rm in}~\R^2_\pm,~\mathbf{x}\neq\mathbf{y}, \\
  	\left[\mathbf{G}(\mathbf{x},\mathbf{y})\right]=0,~\left[\mathbf{P}_{\tilde{\mu},\tilde{\lambda}}^{(\mathbf{x})}(\mathbf{G}(\mathbf{x},\mathbf{y}))\right]=0~~~&{\rm on}~\Gamma_0, \\
  	\lim\limits_{r\rightarrow\infty}\sqrt{r}\left(\pa_r\mathbf{G}_a(\mathbf{x},\mathbf{y})-ik_{a,\pm}\mathbf{G}_a(\mathbf{x},\mathbf{y})\right)=0,~a=p,s,~~~&\mathbf{\hat x}\in\Sp^1_\pm, 
  \end{array}
  \right.
\en
where $\delta(\mathbf{x},\mathbf{y})$ denotes the Dirac delta distribution, $\mathbf{I}$ is the $2\times2$ identity matrix, $[\cdot]$ is the jump across the interface $\Gamma_0$, and $\mathbf{G}_p:=\dive_{\mathbf{x}}\mathbf{G}, \mathbf{G}_s:=\dive^\perp_{\mathbf{x}}\mathbf{G}$. Note that $\Delta^*_{\mathbf{x}}, \mathbf{P}_{\tilde{\mu},\tilde{\lambda}}^{(\mathbf{x})},\dive_{\mathbf{x}}$ and $\dive_{\mathbf{x}}^\perp$ are acting on $\mathbf{G}_j, j=1,2$, where $\mathbf{G}=(\mathbf{G}_1,\mathbf{G}_2)$ with $\mathbf{G}_j$ being the column vectors. Set $G_{p,j}:=\dive_{\mathbf{x}}\mathbf{G}_j$ and $G_{s,j}:=\dive_{\mathbf{x}}^\perp\mathbf{G}_j$ with $j=1,2$. 

The solution $\mathbf{G}$ to problem \eqref{3.1} is clearly unique due to Theorem \ref{thm2.4}. To show the existence, we compute the explicit form of $\mathbf{G}$ and verify that it satisfies \eqref{3.1}. First, it should be pointed out that $\mathbf{G}$ defined by \eqref{3.1} is independent of the choice of $\tilde{\mu}$ and $\tilde{\lambda}$, since by the Helmholtz decomposition and Corollary \ref{cor2.3}, we see that \eqref{3.1} is equivalent to the problem: for $j=1,2$, 
\be\label{3.2}
  \left\{
  \begin{array}{ll}
  	\displaystyle\Delta_\mathbf{x}G_{p,j}(\mathbf{x},\mathbf{y})+k_{p,\pm}^2G_{p,j}(\mathbf{x},\mathbf{y})=-(2\mu+\lambda)^{-1}\dive_{\mathbf{x}}(\delta(\mathbf{x},\mathbf{y})\mathbf{e}_j),~~~&{\rm in}~\R^2_\pm,\\ 
  	\displaystyle\Delta_\mathbf{x}G_{s,j}(\mathbf{x},\mathbf{y})+k_{s,\pm}^2G_{s,j}(\mathbf{x},\mathbf{y})=-\mu^{-1}\dive_{\mathbf{x}}^\perp(\delta(\mathbf{x},\mathbf{y})\mathbf{e}_j),~~~&{\rm in}~\R^2_\pm,\\
  	\left[G_{p,j}(\mathbf{x},\mathbf{y})\right]=\left[G_{s,j}(\mathbf{x},\mathbf{y})\right]=0,~~~&{\rm on}~\Gamma_0, \\
  	\displaystyle\left[k_p^{-2}\pa_{\boldsymbol{\nu}(\mathbf{x})}G_{p,j}(\mathbf{x},\mathbf{y})-k_s^{-2}\pa_{\boldsymbol{\tau}(\mathbf{x})}G_{s,j}(\mathbf{x},\mathbf{y})\right]=0,~~~&{\rm on}~\Gamma_0, \\
  	\displaystyle\left[k_p^{-2}\pa_{\boldsymbol{\tau}(\mathbf{x})}G_{p,j}(\mathbf{x},\mathbf{y})+k_s^{-2}\pa_{\boldsymbol{\nu}(\mathbf{x})}G_{s,j}(\mathbf{x},\mathbf{y})\right]=0,~~~&{\rm on}~\Gamma_0, \\
  	\lim\limits_{r\rightarrow\infty}\sqrt{r}\left(\pa_rG_{a,j}(\mathbf{x},\mathbf{y})-ik_{a,\pm}G_{a,j}(\mathbf{x},\mathbf{y})\right)=0,~a=p,s,~~~&\mathbf{\hat x}\in\Sp^1_\pm, 
  \end{array}
  \right.
\en
where $\mathbf{e}_j$ is the $j$-th standard orthonormal basis of $\R^2$ and $k_a:=k_{a,\pm}$ in $\R^2_\pm$. 

Now we derive the explicit formula for $G_{a,j}$ with $a=p,s$ and $j=1,2$. To this end, we introduce some auxiliary functions $\wid{G}_{a,j}$: 
\be\label{3.3}
\left\{
\begin{array}{ll}
	\displaystyle\Delta_\mathbf{x}\wid G_{p,j}(\mathbf{x},\mathbf{y})+k_{p,\pm}^2\wid G_{p,j}(\mathbf{x},\mathbf{y})=-(2\mu+\lambda)^{-1}\dive_{\mathbf{x}}(\delta(\mathbf{x},\mathbf{y})\mathbf{e}_j),~~~&{\rm in}~\R^2_\pm,\\ 
	\displaystyle\Delta_\mathbf{x}\wid G_{s,j}(\mathbf{x},\mathbf{y})+k_{s,\pm}^2\wid G_{s,j}(\mathbf{x},\mathbf{y})=-\mu^{-1}\dive_{\mathbf{x}}^\perp(\delta(\mathbf{x},\mathbf{y})\mathbf{e}_j),~~~&{\rm in}~\R^2_\pm,\\
	\left[\wid G_{p,j}(\mathbf{x},\mathbf{y})\right]=\left[\wid G_{s,j}(\mathbf{x},\mathbf{y})\right]=0,~~~&{\rm on}~\Gamma_0, \\
	\displaystyle\left[k_p^{-2}\pa_{\boldsymbol{\nu}(\mathbf{x})}\wid G_{p,j}(\mathbf{x},\mathbf{y})\right]=\displaystyle\left[k_s^{-2}\pa_{\boldsymbol{\nu}(\mathbf{x})}\wid G_{s,j}(\mathbf{x},\mathbf{y})\right]=0,~~~&{\rm on}~\Gamma_0, \\ 
	\lim\limits_{r\rightarrow\infty}\sqrt{r}\left(\pa_r\wid G_{a,j}(\mathbf{x},\mathbf{y})-ik_{a,\pm}\wid G_{a,j}(\mathbf{x},\mathbf{y})\right)=0,~a=p,s,~~~&\mathbf{\hat x}\in\Sp^1_\pm. 
\end{array}
\right.
\en
These functions are basically the derivatives of the two-layered Green's function in the acoustic scattering, while the latter has been rigorously studied in \cite{HED05,CP17,LJ24}. Therefore, here we only refer the explicit form of $\wid G_{a,j}$, for more properties, such as the asymptotic behavior, see details in  \cite{HED05,CP17,LJ24}. 

For $a=p,s$, define $\beta_{a,\pm}:=\sqrt{\xi^2-k_{a,\pm}^2}=\sqrt{\xi-k_{a,\pm}}\sqrt{\xi+k_{a,\pm}}$, where the relevant branches are $-3\pi/2\leq\arg(\xi-k_{a,\pm})<\pi/2$ for $\sqrt{\xi-k_{a,\pm}}$ and $-\pi/2\leq\arg(\xi+k_{a,\pm})<3\pi/2$ for $\sqrt{\xi+k_{a,\pm}}$. In particular, for real $\xi$ we have $\beta_{a,\pm}(\xi)=\sqrt{\xi^2-k_{a,\pm}^2}$ if $|\xi|\geq k_{a,\pm}$ and $\beta_{a,\pm}(\xi)=-i\sqrt{k_{a,\pm}^2-\xi^2}$ if $|\xi|<k_{a,\pm}$. The branch cuts and domain of definition of $\beta_{a,\pm}$ are depicted in Figure \ref{fig3.1}. Further, set  
$$
R_a(\beta_{a,+},\beta_{a,-}):=\frac{k_{a,+}^{-2}\beta_{a,+}-k_{a,-}^{-2}\beta_{a,-}}{k_{a,+}^{-2}\beta_{a,+}+k_{a,-}^{-2}\beta_{a,-}},~~~
T_a(\beta_{a,+},\beta_{a,-}):=\frac{2k_{a,+}^{-2}\beta_{a,+}}{k_{a,+}^{-2}\beta_{a,+}+k_{a,-}^{-2}\beta_{a,-}}. 
$$
Denote by $\Phi_{k_{a,\pm}}(\mathbf{x},\mathbf{y}):=\frac{i}{4}H_0^{(1)}(k_{a,\pm}|\mathbf{x}-\mathbf{y}|)$, $a=p,s$,  the fundamental solution for $\Delta+k_{a,\pm}^2$ in $\R^2$. It is known that: 

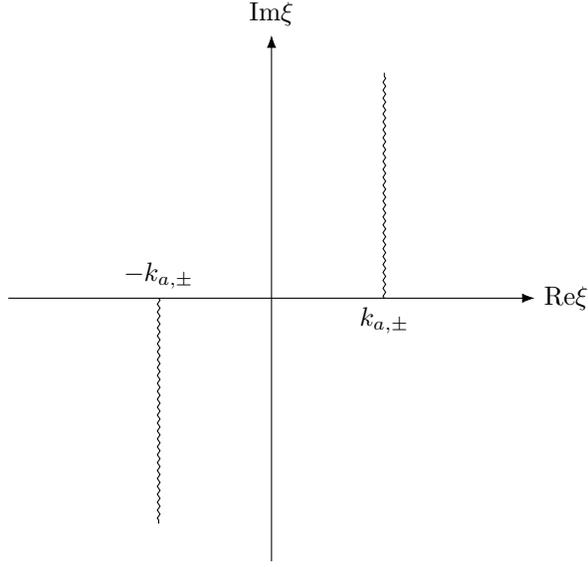
\begin{figure}\label{fig3.1}
	\begin{adjustbox}{margin=3cm 0cm 0cm 0cm}
		\begin{tikzpicture}
			\draw[-Latex] (-3.5,0) -- (3.5,0) node[right]{$\Rt\xi$}; 
			\draw[-Latex] (0,-3.5) -- (0,3.5) node[above]{$\I\xi$}; 
			\draw[decorate,decoration={zigzag, amplitude=0.5pt, segment length=3pt}] (1.5,0) node[below]{$k_{a,\pm}$} -- (1.5,3); 
			\draw[decorate,decoration={zigzag, amplitude=0.5pt, segment length=3pt}] (-1.5,0) node[above]{$-k_{a,\pm}$} -- (-1.5,-3); 
		\end{tikzpicture}
	\end{adjustbox}
	\caption{Branch cuts and domain of definition for the function $\beta_{a,\pm}$.}
\end{figure}

\begin{theorem}\label{thm3.1}
	For $a=p,s$ and $j=1,2$, the function $\wid G_{a,j}$ exists and is unique. Further, they have the forms 
	\begin{align*}
	  &\wid G_{p,1}(\mathbf{x},\mathbf{y})=\frac{1}{2\mu+\lambda} \\
	  &\left\{
	  \begin{array}{ll}
	  	\displaystyle\frac{\pa\Phi_{k_{p,+}}(\mathbf{x},\mathbf{y})}{\pa x_1}+\frac{i}{4\pi}\int_{-\infty}^{+\infty}\frac{\xi}{\beta_{p,+}}R_p(\beta_{p,+},\beta_{p,-})e^{-\beta_{p,+}(x_2+y_2)}e^{i\xi(x_1-y_1)}d\xi,&\mathbf{x}\in\R^2_+,\;\mathbf{y}\in\R^2_+,\\
	  	\displaystyle\frac{i}{4\pi}\int_{-\infty}^{+\infty}\frac{\xi}{\beta_{p,+}}T_p(\beta_{p,+},\beta_{p,-})e^{\beta_{p,-}x_2-\beta_{p,+}y_2}e^{i\xi(x_1-y_1)}d\xi,&\mathbf{x}\in\R^2_-,\;\mathbf{y}\in\R^2_+, \\
	  	\displaystyle-\frac{i}{4\pi}\int_{-\infty}^{+\infty}\frac{\xi}{\beta_{p,-}}(T_p(\beta_{p,+},\beta_{p,-})-2)e^{-\beta_{p,+}x_2+\beta_{p,-}y_2}e^{i\xi(x_1-y_1)}d\xi,&\mathbf{x}\in\R^2_+,\;\mathbf{y}\in\R^2_-, \\
	  	\displaystyle\frac{\pa\Phi_{k_{p,-}}(\mathbf{x},\mathbf{y})}{\pa x_1}-\frac{i}{4\pi}\int_{-\infty}^{+\infty}\frac{\xi}{\beta_{p,-}}R_p(\beta_{p,+},\beta_{p,-})e^{\beta_{p,-}(x_2+y_2)}e^{i\xi(x_1-y_1)}d\xi,&\mathbf{x}\in\R^2_-,\;\mathbf{y}\in\R^2_-, 
	  \end{array}
	  \right.
	\end{align*}
	\begin{align*}
	&\wid G_{s,1}(\mathbf{x},\mathbf{y})=\frac{1}{\mu} \\
	&\left\{
	\begin{array}{ll}
		\displaystyle\frac{\pa\Phi_{k_{s,+}}(\mathbf{x},\mathbf{y})}{\pa x_2}+\frac{1}{4\pi}\int_{-\infty}^{+\infty}R_s(\beta_{s,+},\beta_{s,-})e^{-\beta_{s,+}(x_2+y_2)}e^{i\xi(x_1-y_1)}d\xi,&\mathbf{x}\in\R^2_+,\;\mathbf{y}\in\R^2_+,\\
		\displaystyle\frac{1}{4\pi}\int_{-\infty}^{+\infty}T_s(\beta_{s,+},\beta_{s,-})e^{\beta_{s,-}x_2-\beta_{s,+}y_2}e^{i\xi(x_1-y_1)}d\xi,&\mathbf{x}\in\R^2_-,\;\mathbf{y}\in\R^2_+, \\
		\displaystyle\frac{1}{4\pi}\int_{-\infty}^{+\infty}(T_s(\beta_{s,+},\beta_{s,-})-2)e^{-\beta_{s,+}x_2+\beta_{s,-}y_2}e^{i\xi(x_1-y_1)}d\xi,&\mathbf{x}\in\R^2_+,\;\mathbf{y}\in\R^2_-, \\
		\displaystyle\frac{\pa\Phi_{k_{s,-}}(\mathbf{x},\mathbf{y})}{\pa x_2}+\frac{1}{4\pi}\int_{-\infty}^{+\infty}R_s(\beta_{s,+},\beta_{s,-})e^{\beta_{s,-}(x_2+y_2)}e^{i\xi(x_1-y_1)}d\xi,&\mathbf{x}\in\R^2_-,\;\mathbf{y}\in\R^2_-,
	\end{array}
	\right.
	\end{align*}
	\begin{align*}
	&\wid G_{p,2}(\mathbf{x},\mathbf{y})=\frac{1}{2\mu+\lambda} \\
	&\left\{
	\begin{array}{ll}
		\displaystyle\frac{\pa\Phi_{k_{p,+}}(\mathbf{x},\mathbf{y})}{\pa x_2}+\frac{1}{4\pi}\int_{-\infty}^{+\infty}R_p(\beta_{p,+},\beta_{p,-})e^{-\beta_{p,+}(x_2+y_2)}e^{i\xi(x_1-y_1)}d\xi,&\mathbf{x}\in\R^2_+,\;\mathbf{y}\in\R^2_+,\\
		\displaystyle\frac{1}{4\pi}\int_{-\infty}^{+\infty}T_p(\beta_{p,+},\beta_{p,-})e^{\beta_{p,-}x_2-\beta_{p,+}y_2}e^{i\xi(x_1-y_1)}d\xi,&\mathbf{x}\in\R^2_-,\;\mathbf{y}\in\R^2_+, \\
		\displaystyle\frac{1}{4\pi}\int_{-\infty}^{+\infty}(T_p(\beta_{p,+},\beta_{p,-})-2)e^{-\beta_{p,+}x_2+\beta_{p,-}y_2}e^{i\xi(x_1-y_1)}d\xi,&\mathbf{x}\in\R^2_+,\;\mathbf{y}\in\R^2_-, \\
		\displaystyle\frac{\pa\Phi_{k_{p,-}}(\mathbf{x},\mathbf{y})}{\pa x_2}+\frac{1}{4\pi}\int_{-\infty}^{+\infty}R_p(\beta_{p,+},\beta_{p,-})e^{\beta_{p,-}(x_2+y_2)}e^{i\xi(x_1-y_1)}d\xi,&\mathbf{x}\in\R^2_-,\;\mathbf{y}\in\R^2_-,
	\end{array}
	\right.
	\end{align*}
	and
	\begin{align*}
	&\wid G_{s,2}(\mathbf{x},\mathbf{y})=-\frac{1}{\mu} \\
	&\left\{
	\begin{array}{ll}
		\displaystyle\frac{\pa\Phi_{k_{s,+}}(\mathbf{x},\mathbf{y})}{\pa x_1}+\frac{i}{4\pi}\int_{-\infty}^{+\infty}\frac{\xi}{\beta_{s,+}}R_s(\beta_{s,+},\beta_{s,-})e^{-\beta_{s,+}(x_2+y_2)}e^{i\xi(x_1-y_1)}d\xi,&\mathbf{x}\in\R^2_+,\;\mathbf{y}\in\R^2_+,\\
		\displaystyle\frac{i}{4\pi}\int_{-\infty}^{+\infty}\frac{\xi}{\beta_{s,+}}T_s(\beta_{s,+},\beta_{s,-})e^{\beta_{s,-}x_2-\beta_{s,+}y_2}e^{i\xi(x_1-y_1)}d\xi,&\mathbf{x}\in\R^2_-,\;\mathbf{y}\in\R^2_+, \\
		\displaystyle-\frac{i}{4\pi}\int_{-\infty}^{+\infty}\frac{\xi}{\beta_{s,-}}(T_s(\beta_{s,+},\beta_{s,-})-2)e^{-\beta_{s,+}x_2+\beta_{s,-}y_2}e^{i\xi(x_1-y_1)}d\xi,&\mathbf{x}\in\R^2_+,\;\mathbf{y}\in\R^2_-, \\
		\displaystyle\frac{\pa\Phi_{k_{s,-}}(\mathbf{x},\mathbf{y})}{\pa x_1}-\frac{i}{4\pi}\int_{-\infty}^{+\infty}\frac{\xi}{\beta_{s,-}}R_s(\beta_{s,+},\beta_{s,-})e^{\beta_{s,-}(x_2+y_2)}e^{i\xi(x_1-y_1)}d\xi,&\mathbf{x}\in\R^2_-,\;\mathbf{y}\in\R^2_-. 
	\end{array}
	\right.
	\end{align*}
\end{theorem}

For $U_{p,j}:=G_{p,j}-\wid G_{p,j}$, $a=p,s$, $j=1,2$, it is then deduced that 
\be\label{3.4}
\left\{
\begin{array}{ll}
	\displaystyle\Delta_\mathbf{x}U_{p,j}(\mathbf{x},\mathbf{y})+k_{p,\pm}^2U_{p,j}(\mathbf{x},\mathbf{y})=0,&{\rm in}\;\R^2_\pm,~ \\ 
	\displaystyle\Delta_\mathbf{x}U_{s,j}(\mathbf{x},\mathbf{y})+k_{s,\pm}^2U_{s,j}(\mathbf{x},\mathbf{y})=0,&{\rm in}\;\R^2_\pm,~ \\
	\left[U_{p,j}(\mathbf{x},\mathbf{y})\right]=\left[U_{s,j}(\mathbf{x},\mathbf{y})\right]=0,&{\rm on}\;\Gamma_0,~ \\
	\displaystyle\left[k_p^{-2}\pa_{\boldsymbol{\nu}(\mathbf{x})}U_{p,j}(\mathbf{x},\mathbf{y})-k_s^{-2}\pa_{\boldsymbol{\tau}(\mathbf{x})}U_{s,j}(\mathbf{x},\mathbf{y})\right]=[k_{s}^{-2}\pa_{\boldsymbol{\tau}(\mathbf{x})}\wid G_{s,j}(\mathbf{x},\mathbf{y})],&{\rm on}\;\Gamma_0,~ \\
	\displaystyle\left[-k_p^{-2}\pa_{\boldsymbol{\tau}(\mathbf{x})}U_{p,j}(\mathbf{x},\mathbf{y})-k_s^{-2}\pa_{\boldsymbol{\nu}(\mathbf{x})}U_{s,j}(\mathbf{x},\mathbf{y})\right]=[k_{p}^{-2}\pa_{\boldsymbol{\tau}(\mathbf{x})}\wid G_{p,j}(\mathbf{x},\mathbf{y})],&{\rm on}\;\Gamma_0,~ \\
	\lim\limits_{r\rightarrow\infty}\sqrt{r}\left(\pa_rU_{a,j}(\mathbf{x},\mathbf{y})-ik_{a,\pm}U_{a,j}(\mathbf{x},\mathbf{y})\right)=0,~a=p,s,&\mathbf{\hat x}\in\Sp^1_\pm.~ 
\end{array}
\right.
\en
Again, $U_{a,j}$ is uniquely defined by Theorem \ref{thm2.4}. We now utilize the Fourier transform and its inverse to derive the expressions of $U_{a,j}$. 

Take the Fourier transform on the equations of \eqref{3.4} in the $x_1$-variable to obtain that 
$$
 \pa_{x_2}^2\hat U_{a,j}(\xi,x_2,y_2)+(k_{a}^2-\xi^2)\hat U_{a,j}(\xi,x_2,y_2)=0
$$
for $a=p,s$ and $j=1,2$. From the basic knowledge of ODEs and the radiation condition, we find that 
\be\label{3.5}
\hat U_{a,j}(\xi,x_2,y_2)=
\left\{
\begin{array}{ll}
	A_{a,j}^+e^{-\beta_{a,+}x_2},~~~&x_2>0,\;y_2>0,\\
	A_{a,j}^-e^{\beta_{a,-}x_2},~~~&x_2<0,\;y_2>0, \\
	B_{a,j}^+e^{-\beta_{a,+}x_2},~~~&x_2>0,\;y_2<0, \\
	B_{a,j}^-e^{\beta_{a,-}x_2},~~~&x_2<0,\;y_2<0. 
\end{array}
\right.
\en
Since $\hat U_{a,j}$ is continuous at $x_2=0$, it is deduced that $A_{a,j}^+=A_{a,j}^-=:A_{a,j}$ and $B_{a,j}^+=B_{a,j}^-=:B_{a,j}$ for $a=p,s$ and $j=1,2$. Further, the boundary conditions concerning the derivatives in \eqref{3.4} after the Fourier transform become 
\begin{align*}
  &\left[k_p^{-2}\pa_{x_2}\hat U_{p,j}(\xi,x_2,y_2)-i\xi k_s^{-2}\hat U_{s,j}(\xi,x_2,y_2)\right]=[i\xi k_{s}^{-2}\hat{\wid G}_{s,j}(\xi,x_2,y_2)], \\
  &\left[i\xi k_p^{-2}\hat U_{p,j}(\xi,x_2,y_2)+k_s^{-2}\pa_{x_2}\hat U_{s,j}(\xi,x_2,y_2)\right]=-[i\xi k_{p}^{-2}\hat{\wid G}_{p,j}(\xi,x_2,y_2)]. 
\end{align*}
Substituting the expressions \eqref{3.5} into the above identites and solving the resulting system of equations for the $A_{a,j}$ and $B_{a,j}$, we readily get 
\ben
  &&A_{p,1}=\frac{C_0\xi}{D}i\left[k_{s,+}^{-2}\beta_{s,+}e^{-\beta_{s,+}y_2}-\frac{(k_{s,+}^{-2}-k_{s,-}^{-2})k_{p,+}^{-2}\xi^2}{k_{p,+}^{-2}\beta_{p,+}+k_{p,-}^{-2}\beta_{p,-}}e^{-\beta_{p,+}y_2}\right], \\
  &&B_{p,1}=\frac{C_0\xi}{D}i\left[-k_{s,-}^{-2}\beta_{s,-}e^{\beta_{s,-}y_2}-\frac{(k_{s,+}^{-2}-k_{s,-}^{-2})k_{p,-}^{-2}\xi^2}{k_{p,+}^{-2}\beta_{p,+}+k_{p,-}^{-2}\beta_{p,-}}e^{\beta_{p,-}y_2}\right], \\
  &&A_{s,1}=\frac{C_0\xi}{D}\left[-\frac{(k_{p,+}^{-2}-k_{p,-}^{-2})k_{s,+}^{-2}\beta_{s,+}\xi }{k_{s,+}^{-2}\beta_{s,+}+k_{s,-}^{-2}\beta_{s,-}}e^{-\beta_{s,+}y_2}+k_{p,+}^{-2}\xi e^{-\beta_{p,+}y_2}\right], \\
  &&B_{s,1}=\frac{C_0\xi}{D}\left[\frac{(k_{p,+}^{-2}-k_{p,-}^{-2})k_{s,-}^{-2}\beta_{s,-}\xi }{k_{s,+}^{-2}\beta_{s,+}+k_{s,-}^{-2}\beta_{s,-}}e^{\beta_{s,-}y_2}+k_{p,-}^{-2}\xi e^{\beta_{p,-}y_2}\right], \\
  &&A_{p,2}=\frac{C_0\xi}{D}\left[k_{s,+}^{-2}\xi e^{-\beta_{s,+}y_2}-\frac{(k_{s,+}^{-2}-k_{s,-}^{-2})k_{p,+}^{-2}\beta_{p,+}\xi }{k_{p,+}^{-2}\beta_{p,+}+k_{p,-}^{-2}\beta_{p,-}}e^{-\beta_{p,+}y_2}\right], \\
  &&B_{p,2}=\frac{C_0\xi}{D}\left[k_{s,-}^{-2}\xi e^{\beta_{s,-}y_2}+\frac{(k_{s,+}^{-2}-k_{s,-}^{-2})k_{p,-}^{-2}\beta_{p,-}\xi }{k_{p,+}^{-2}\beta_{p,+}+k_{p,-}^{-2}\beta_{p,-}}e^{\beta_{p,-}y_2}\right], \\
  &&A_{s,2}=\frac{C_0\xi}{D}i\left[\frac{(k_{p,+}^{-2}-k_{p,-}^{-2})k_{s,+}^{-2}\xi^2}{k_{s,+}^{-2}\beta_{s,+}+k_{s,-}^{-2}\beta_{s,-}}e^{-\beta_{s,+}y_2}-k_{p,+}^{-2}\beta_{p,+}e^{-\beta_{p,+}y_2}\right], \\
  &&B_{s,2}=\frac{C_0\xi}{D}i\left[\frac{(k_{p,+}^{-2}-k_{p,-}^{-2})k_{s,-}^{-2}\xi^2}{k_{s,+}^{-2}\beta_{s,+}+k_{s,-}^{-2}\beta_{s,-}}e^{\beta_{s,-}y_2}+k_{p,-}^{-2}\beta_{p,-}e^{\beta_{p,-}y_2}\right], 
\enn
where the constant $\displaystyle C_0:=\frac{1}{\rho_+\om^2}-\frac{1}{\rho_-\om^2}$ and 
\ben
\begin{aligned}
	D:=&
\left|
\begin{array}{cc}
	i(k_{p,+}^{-2}\beta_{p,+}+k_{p,-}^{-2}\beta_{p,-}) & -(k_{s,+}^{-2}-k_{s,-}^{-2})\xi  \\
	(k_{p,+}^{-2}-k_{p,-}^{-2})\xi  & i(k_{s,+}^{-2}\beta_{s,+}+k_{s,-}^{-2}\beta_{s,-})
\end{array}
\right| \\
=&-(k_{p,+}^{-2}\beta_{p,+}+k_{p,-}^{-2}\beta_{p,-})(k_{s,+}^{-2}\beta_{s,+}+k_{s,-}^{-2}\beta_{s,-})+(k_{p,+}^{-2}-k_{p,-}^{-2})(k_{s,+}^{-2}-k_{s,-}^{-2})\xi^2. 
\end{aligned}
\enn
Note that it is easy to deduce that $D$ have no zero in the real line. The function $U_{a,j}$ is then obtained from its Fourier transform by 
\be\label{3.6}
  U_{a,j}(\mathbf{x},\mathbf{y})=\frac{1}{2\pi}\int_{-\infty}^{+\infty}\hat{U}_{a,j}(\xi,x_2,y_2)e^{i(x_1-y_1)\xi}d\xi 
\en
for $a=p,s$ and $j=1,2$. 

It remains to verify that $U_{a,j}$ given by \eqref{3.6} actually satisfies \eqref{3.4}, while the equations and the boundary conditions in \eqref{3.4} clearly hold. Thus, it suffices to check the radiation condition. To this end, we consider the integral \eqref{3.6} in the whole complex plane and derive the asymptotic behavior of $U_{a,j}$ as $r\rightarrow\infty$ through the standard procedure using the steepest descent method, for which we refer to \cite{HED05,NB84,LJ24,CP17}. The steepest-descents formulas we mainly used are given in \cite[Appendix A]{CP17}. 

The analysis presented here follows closely to \cite[Section 2.3.4]{CP17}. We start with the asymptotic behavior of $U_{p,1}$ when $\mathbf{x},\mathbf{y}\in\R^2_+$. Expressing $\mathbf{x}$ in polar coordinates, i.e., $\mathbf{x}=r\hat{\mathbf{x}}=r(\cos\theta_x,\sin\theta_x)$, $\theta_x\in(0,\pi)$, and letting $\eta=\xi/k_{p,+}$, \eqref{3.6} becomes 
\be\label{3.7}
  U_{p,1}(\mathbf{x},\mathbf{y})=\frac{k_{p,+}}{2\pi}\int_{-\infty}^{+\infty}A_{p,1}(k_{p,+}\eta,y_2)e^{-ik_{p,+}y_1\eta}e^{r\phi(\eta)}d\eta
\en
where $\phi(\eta)=ik_{p,+}(\eta\cos\theta_x+i\sqrt{\eta^2-1}\sin\theta_x)$ and $\sqrt{\eta^2-1}=\beta_{p,+}(k_{p,+}\eta)/k_{p,+}$. Note that $\phi(\eta)$ here is the same as in \cite[Section 2.3.4]{CP17}, thus our analysis concerning the saddle point and the path of steepest descent is also the same. We know from \cite{CP17} that there is only one saddle point $\eta_0=\cos\theta_x$ for $\phi$ with $\phi(\eta_0)=ik_{p,+}$ and $\phi''(\eta_0)=-ik_{p,+}/\sin^2\theta_x$, and the directions of steepest descent at $\eta_0$ are given by $\alpha_0=3\pi/4$ and $\alpha_1=7\pi/4$. Moreover, the path of steepest descent $\mathcal{D}$ approaches the line 
$$
  \I\eta=\frac{\cos\theta_x}{\sin\theta_x}\Rt\eta-\frac{1}{\sin\theta_x}~~{\rm as}~~|\eta|\rightarrow\infty
$$
for $\Rt\eta>0$ and  
$$
\I\eta=-\frac{\cos\theta_x}{\sin\theta_x}\Rt\eta+\frac{1}{\sin\theta_x}~~{\rm as}~~|\eta|\rightarrow\infty
$$
for $\Rt\eta<0$, while it intersects the real axis at two points $\eta=\cos\theta_x,1/\cos\theta_x$. The paths of steepest descent with $x_1>0$ and $x_1<0$ in one case are depicted in Figure 3.2. 

\begin{figure}\label{fig3.2}
\begin{adjustbox}{margin=-0.5cm 0cm 0cm 0cm}
\begin{tikzpicture}
	\tikzmath{\value=sqrt(3)/20;}
	\draw[-Latex] (-6,0) -- (6,0) node[right]{$\Rt\eta$}; 
	\draw[-Latex] (0,-4.5) -- (0,4.5) node[above]{$\I\eta$}; 
	\draw[decorate,decoration={zigzag, amplitude=0.5pt, segment length=3pt}] (1.25,0) node[font=\tiny, below]{$\frac{k_{p,-}}{k_{p,+}}$} -- (1.25,3.5); 
	\draw[decorate,decoration={zigzag, amplitude=0.5pt, segment length=3pt}] (-1.25,0) node[font=\tiny,above]{$-\frac{k_{p,-}}{k_{p,+}}$} -- (-1.25,-3.5); 
	\draw[decorate,decoration={zigzag, amplitude=0.5pt, segment length=3pt}] (2.5,0) node[below]{$1$} -- (2.5,3.5); 
	\draw[decorate,decoration={zigzag, amplitude=0.5pt, segment length=3pt}] (-2.5,0) node[above]{$-1$} -- (-2.5,-3.5); 
	\draw[decorate,decoration={zigzag, amplitude=0.5pt, segment length=3pt}] (3.75,0) node[font=\tiny,below]{$\frac{k_{s,-}}{k_{p,+}}$} -- (3.75,3.5); 
	\draw[decorate,decoration={zigzag, amplitude=0.5pt, segment length=3pt}] (-3.75,0) node[font=\tiny,above]{$-\frac{k_{s,-}}{k_{p,+}}$} -- (-3.75,-3.5); 
	\draw[decorate,decoration={zigzag, amplitude=0.5pt, segment length=3pt}] (5,0) node[font=\tiny,below]{$\frac{k_{s,+}}{k_{p,+}}$} -- (5,3.5); 
	\draw[decorate,decoration={zigzag, amplitude=0.5pt, segment length=3pt}] (-5,0) node[font=\tiny,above]{$-\frac{k_{s,+}}{k_{p,+}}$} -- (-5,-3.5); 
	\draw[red, smooth, -Latex] (-0.8,3.4) -- (0.7,0) node[font=\tiny, black, below left]{$\cos\theta_x$} to[out=-atan(34/15), in=180] (2.3,-1.4) to[out=0,in=180+atan(9/4)] (4.6,0) node[font=\tiny,black, above left]{$\frac{1}{\cos\theta_x}$} -- (4.95,0.7875) -- (4.95,\value) arc (120:420:0.1) -- (5.05,1.0125) -- (6.2,3.4) node[black,above]{($x_1>0$)} node[black, below right]{$\mathcal{D}$}; 
	\draw[red, smooth,Latex-] (0.8,-3.4) node[black, above right]{$\mathcal{D}$} -- (-0.7,0) node[font=\tiny, black, above right]{$\cos\theta_x$} to[out=180-atan(34/15), in=0] (-2.3,1.4) to[out=180,in=atan(9/4)] (-4.6,0) node[font=\tiny, black, below right]{$\frac{1}{\cos\theta_x}$} -- (-4.95,-0.7875) -- (-4.95,-\value) arc (-60:240:0.1) -- (-5.05,-1.0125) -- (-6.2,-3.4) node[black,below]{$(x_1<0)$}; 
\end{tikzpicture}
\end{adjustbox}
\caption{Paths of descent for the integrals \eqref{3.7} in the case that $\sin\theta_x>0$ and $|\cos\theta_x|<\frac{k_{p,-}}{k_{p,+}}<1$,  $\frac{k_{s,-}}{k_{p,+}}<\left|\frac{1}{\cos\theta_x}\right|<\frac{k_{s,+}}{k_{p,+}}$.}
\end{figure}
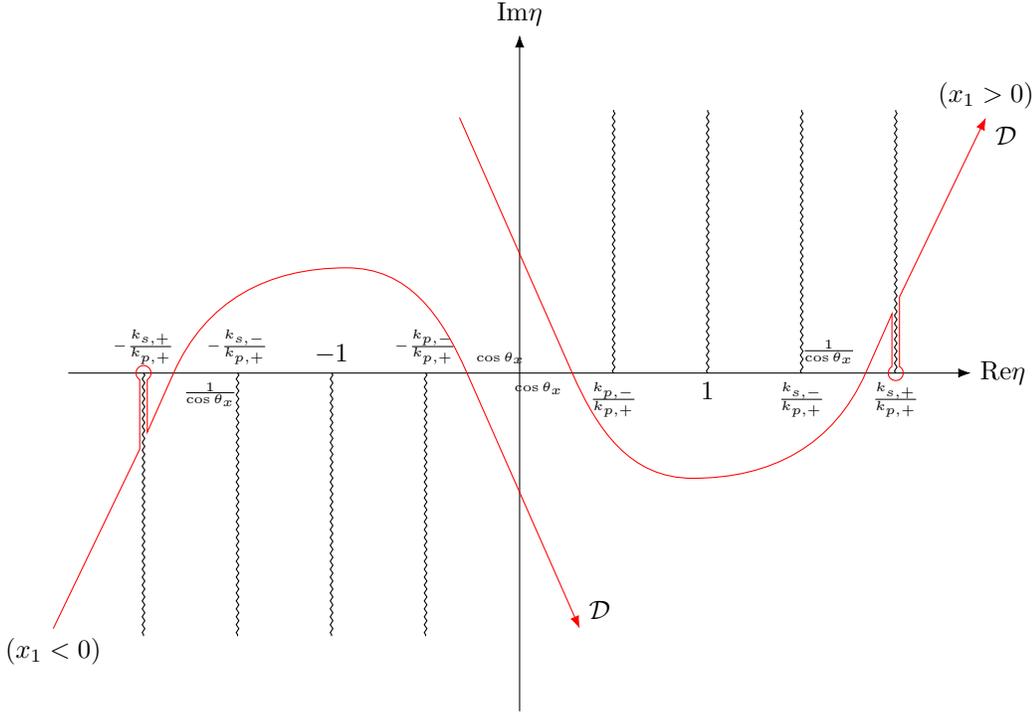

There are three possible contributions to the asymptotic expansion of \eqref{3.7}. The first contribution comes from the poles of $A_{p,1}$. We need to calculate the residues corresponding to the poles of $A_{p,1}$ when deforming the oringinal contour to the path of steepest descent $\mathcal{D}$. Since by \cite[Section 2.3.4]{CP17} we know that $k_{a,+}^{-2}\beta_{a,+}+k_{a,-}^{-2}\beta_{a,-},a=p,s$ have no zero in the complex plane, the poles can only be the zeros of $D$. Let $\wid\eta$ be a zero of $D(k_{p,+}\eta)$ lying in the region enclosed by the real line and the path $\mathcal{D}$, then by direct calculation the corresponding residue would be of $O(|f(r)|e^{r\Rt\phi(\wid\eta)})$ as $r\rightarrow\infty$, where $f$ is a polynomial. Since $\wid\eta$ is in the region between the real line and the path of steepest descent, we can connect $\eta_0$ and $\wid\eta$ by first moving along the path $\mathcal{D}$ and then moving horizontally without intersecting with the branch cut of $\sqrt{\eta^2-1}$, while from \cite[Section 7.2]{NB84} it is known that through this path the real part of $\phi$ is non-increasing, which implies that $\Rt\phi(\wid\eta)<\Rt\phi(\eta_0)=0$. Thus, the residue in consideration decays exponentially. 

The second is the saddle point contribution, which is calculated by $(A.2)$ or $(A.3)$ in \cite{CP17}. If the saddle point $\eta_0=\pm k/k_{p,+}$ with $k\in\{k_{p,\pm},k_{s,\pm}\}$ and $k<k_{p,+}$, it can be shown that the integrand $A_{p,1}$ can be expressed as $p(\eta)+q(\eta)(\eta-\eta_0)^{1/2}$, where $p$ and $q$ are analytic in a neighborhood of $\eta_0$. Then the saddle point contribution to the asymptotic expansion of \eqref{3.7} as $r\rightarrow\infty$, calculated by $(A.3)$ in \cite{CP17}, is given by 
\be\label{3.8}
  e^{-i\frac{\pi}{4}}\sqrt{\frac{k_{p,+}}{2\pi}}\frac{e^{ik_{p,+}r}}{\sqrt{r}}\sin\theta_xA_{p,1}(k_{p,+}\cos\theta_x,y_2)e^{-ik_{p,+}y_1\cos\theta_x}+O\left(r^{-\frac{3}{4}}\right). 
\en
In other cases, we can directly apply $(A.2)$ in \cite{CP17} to obtain that the corresponding contribution is the same as \eqref{3.8}, except that the remainder term becomes $O(r^{-3/2})$. 

The last possible contribution, given by $(A.4)$ in \cite{CP17}, arises when the path of steepest descent has to be locally deformed around the branch cuts, as it is shown in Figure 3.2. Suppose that the path $\mathcal{D}$ cross the branch cut starting at the branch point $\eta=k/k_{p,+}$ with $k\in\{k_{p,\pm},k_{s,\pm}\}$ and $k\neq k_{p,+}$. Following the lines in \cite[Section 2.3.4]{CP17}, by patient but straightforword calculation, we see that the jump of $A_{p,1}$ across the branch cut stemming from $\eta=k/k_{p,+}$ is $C(\eta-k/k_{p,+})^{1/2}$ as $\eta\rightarrow k/k_{p,+}$, where $C$ is independent of $\eta$. Therefore, in view of $(A.4)$ in \cite{CP17}, the corresponding contribution would be of $O(r^{-3/2})$ as $r\rightarrow\infty$. 

Combining the above three aspects, we finally obtain that the asymptotic expansion of \eqref{3.7} is given by \eqref{3.8}. Applying the similar analysis to all $U_{a,j}$ and $\nabla_{\mathbf{y}}U_{a,j}$, we can obtain the following results. 
\begin{theorem}\label{thm3.2}
	For $a=p,s$ and $j=1,2$, if $\mathbf{y}=(y_1,y_2)\in\R^2_+\cup\R^2_-$ and $\mathbf{x}=r\hat{\mathbf{x}}=r(\cos\theta_x,\sin\theta_x)\in \R^2_+$ with $\theta_x\in(0,\pi)$, we have the asymptotic behaviors
	\begin{align*}
	  U_{a,j}(\mathbf{x},\mathbf{y})&=\frac{e^{ik_{a,+}r}}{\sqrt{r}}U_{a,j}^\infty(\hat{\mathbf{x}},\mathbf{y})+O(r^{-\frac{3}{4}}), \\
	  \nabla_{\mathbf{y}}U_{a,j}(\mathbf{x},\mathbf{y})&=\frac{e^{ik_{a,+}r}}{\sqrt{r}}\nabla_{\mathbf{y}}U_{a,j}^\infty(\hat{\mathbf{x}},\mathbf{y})+O(r^{-\frac{3}{4}}), 
	\end{align*}
	where $U_{a,j}^\infty$ are defined by 
	\ben
	U_{a,j}^\infty(\hat{\mathbf{x}},\mathbf{y}):=e^{-i\frac{\pi}{4}}\sqrt{\frac{k_{a,+}}{2\pi}}
	\left\{
	\begin{array}{ll}
		\sin\theta_xA_{a,j}(k_{a,+}\cos\theta_x,y_2)e^{-ik_{a,+}y_1\cos\theta_x},\;&\mathbf{x}\in\R^2_+,\;\mathbf{y}\in\R^2_+,\\
		\sin\theta_xB_{a,j}(k_{a,+}\cos\theta_x,y_2)e^{-ik_{a,+}y_1\cos\theta_x},\;&\mathbf{x}\in\R^2_+,\;\mathbf{y}\in\R^2_-. 
	\end{array}
	\right.
	\enn
\end{theorem}
\begin{theorem}\label{thm3.3}
	For $a=p,s$ and $j=1,2$, if $\mathbf{y}=(y_1,y_2)\in\R^2_+\cup\R^2_-$ and $\mathbf{x}=r\hat{\mathbf{x}}=r(\cos\theta_x,\sin\theta_x)\in \R^2_-$ with $\theta_x\in(\pi,2\pi)$, we have the asymptotic behaviors 
	\begin{align*}
	U_{a,j}(\mathbf{x},\mathbf{y})&=\frac{e^{ik_{a,-}r}}{\sqrt{r}}U_{a,j}^\infty(\hat{\mathbf{x}},\mathbf{y})+O(r^{-\frac{3}{4}}), \\
	\nabla_{\mathbf{y}}U_{a,j}(\mathbf{x},\mathbf{y})&=\frac{e^{ik_{a,-}r}}{\sqrt{r}}\nabla_{\mathbf{y}}U_{a,j}^\infty(\hat{\mathbf{x}},\mathbf{y})+O(r^{-\frac{3}{4}}), 
	\end{align*}
	where $U_{a,j}^\infty$ are defined by 
	\ben
	U_{a,j}^\infty(\hat{\mathbf{x}},\mathbf{y}):=e^{i\frac{3\pi}{4}}\sqrt{\frac{k_{a,-}}{2\pi}}
	\left\{
	\begin{array}{ll}
		\sin\theta_xA_{a,j}(k_{a,-}\cos\theta_x,y_2)e^{-ik_{a,-}y_1\cos\theta_x},\;&\mathbf{x}\in\R^2_-,\;\mathbf{y}\in\R^2_+,\\
		\sin\theta_xB_{a,j}(k_{a,-}\cos\theta_x,y_2)e^{-ik_{a,-}y_1\cos\theta_x},\;&\mathbf{x}\in\R^2_-,\;\mathbf{y}\in\R^2_-. 
	\end{array}
	\right.
	\enn
\end{theorem}

From Theorems \ref{thm3.2} and \ref{thm3.3}, clearly all $U_{a,j}$ satisfy the Sommerfeld radiation condition, thus we finally obtain the existence of the Green's tensor. 
\begin{theorem}\label{thm3.4}
	The two-layered Green's tensor $\mathbf{G}(\mathbf{x},\mathbf{y})$ defined by \eqref{3.1} exists and is unique. Moreover, it has the explicit form $\mathbf{G}(\mathbf{x},\mathbf{y})=(\mathbf{G}_1(\mathbf{x},\mathbf{y}),\mathbf{G}_2(\mathbf{x},\mathbf{y}))$ with the column vector $\mathbf{G}_j=-k_p^{-2}\grad_{\mathbf{x}}G_{p,j}-k_s^{-2}\grad_{\mathbf{x}}^\perp G_{s,j}$, $j=1,2$, where $G_{a,j}=\wid G_{a,j}+U_{a,j}$ with $\wid G_{a,j}$ and $U_{a,j}$ given by Theorem {\rm \ref{thm3.1}} and \eqref{3.6}, respectively. 
\end{theorem}
\begin{remark}\label{rem3.5}
	By the explicit expression of $G_{a,j}$, it can be verified that $\mathbf{G}(\mathbf{x},\mathbf{y})=\mathbf{G}(\mathbf{y},\mathbf{x})^\top$ for $\mathbf{x},\mathbf{y}\in\R^2_\pm$. 
\end{remark}

\subsection{Three-dimensional case}\label{sec3.2}
For any source point $\mathbf{y}\in\R^3_+\cup\R^3_-$, the two-layered Green's tensor $\mathbf{G}(\mathbf{x},\mathbf{y})$ in three dimensions is defined by: 
\be\label{3.9}
   \left\{
  \begin{array}{ll}
  	\Delta^*_{\mathbf{x}}\mathbf{G}(\mathbf{x},\mathbf{y})+\rho_\pm\om^2\mathbf{G}(\mathbf{x},\mathbf{y})=-\delta(\mathbf{x},\mathbf{y})\mathbf{I},~~~&{\rm in}~\R^3_\pm,~\mathbf{x}\neq\mathbf{y}\\
  	\left[\mathbf{G}(\mathbf{x},\mathbf{y})\right]=0,~\left[\mathbf{P}_{\tilde{\mu},\tilde{\lambda}}^{(\mathbf{x})}(\mathbf{G}(\mathbf{x},\mathbf{y}))\right]=0~~~&{\rm on}~\Gamma_0, \\
  	\lim\limits_{r\rightarrow\infty}r\left(\pa_r\mathbf{G}_p(\mathbf{x},\mathbf{y})-ik_{p,\pm}\mathbf{G}_p(\mathbf{x},\mathbf{y})\right)=0,~~~&\mathbf{\hat x}\in\Sp^2_\pm, \\
  	\lim\limits_{r\rightarrow\infty}r(\curl\mathbf{G}_s(\mathbf{x},\mathbf{y})\times\hat{\mathbf{x}}-ik_{s,\pm}\mathbf{G}_s(\mathbf{x},\mathbf{y}))=0,~~~&\mathbf{\hat x}\in\Sp^2_\pm, 
  \end{array}
  \right.
\en
where $\delta(\mathbf{x},\mathbf{y})$ denotes the Dirac delta distribution, $\mathbf{I}$ is the $3\times3$ identity matrix, and $\mathbf{G}_p:=\dive_{\mathbf{x}}\mathbf{G}, \mathbf{G}_s:=\curl_{\mathbf{x}}\mathbf{G}$. Still, $\Delta^*_{\mathbf{x}}, \mathbf{P}_{\tilde{\mu},\tilde{\lambda}}^{(\mathbf{x})},\dive_{\mathbf{x}}$ and $\curl_{\mathbf{x}}$ are acting on $\mathbf{G}_j, j=1,2,3$, where $\mathbf{G}=(\mathbf{G}_1,\mathbf{G}_2,\mathbf{G}_3)$ with $\mathbf{G}_j$ being the column vectors. Set $G_{p,j}:=\dive_{\mathbf{x}}\mathbf{G}_j$ and $\mathbf{G}_{s,j}:=\curl_{\mathbf{x}}\mathbf{G}_j=(G_{s,j}^{(1)},G_{s,j}^{(2)},G_{s,j}^{(3)})^\top$ with $j=1,2,3$. 

From Theorem \ref{thm2.4}, $\mathbf{G}$ defined above is unique. Again by the Helmholtz decomposition and Corollary \ref{cor2.3}, we can equivalently reduce \eqref{3.9} to the following: for $j=1,2,3$, 
\be\label{3.10}
\left\{
\begin{array}{ll}
	\displaystyle\Delta_\mathbf{x}G_{p,j}(\mathbf{x},\mathbf{y})+k_p^2G_{p,j}(\mathbf{x},\mathbf{y})=-(2\mu+\lambda)^{-1}\dive_{\mathbf{x}}(\delta(\mathbf{x},\mathbf{y})\mathbf{e}_j),~~~&{\rm in}~\R^3,\\ 
	\displaystyle\curl_\mathbf{x}\curl_{\mathbf{x}}\mathbf{G}_{s,j}(\mathbf{x},\mathbf{y})-k_s^2\mathbf{G}_{s,j}(\mathbf{x},\mathbf{y})=\mu^{-1}\curl_{\mathbf{x}}(\delta(\mathbf{x},\mathbf{y})\mathbf{e}_j),~~~&{\rm in}~\R^3,\\
	\left[G_{p,j}(\mathbf{x},\mathbf{y})\right]=0,~\left[\boldsymbol{\nu}\times \mathbf{G}_{s,j}(\mathbf{x},\mathbf{y})\right]=0,~~~&{\rm on}~\Gamma_0, \\
	\displaystyle\left[k_p^{-2}\pa_{\boldsymbol{\nu}(\mathbf{x})}G_{p,j}(\mathbf{x},\mathbf{y})-k_s^{-2}\boldsymbol{\nu}\cdot\curl_{\mathbf{x}}\mathbf{G}_{s,j}(\mathbf{x},\mathbf{y})\right]=0,~~~&{\rm on}~\Gamma_0, \\
	\displaystyle\left[k_p^{-2}\boldsymbol{\nu}\times\grad_{\mathbf{x}}G_{p,j}(\mathbf{x},\mathbf{y})-k_s^{-2}\boldsymbol{\nu}\times\curl_{\mathbf{x}}\mathbf{G}_{s,j}(\mathbf{x},\mathbf{y})\right]=0,~~~&{\rm on}~\Gamma_0, \\
	\lim\limits_{r\rightarrow\infty}r\left(\pa_rG_{p,j}(\mathbf{x},\mathbf{y})-ik_{p,\pm}G_{p,j}(\mathbf{x},\mathbf{y})\right)=0,~~~&\mathbf{\hat x}\in\Sp^2_\pm, \\
	\lim\limits_{r\rightarrow\infty}r(\curl\mathbf{G}_{s,j}(\mathbf{x},\mathbf{y})\times\hat{\mathbf{x}}-ik_{s,\pm}\mathbf{G}_{s,j}(\mathbf{x},\mathbf{y}))=0,~~~&\mathbf{\hat x}\in\Sp^2_\pm. 
\end{array}
\right.
\en
Note that here $\mathbf{e}_j$ is the $j$-th standard orthonormal basis of $\R^3$ and $k_a:=k_{a,\pm}$ in $\R^3_\pm$. We will derive the explicit expressions for $G_{p,j}$ and $\mathbf{G}_{s,j}$, $j=1,2,3$, for which we also introduce some auxiliary functions: 
\be\label{3.11}
\left\{
\begin{array}{ll}
	\displaystyle\Delta_\mathbf{x}\wid G_{p,j}(\mathbf{x},\mathbf{y})+k_p^2\wid G_{p,j}(\mathbf{x},\mathbf{y})=-(2\mu+\lambda)^{-1}\dive_{\mathbf{x}}(\delta(\mathbf{x},\mathbf{y})\mathbf{e}_j),~~~&{\rm in}~\R^3,\\ 
	\displaystyle\curl_\mathbf{x}\curl_{\mathbf{x}}\mathbf{\wid G}_{s,j}(\mathbf{x},\mathbf{y})-k_s^2 \mathbf{\wid G}_{s,j}(\mathbf{x},\mathbf{y})=\mu^{-1}\curl_{\mathbf{x}}(\delta(\mathbf{x},\mathbf{y})\mathbf{e}_j),~~~&{\rm in}~\R^3,\\
	\left[\wid G_{p,j}(\mathbf{x},\mathbf{y})\right]=0,~\left[\boldsymbol{\nu}\times\mathbf{\wid G}_{s,j}(\mathbf{x},\mathbf{y})\right]=0,~~~&{\rm on}~\Gamma_0, \\
	\displaystyle\left[k_p^{-2}\pa_{\boldsymbol{\nu}(\mathbf{x})}\wid G_{p,j}(\mathbf{x},\mathbf{y})\right]=0,~\displaystyle\left[k_s^{-2}\boldsymbol{\nu}\times\curl_{\mathbf{x}}\mathbf{\wid G}_{s,j}(\mathbf{x},\mathbf{y})\right]=0,~~~&{\rm on}~\Gamma_0, \\ 
	\lim\limits_{r\rightarrow\infty}r\left(\pa_r\wid G_{p,j}(\mathbf{x},\mathbf{y})-ik_{p,\pm}\wid G_{p,j}(\mathbf{x},\mathbf{y})\right)=0,~~~&\mathbf{\hat x}\in\Sp^2_\pm, \\
	\lim\limits_{r\rightarrow\infty}r(\curl\mathbf{\wid G}_{s,j}(\mathbf{x},\mathbf{y})\times\hat{\mathbf{x}}-ik_{s,\pm}\mathbf{\wid G}_{s,j}(\mathbf{x},\mathbf{y}))=0,~~~&\mathbf{\hat x}\in\Sp^2_\pm. 
\end{array}
\right.
\en

Set $\boldsymbol{\zeta}=(\zeta_1,\zeta_2,0)$, $\mathbf{x}'=(x_1,x_2,0)$ and $\mathbf{y}'=(y_1,y_2,0)$. For a moment, replace $\xi^2$ by $|\boldsymbol{\zeta}|^2=\zeta_1^2+\zeta_2^2$ in the definition of $\beta_{a,\pm}$, $a=p,s$. As in the two-dimensional case, utilizing the Fourier transform and its inverse, solving the resulting systems of linear equations, we can obtain the explicit expressions of $\wid G_{p,j}$ and $\mathbf{\wid G}_{s,j}$. 

Still denote the fundamental solution for $\Delta+k_{a,\pm}^2$ in $\R^3$ by $\Phi_{k_{a,\pm}}(\mathbf{x},\mathbf{y}):=\exp(ik_{a,\pm}|\mathbf{x}-\mathbf{y}|)/(4\pi|\mathbf{x}-\mathbf{y}|)$. Then we have that 
\begin{align*}
&\wid G_{p,j}(\mathbf{x},\mathbf{y})=\frac{1}{2\mu+\lambda} \\
&\left\{
\begin{array}{ll}
	\displaystyle\frac{\pa\Phi_{k_{p,+}}(\mathbf{x},\mathbf{y})}{\pa x_j}+\frac{1}{(2\pi)^2}\int_{-\infty}^{+\infty}\int_{-\infty}^{+\infty}A_{p,j}^+\zeta_je^{-\beta_{p,+}x_3}e^{i\boldsymbol{\zeta}\cdot(\mathbf{x}'-\mathbf{y}')}d\zeta_1d\zeta_2,&\mathbf{x}\in\R^3_+,\;\mathbf{y}\in\R^3_+,\\
	\displaystyle\frac{1}{(2\pi)^2}\int_{-\infty}^{+\infty}\int_{-\infty}^{+\infty}A_{p,j}^-\zeta_je^{\beta_{p,-}x_3}e^{i\boldsymbol{\zeta}\cdot(\mathbf{x}'-\mathbf{y}')}d\zeta_1d\zeta_2,&\mathbf{x}\in\R^3_-,\;\mathbf{y}\in\R^3_+, \\
	\displaystyle\frac{1}{(2\pi)^2}\int_{-\infty}^{+\infty}\int_{-\infty}^{+\infty}B_{p,j}^+\zeta_je^{-\beta_{p,+}x_3}e^{i\boldsymbol{\zeta}\cdot(\mathbf{x}'-\mathbf{y}')}d\zeta_1d\zeta_2,&\mathbf{x}\in\R^3_+,\;\mathbf{y}\in\R^3_-, \\
	\displaystyle\frac{\pa\Phi_{k_{p,-}}(\mathbf{x},\mathbf{y})}{\pa x_j}+\frac{1}{(2\pi)^2}\int_{-\infty}^{+\infty}\int_{-\infty}^{+\infty}B_{p,j}^-\zeta_je^{\beta_{p,-}x_3}e^{i\boldsymbol{\zeta}\cdot(\mathbf{x}'-\mathbf{y}')}d\zeta_1d\zeta_2,&\mathbf{x}\in\R^3_-,\;\mathbf{y}\in\R^3_-, 
\end{array}
\right.
\end{align*}
for $j=1,2$, and 
\begin{align*}
&\wid G_{p,3}(\mathbf{x},\mathbf{y})=\frac{1}{2\mu+\lambda} \\
&\left\{
\begin{array}{ll}
	\displaystyle\frac{\pa\Phi_{k_{p,+}}(\mathbf{x},\mathbf{y})}{\pa x_3}+\frac{1}{(2\pi)^2}\int_{-\infty}^{+\infty}\int_{-\infty}^{+\infty}A_{p,3}^+e^{-\beta_{p,+}x_3}e^{i\boldsymbol{\zeta}\cdot(\mathbf{x}'-\mathbf{y}')}d\zeta_1d\zeta_2,&\mathbf{x}\in\R^3_+,\;\mathbf{y}\in\R^3_+,\\
	\displaystyle\frac{1}{(2\pi)^2}\int_{-\infty}^{+\infty}\int_{-\infty}^{+\infty}A_{p,3}^-e^{\beta_{p,-}x_3}e^{i\boldsymbol{\zeta}\cdot(\mathbf{x}'-\mathbf{y}')}d\zeta_1d\zeta_2,&\mathbf{x}\in\R^3_-,\;\mathbf{y}\in\R^3_+, \\
	\displaystyle\frac{1}{(2\pi)^2}\int_{-\infty}^{+\infty}\int_{-\infty}^{+\infty}B_{p,3}^+e^{-\beta_{p,+}x_3}e^{i\boldsymbol{\zeta}\cdot(\mathbf{x}'-\mathbf{y}')}d\zeta_1d\zeta_2,&\mathbf{x}\in\R^3_+,\;\mathbf{y}\in\R^3_-, \\
	\displaystyle\frac{\pa\Phi_{k_{p,-}}(\mathbf{x},\mathbf{y})}{\pa x_j}+\frac{1}{(2\pi)^2}\int_{-\infty}^{+\infty}\int_{-\infty}^{+\infty}B_{p,3}^-e^{\beta_{p,-}x_3}e^{i\boldsymbol{\zeta}\cdot(\mathbf{x}'-\mathbf{y}')}d\zeta_1d\zeta_2,&\mathbf{x}\in\R^3_-,\;\mathbf{y}\in\R^3_-, 
\end{array}
\right.
\end{align*}
where 
\begin{align*}
&A_{p,1}^+=A_{p,2}^+=\frac{i}{2\beta_{p,+}}R_p(\beta_{p,+},\beta_{p,-})e^{-\beta_{p,+}y_3}, \\ &A_{p,1}^-=A_{p,2}^-=\frac{i}{2\beta_{p,+}}T_p(\beta_{p,+},\beta_{p,-})e^{-\beta_{p,+}y_3}, \\
&B_{p,1}^+=B_{p,2}^+=-\frac{i}{2\beta_{p,-}}(T_p(\beta_{p,+},\beta_{p,-})-2)e^{\beta_{p,-}y_3}, \\
&B_{p,1}^-=B_{p,2}^-=-\frac{i}{2\beta_{p,-}}R_p(\beta_{p,+},\beta_{p,-})e^{\beta_{p,-}y_3}, 
\end{align*}
and 
\begin{align*}
&A_{p,3}^+=\frac{1}{2}R_p(\beta_{p,+},\beta_{p,-})e^{-\beta_{p,+}y_3}, \\
&A_{p,3}^-=\frac{1}{2}T_p(\beta_{p,+},\beta_{p,-})e^{\beta_{p,+}y_3}, \\
&B_{p,3}^+=\frac{1}{2}(T_p(\beta_{p,+},\beta_{p,-}-2))e^{\beta_{p,-}y_3}, \\
&B_{p,3}^-=\frac{1}{2}R_p(\beta_{p,+},\beta_{p,-})e^{\beta_{p,-}y_3}. 
\end{align*}

For $\mathbf{\wid G}_{s,j}$, we note that the corresponding Maxwell equation in \eqref{3.11} is equivalent to 
$$
  \Delta_{\mathbf{x}}\mathbf{\wid G}_{s,j}(\mathbf{x},\mathbf{y})+k_s^2\mathbf{\wid G}_{s,j}(\mathbf{x},\mathbf{y})=-\mu^{-1}\curl_{\mathbf{x}}(\delta(\mathbf{x},\mathbf{y})\mathbf{e}_j)
  ~\text{and}~\dive_{\mathbf{x}}\mathbf{\wid G}_{s,j}(\mathbf{x},\mathbf{y})=0. 
$$
Therefore, we derive that 
\begin{align*}
&\wid G_{s,1}^{(1)}(\mathbf{x},\mathbf{y})=\frac{1}{\mu} \\
&\left\{
\begin{array}{ll}
	\displaystyle\frac{1}{(2\pi)^2}\int_{-\infty}^{+\infty}\int_{-\infty}^{+\infty}A_{s,1}^{(1)}e^{-\beta_{s,+}x_3}\zeta_1\zeta_2e^{i\boldsymbol{\zeta}\cdot(\mathbf{x}'-\mathbf{y}')}d\zeta_1d\zeta_2,&\mathbf{x}\in\R^3_+,\;\mathbf{y}\in\R^3_+,\\
	\displaystyle\frac{1}{(2\pi)^2}\int_{-\infty}^{+\infty}\int_{-\infty}^{+\infty}A_{s,1}^{(1)}e^{\beta_{s,-}x_3}\zeta_1\zeta_2e^{i\boldsymbol{\zeta}\cdot(\mathbf{x}'-\mathbf{y}')}d\zeta_1d\zeta_2,&\mathbf{x}\in\R^3_-,\;\mathbf{y}\in\R^3_+, \\
	\displaystyle\frac{1}{(2\pi)^2}\int_{-\infty}^{+\infty}\int_{-\infty}^{+\infty}B_{s,1}^{(1)}e^{-\beta_{s,+}x_3}\zeta_1\zeta_2e^{i\boldsymbol{\zeta}\cdot(\mathbf{x}'-\mathbf{y}')}d\zeta_1d\zeta_2,&\mathbf{x}\in\R^3_+,\;\mathbf{y}\in\R^3_-, \\
	\displaystyle\frac{1}{(2\pi)^2}\int_{-\infty}^{+\infty}\int_{-\infty}^{+\infty}B_{s,1}^{(1)}e^{\beta_{s,-}x_3}\zeta_1\zeta_2e^{i\boldsymbol{\zeta}\cdot(\mathbf{x}'-\mathbf{y}')}d\zeta_1d\zeta_2,&\mathbf{x}\in\R^3_-,\;\mathbf{y}\in\R^3_-, 
\end{array}
\right.
\end{align*}
and 
\begin{align*}
&\wid G_{s,1}^{(2)}(\mathbf{x},\mathbf{y})=\frac{1}{\mu} \\
&\left\{
\begin{array}{ll}
	\displaystyle\frac{\pa\Phi_{k_{p,+}}(\mathbf{x},\mathbf{y})}{\pa x_3}+\frac{1}{(2\pi)^2}\int_{-\infty}^{+\infty}\int_{-\infty}^{+\infty}(A_{s,1}^{(2)+}+\hat A_{s,1}^{(2)}\zeta_1^2)e^{-\beta_{s,+}x_3}e^{i\boldsymbol{\zeta}\cdot(\mathbf{x}'-\mathbf{y}')}d\zeta_1d\zeta_2,&\mathbf{x}\in\R^3_+,\;\mathbf{y}\in\R^3_+,\\
	\displaystyle\frac{1}{(2\pi)^2}\int_{-\infty}^{+\infty}\int_{-\infty}^{+\infty}(A_{s,1}^{(2)-}+\hat A_{s,1}^{(2)}\zeta_1^2)e^{\beta_{s,-}x_3}e^{i\boldsymbol{\zeta}\cdot(\mathbf{x}'-\mathbf{y}')}d\zeta_1d\zeta_2,&\mathbf{x}\in\R^3_-,\;\mathbf{y}\in\R^3_+, \\
	\displaystyle\frac{1}{(2\pi)^2}\int_{-\infty}^{+\infty}\int_{-\infty}^{+\infty}(B_{s,1}^{(2)+}+\hat B_{s,1}^{(2)}\zeta_1^2)e^{-\beta_{s,+}x_3}e^{i\boldsymbol{\zeta}\cdot(\mathbf{x}'-\mathbf{y}')}d\zeta_1d\zeta_2,&\mathbf{x}\in\R^3_+,\;\mathbf{y}\in\R^3_-, \\
	\displaystyle\frac{\pa\Phi_{k_{p,-}}(\mathbf{x},\mathbf{y})}{\pa x_j}+\frac{1}{(2\pi)^2}\int_{-\infty}^{+\infty}\int_{-\infty}^{+\infty}(B_{s,1}^{(2)-}+\hat B_{s,1}^{(2)}\zeta_1^2)e^{\beta_{s,-}x_3}e^{i\boldsymbol{\zeta}\cdot(\mathbf{x}'-\mathbf{y}')}d\zeta_1d\zeta_2,&\mathbf{x}\in\R^3_-,\;\mathbf{y}\in\R^3_-, 
\end{array}
\right.
\end{align*}
and 
\begin{align*}
&\wid G_{s,1}^{(3)}(\mathbf{x},\mathbf{y})=-\frac{1}{\mu} \\
&\left\{
\begin{array}{ll}
	\displaystyle\frac{\pa\Phi_{k_{p,+}}(\mathbf{x},\mathbf{y})}{\pa x_2}+\frac{1}{(2\pi)^2}\int_{-\infty}^{+\infty}\int_{-\infty}^{+\infty}A_{s,1}^{(3)+}e^{-\beta_{s,+}x_3}\zeta_2e^{i\boldsymbol{\zeta}\cdot(\mathbf{x}'-\mathbf{y}')}d\zeta_1d\zeta_2,&\mathbf{x}\in\R^3_+,\;\mathbf{y}\in\R^3_+,\\
	\displaystyle\frac{1}{(2\pi)^2}\int_{-\infty}^{+\infty}\int_{-\infty}^{+\infty}A_{s,1}^{(3)-}e^{\beta_{s,-}x_3}\zeta_2e^{i\boldsymbol{\zeta}\cdot(\mathbf{x}'-\mathbf{y}')}d\zeta_1d\zeta_2,&\mathbf{x}\in\R^3_-,\;\mathbf{y}\in\R^3_+, \\
	\displaystyle\frac{1}{(2\pi)^2}\int_{-\infty}^{+\infty}\int_{-\infty}^{+\infty}B_{s,1}^{(3)+}e^{-\beta_{s,+}x_3}\zeta_2e^{i\boldsymbol{\zeta}\cdot(\mathbf{x}'-\mathbf{y}')}d\zeta_1d\zeta_2,&\mathbf{x}\in\R^3_+,\;\mathbf{y}\in\R^3_-, \\
	\displaystyle\frac{\pa\Phi_{k_{p,-}}(\mathbf{x},\mathbf{y})}{\pa x_2}+\frac{1}{(2\pi)^2}\int_{-\infty}^{+\infty}\int_{-\infty}^{+\infty}B_{s,1}^{(3)-}e^{\beta_{s,-}x_3}\zeta_2e^{i\boldsymbol{\zeta}\cdot(\mathbf{x}'-\mathbf{y}')}d\zeta_1d\zeta_2,&\mathbf{x}\in\R^3_-,\;\mathbf{y}\in\R^3_-, 
\end{array}
\right.
\end{align*}
where 
\begin{align*}
  &A_{s,1}^{(1)}=-\hat A_{s,1}^{(2)}=\frac{1}{k_{s,+}^{-2}\beta_{s,+}+k_{s,-}^{-2}\beta_{s,-}}\frac{k_{s,-}^{-2}-k_{s,+}^{-2}}{\beta_{s,+}+\beta_{s,-}}e^{-\beta_{s,+}y_3}, \\
  &B_{s,1}^{(1)}=-\hat B_{s,1}^{(2)}=\frac{1}{k_{s,+}^{-2}\beta_{s,+}+k_{s,-}^{-2}\beta_{s,-}}\frac{k_{s,-}^{-2}-k_{s,+}^{-2}}{\beta_{s,+}+\beta_{s,-}}e^{\beta_{s,-}y_3}, 
\end{align*}
and 
\begin{align*}
  &A_{s,1}^{(2)+}=i\beta_{s,+}A_{s,1}^{(3)+}=\left(\frac{1}{2}-\frac{\beta_{s,+}}{\beta_{s,+}+\beta_{s,-}}\right)e^{-\beta_{s,+}y_3}, \\
  &A_{s,1}^{(2)-}=-i\beta_{s,-}A_{s,1}^{(3)-}=\frac{\beta_{s,-}}{\beta_{s,+}+\beta_{s,-}}e^{-\beta_{s,+}y_3}, \\
  &B_{s,1}^{(2)+}=i\beta_{s,+}B_{s,1}^{(3)+}=-\frac{\beta_{s,+}}{\beta_{s,+}+\beta_{s,-}}e^{\beta_{s,-}y_3}, \\
  &B_{s,1}^{(2)-}=-i\beta_{s,-}B_{s,1}^{(3)-}=-\left(\frac{1}{2}-\frac{\beta_{s,-}}{\beta_{s,+}+\beta_{s,-}}\right)e^{\beta_{s,-}y_3}. 
\end{align*}

Further, it is deduced that 
\begin{align*}
&\wid G_{s,2}^{(1)}(\mathbf{x},\mathbf{y})=-\frac{1}{\mu} \\
&\left\{
\begin{array}{ll}
	\displaystyle\frac{\pa\Phi_{k_{p,+}}(\mathbf{x},\mathbf{y})}{\pa x_3}+\frac{1}{(2\pi)^2}\int_{-\infty}^{+\infty}\int_{-\infty}^{+\infty}(A_{s,2}^{(1)+}+\hat A_{s,2}^{(1)}\zeta_2^2)e^{-\beta_{s,+}x_3}e^{i\boldsymbol{\zeta}\cdot(\mathbf{x}'-\mathbf{y}')}d\zeta_1d\zeta_2,&\mathbf{x}\in\R^3_+,\;\mathbf{y}\in\R^3_+,\\
	\displaystyle\frac{1}{(2\pi)^2}\int_{-\infty}^{+\infty}\int_{-\infty}^{+\infty}(A_{s,2}^{(1)-}+\hat A_{s,2}^{(1)}\zeta_2^2)e^{\beta_{s,-}x_3}e^{i\boldsymbol{\zeta}\cdot(\mathbf{x}'-\mathbf{y}')}d\zeta_1d\zeta_2,&\mathbf{x}\in\R^3_-,\;\mathbf{y}\in\R^3_+, \\
	\displaystyle\frac{1}{(2\pi)^2}\int_{-\infty}^{+\infty}\int_{-\infty}^{+\infty}(B_{s,2}^{(1)+}+\hat B_{s,2}^{(1)}\zeta_2^2)e^{-\beta_{s,+}x_3}e^{i\boldsymbol{\zeta}\cdot(\mathbf{x}'-\mathbf{y}')}d\zeta_1d\zeta_2,&\mathbf{x}\in\R^3_+,\;\mathbf{y}\in\R^3_-, \\
	\displaystyle\frac{\pa\Phi_{k_{p,-}}(\mathbf{x},\mathbf{y})}{\pa x_j}+\frac{1}{(2\pi)^2}\int_{-\infty}^{+\infty}\int_{-\infty}^{+\infty}(B_{s,2}^{(1)-}+\hat B_{s,2}^{(1)}\zeta_2^2)e^{\beta_{s,-}x_3}e^{i\boldsymbol{\zeta}\cdot(\mathbf{x}'-\mathbf{y}')}d\zeta_1d\zeta_2,&\mathbf{x}\in\R^3_-,\;\mathbf{y}\in\R^3_-, 
\end{array}
\right.
\end{align*}
and 
\begin{align*}
&\wid G_{s,2}^{(2)}(\mathbf{x},\mathbf{y})=\frac{1}{\mu} \\
&\left\{
\begin{array}{ll}
	\displaystyle\frac{1}{(2\pi)^2}\int_{-\infty}^{+\infty}\int_{-\infty}^{+\infty}A_{s,2}^{(2)}e^{-\beta_{s,+}x_3}\zeta_1\zeta_2e^{i\boldsymbol{\zeta}\cdot(\mathbf{x}'-\mathbf{y}')}d\zeta_1d\zeta_2,&\mathbf{x}\in\R^3_+,\;\mathbf{y}\in\R^3_+,\\
	\displaystyle\frac{1}{(2\pi)^2}\int_{-\infty}^{+\infty}\int_{-\infty}^{+\infty}A_{s,2}^{(2)}e^{\beta_{s,-}x_3}\zeta_1\zeta_2e^{i\boldsymbol{\zeta}\cdot(\mathbf{x}'-\mathbf{y}')}d\zeta_1d\zeta_2,&\mathbf{x}\in\R^3_-,\;\mathbf{y}\in\R^3_+, \\
	\displaystyle\frac{1}{(2\pi)^2}\int_{-\infty}^{+\infty}\int_{-\infty}^{+\infty}B_{s,2}^{(2)}e^{-\beta_{s,+}x_3}\zeta_1\zeta_2e^{i\boldsymbol{\zeta}\cdot(\mathbf{x}'-\mathbf{y}')}d\zeta_1d\zeta_2,&\mathbf{x}\in\R^3_+,\;\mathbf{y}\in\R^3_-, \\
	\displaystyle\frac{1}{(2\pi)^2}\int_{-\infty}^{+\infty}\int_{-\infty}^{+\infty}B_{s,2}^{(2)}e^{\beta_{s,-}x_3}\zeta_1\zeta_2e^{i\boldsymbol{\zeta}\cdot(\mathbf{x}'-\mathbf{y}')}d\zeta_1d\zeta_2,&\mathbf{x}\in\R^3_-,\;\mathbf{y}\in\R^3_-, 
\end{array}
\right.
\end{align*}
and 
\begin{align*}
&\wid G_{s,2}^{(3)}(\mathbf{x},\mathbf{y})=\frac{1}{\mu} \\
&\left\{
\begin{array}{ll}
	\displaystyle\frac{\pa\Phi_{k_{p,+}}(\mathbf{x},\mathbf{y})}{\pa x_1}+\frac{1}{(2\pi)^2}\int_{-\infty}^{+\infty}\int_{-\infty}^{+\infty}A_{s,2}^{(3)+}e^{-\beta_{s,+}x_3}\zeta_1e^{i\boldsymbol{\zeta}\cdot(\mathbf{x}'-\mathbf{y}')}d\zeta_1d\zeta_2,&\mathbf{x}\in\R^3_+,\;\mathbf{y}\in\R^3_+,\\
	\displaystyle\frac{1}{(2\pi)^2}\int_{-\infty}^{+\infty}\int_{-\infty}^{+\infty}A_{s,2}^{(3)-}e^{\beta_{s,-}x_3}\zeta_1e^{i\boldsymbol{\zeta}\cdot(\mathbf{x}'-\mathbf{y}')}d\zeta_1d\zeta_2,&\mathbf{x}\in\R^3_-,\;\mathbf{y}\in\R^3_+, \\
	\displaystyle\frac{1}{(2\pi)^2}\int_{-\infty}^{+\infty}\int_{-\infty}^{+\infty}B_{s,2}^{(3)+}e^{-\beta_{s,+}x_3}\zeta_1e^{i\boldsymbol{\zeta}\cdot(\mathbf{x}'-\mathbf{y}')}d\zeta_1d\zeta_2,&\mathbf{x}\in\R^3_+,\;\mathbf{y}\in\R^3_-, \\
	\displaystyle\frac{\pa\Phi_{k_{p,-}}(\mathbf{x},\mathbf{y})}{\pa x_1}+\frac{1}{(2\pi)^2}\int_{-\infty}^{+\infty}\int_{-\infty}^{+\infty}B_{s,2}^{(3)-}e^{\beta_{s,-}x_3}\zeta_1e^{i\boldsymbol{\zeta}\cdot(\mathbf{x}'-\mathbf{y}')}d\zeta_1d\zeta_2,&\mathbf{x}\in\R^3_-,\;\mathbf{y}\in\R^3_-, 
\end{array}
\right.
\end{align*}
where 
\begin{align*}
  &A_{s,2}^{(1)+}=i\beta_{s,+}A_{s,2}^{(3)+}=A_{s,1}^{(2)+}, \\
  &\hat A_{s,2}^{(1)}=A_{s,2}^{(2)}=\hat A_{s,1}^{(2)}, \\
  &A_{s,2}^{(1)-}=-i\beta_{s,-}A_{s,2}^{(3)-}=A_{s,1}^{(2)-}, 
\end{align*}
and 
\begin{align*}
  &B_{s,2}^{(1)+}=i\beta_{s,+}B_{s,2}^{(3)+}=B_{s,1}^{(2)+}, \\
  &\hat B_{s,2}^{(1)}=B_{s,2}^{(2)}=\hat B_{s,1}^{(2)}, \\
  &B_{s,2}^{(1)-}=-i\beta_{s,-}B_{s,2}^{(3)-}=B_{s,1}^{(2)-}, 
\end{align*}

Moreover, we obtain the following 
\begin{align*}
&\wid G_{s,3}^{(1)}(\mathbf{x},\mathbf{y})=\frac{1}{\mu} \\
&\left\{
\begin{array}{ll}
	\displaystyle\frac{\pa\Phi_{k_{s,+}}(\mathbf{x},\mathbf{y})}{\pa x_2}+\frac{1}{(2\pi)^2}\int_{-\infty}^{+\infty}\int_{-\infty}^{+\infty}A_{s,3}^{(1)+}e^{-\beta_{s,+}x_3}\zeta_2e^{i\boldsymbol{\zeta}\cdot(\mathbf{x}'-\mathbf{y}')}d\zeta_1d\zeta_2,&\mathbf{x}\in\R^3_+,\;\mathbf{y}\in\R^3_+,\\
	\displaystyle\frac{1}{(2\pi)^2}\int_{-\infty}^{+\infty}\int_{-\infty}^{+\infty}A_{s,3}^{(1)-}e^{\beta_{s,-}x_3}\zeta_2e^{i\boldsymbol{\zeta}\cdot(\mathbf{x}'-\mathbf{y}')}d\zeta_1d\zeta_2,&\mathbf{x}\in\R^3_-,\;\mathbf{y}\in\R^3_+, \\
	\displaystyle\frac{1}{(2\pi)^2}\int_{-\infty}^{+\infty}\int_{-\infty}^{+\infty}B_{s,3}^{(1)+}e^{-\beta_{s,+}x_3}\zeta_2e^{i\boldsymbol{\zeta}\cdot(\mathbf{x}'-\mathbf{y}')}d\zeta_1d\zeta_2,&\mathbf{x}\in\R^3_+,\;\mathbf{y}\in\R^3_-, \\
	\displaystyle\frac{\pa\Phi_{k_{s,-}}(\mathbf{x},\mathbf{y})}{\pa x_2}+\frac{1}{(2\pi)^2}\int_{-\infty}^{+\infty}\int_{-\infty}^{+\infty}B_{s,3}^{(1)-}e^{\beta_{s,-}x_3}\zeta_2e^{i\boldsymbol{\zeta}\cdot(\mathbf{x}'-\mathbf{y}')}d\zeta_1d\zeta_2,&\mathbf{x}\in\R^3_-,\;\mathbf{y}\in\R^3_-, 
\end{array}
\right.
\end{align*}
and 
\begin{align*}
&\wid G_{s,3}^{(2)}(\mathbf{x},\mathbf{y})=-\frac{1}{\mu} \\
&\left\{
\begin{array}{ll}
	\displaystyle\frac{\pa\Phi_{k_{s,+}}(\mathbf{x},\mathbf{y})}{\pa x_1}+\frac{1}{(2\pi)^2}\int_{-\infty}^{+\infty}\int_{-\infty}^{+\infty}A_{s,3}^{(2)+}e^{-\beta_{s,+}x_3}\zeta_1e^{i\boldsymbol{\zeta}\cdot(\mathbf{x}'-\mathbf{y}')}d\zeta_1d\zeta_2,&\mathbf{x}\in\R^3_+,\;\mathbf{y}\in\R^3_+,\\
	\displaystyle\frac{1}{(2\pi)^2}\int_{-\infty}^{+\infty}\int_{-\infty}^{+\infty}A_{s,3}^{(2)-}e^{\beta_{s,-}x_3}\zeta_1e^{i\boldsymbol{\zeta}\cdot(\mathbf{x}'-\mathbf{y}')}d\zeta_1d\zeta_2,&\mathbf{x}\in\R^3_-,\;\mathbf{y}\in\R^3_+, \\
	\displaystyle\frac{1}{(2\pi)^2}\int_{-\infty}^{+\infty}\int_{-\infty}^{+\infty}B_{s,3}^{(2)+}e^{-\beta_{s,+}x_3}\zeta_1e^{i\boldsymbol{\zeta}\cdot(\mathbf{x}'-\mathbf{y}')}d\zeta_1d\zeta_2,&\mathbf{x}\in\R^3_+,\;\mathbf{y}\in\R^3_-, \\
	\displaystyle\frac{\pa\Phi_{k_{s,-}}(\mathbf{x},\mathbf{y})}{\pa x_1}+\frac{1}{(2\pi)^2}\int_{-\infty}^{+\infty}\int_{-\infty}^{+\infty}B_{s,3}^{(2)-}e^{\beta_{s,-}x_3}\zeta_1e^{i\boldsymbol{\zeta}\cdot(\mathbf{x}'-\mathbf{y}')}d\zeta_1d\zeta_2,&\mathbf{x}\in\R^3_-,\;\mathbf{y}\in\R^3_-, 
\end{array}
\right.
\end{align*}
and $\wid G_{s,3}^{(3)}(\mathbf{x},\mathbf{y})=0$, where 
\begin{align*}
&A_{s,3}^{(1)+}=A_{s,3}^{(2)+}=\frac{i}{2\beta_{s,+}}R_s(\beta_{s,+},\beta_{s,-})e^{-\beta_{s,+}y_3}, \\ &A_{s,3}^{(1)-}=A_{s,3}^{(2)-}=\frac{i}{2\beta_{s,+}}T_s(\beta_{s,+},\beta_{s,-})e^{-\beta_{s,+}y_3}, \\
&B_{s,3}^{(1)+}=B_{s,3}^{(2)+}=-\frac{i}{2\beta_{s,-}}(T_s(\beta_{s,+},\beta_{s,-})-2)e^{\beta_{s,-}y_3}, \\
&B_{s,3}^{(1)-}=B_{s,3}^{(2)-}=-\frac{i}{2\beta_{s,-}}R_s(\beta_{s,+},\beta_{s,-})e^{\beta_{s,-}y_3}. 
\end{align*}

For $U_{p,j}:=G_{p,j}-\wid G_{p,j}$ and $\mathbf{U}_{s,j}:=\mathbf{G}_{s,j}-\mathbf{\wid G}_{s,j}$ with $j=1,2,3$, it is easily verified that 
\ben
\left\{
\begin{array}{ll}
	\displaystyle\Delta_\mathbf{x}U_{p,j}(\mathbf{x},\mathbf{y})+k_p^2U_{p,j}(\mathbf{x},\mathbf{y})=0,&{\rm in}\;\R^3,~ \\ 
	\displaystyle\curl_\mathbf{x}\curl_{\mathbf{x}}\mathbf{U}_{s,j}(\mathbf{x},\mathbf{y})-k_s^2\mathbf{U}_{s,j}(\mathbf{x},\mathbf{y})=0,&{\rm in}\;\R^3,~ \\
	\left[U_{p,j}(\mathbf{x},\mathbf{y})\right]=0,~\left[\boldsymbol{\nu}\times \mathbf{U}_{s,j}(\mathbf{x},\mathbf{y})\right]=0,&{\rm on}\;\Gamma_0,~ \\
	\displaystyle\left[k_p^{-2}\pa_{\boldsymbol{\nu}(\mathbf{x})}U_{p,j}(\mathbf{x},\mathbf{y})-k_s^{-2}\boldsymbol{\nu}\cdot\curl_{\mathbf{x}}\mathbf{U}_{s,j}(\mathbf{x},\mathbf{y})\right]=f_j,&{\rm on}\;\Gamma_0,~ \\
	\displaystyle\left[-k_p^{-2}\boldsymbol{\nu}\times\grad_{\mathbf{x}}U_{p,j}(\mathbf{x},\mathbf{y})+k_s^{-2}\boldsymbol{\nu}\times\curl_{\mathbf{x}}\mathbf{U}_{s,j}(\mathbf{x},\mathbf{y})\right]=\mathbf{F}_j,&{\rm on}\;\Gamma_0,~ \\
	\lim\limits_{r\rightarrow\infty}r\left(\pa_rU_{p,j}(\mathbf{x},\mathbf{y})-ik_{p,\pm}U_{p,j}(\mathbf{x},\mathbf{y})\right)=0,~~~&\mathbf{\hat x}\in\Sp^2_\pm, \\
	\lim\limits_{r\rightarrow\infty}r(\curl\mathbf{U}_{s,j}(\mathbf{x},\mathbf{y})\times\hat{\mathbf{x}}-ik_{s,\pm}\mathbf{U}_{s,j}(\mathbf{x},\mathbf{y}))=0,~~~&\mathbf{\hat x}\in\Sp^2_\pm, 
\end{array}
\right.
\enn
where $f_j=[k_{s}^{-2}\boldsymbol{\nu}\cdot\curl_{\mathbf{x}}\mathbf{\wid G}_{s,j}(\mathbf{x},\mathbf{y})]$ and $\mathbf{F}_j=[k_{p}^{-2}\boldsymbol{\nu}\times\grad_{\mathbf{x}}\wid G_{p,j}(\mathbf{x},\mathbf{y})]$. Again, applying the Fourier transform and using the forms of $\wid G_{p,j}$ and $\mathbf{\wid G}_{s,j}$, by long but straightforward calculations, we can get the expressions of $U_{p,j}$ and $\mathbf{U}_{s,j}$. 

In particular, we have for $j=1,2$, 
\begin{align*}
&U_{p,j}(\mathbf{x},\mathbf{y})=\\
&\left\{
\begin{array}{ll}
	\displaystyle\frac{1}{(2\pi)^2}\int_{-\infty}^{+\infty}\int_{-\infty}^{+\infty}R_{p,j}e^{-\beta_{p,+}x_3}\zeta_je^{i\boldsymbol{\zeta}\cdot(\mathbf{x}'-\mathbf{y}')}d\zeta_1d\zeta_2,&\mathbf{x}\in\R^3_+,\;\mathbf{y}\in\R^3_+,\\
	\displaystyle\frac{1}{(2\pi)^2}\int_{-\infty}^{+\infty}\int_{-\infty}^{+\infty}R_{p,j}e^{\beta_{p,-}x_3}\zeta_je^{i\boldsymbol{\zeta}\cdot(\mathbf{x}'-\mathbf{y}')}d\zeta_1d\zeta_2,&\mathbf{x}\in\R^3_-,\;\mathbf{y}\in\R^3_+, \\
	\displaystyle\frac{1}{(2\pi)^2}\int_{-\infty}^{+\infty}\int_{-\infty}^{+\infty}T_{p,j}e^{-\beta_{p,+}x_3}\zeta_je^{i\boldsymbol{\zeta}\cdot(\mathbf{x}'-\mathbf{y}')}d\zeta_1d\zeta_2,&\mathbf{x}\in\R^3_+,\;\mathbf{y}\in\R^3_-, \\
	\displaystyle\frac{1}{(2\pi)^2}\int_{-\infty}^{+\infty}\int_{-\infty}^{+\infty}T_{p,j}e^{\beta_{p,-}x_3}\zeta_je^{i\boldsymbol{\zeta}\cdot(\mathbf{x}'-\mathbf{y}')}d\zeta_1d\zeta_2,&\mathbf{x}\in\R^3_-,\;\mathbf{y}\in\R^3_-, 
\end{array}
\right.
\end{align*}
and 
\begin{align*}
&U_{p,3}(\mathbf{x},\mathbf{y})=\\
&\left\{
\begin{array}{ll}
	\displaystyle\frac{1}{(2\pi)^2}\int_{-\infty}^{+\infty}\int_{-\infty}^{+\infty}R_{p,3}e^{-\beta_{p,+}x_3}e^{i\boldsymbol{\zeta}\cdot(\mathbf{x}'-\mathbf{y}')}d\zeta_1d\zeta_2,&\mathbf{x}\in\R^3_+,\;\mathbf{y}\in\R^3_+,\\
	\displaystyle\frac{1}{(2\pi)^2}\int_{-\infty}^{+\infty}\int_{-\infty}^{+\infty}R_{p,3}e^{\beta_{p,-}x_3}e^{i\boldsymbol{\zeta}\cdot(\mathbf{x}'-\mathbf{y}')}d\zeta_1d\zeta_2,&\mathbf{x}\in\R^3_-,\;\mathbf{y}\in\R^3_+, \\
	\displaystyle\frac{1}{(2\pi)^2}\int_{-\infty}^{+\infty}\int_{-\infty}^{+\infty}T_{p,3}e^{-\beta_{p,+}x_3}e^{i\boldsymbol{\zeta}\cdot(\mathbf{x}'-\mathbf{y}')}d\zeta_1d\zeta_2,&\mathbf{x}\in\R^3_+,\;\mathbf{y}\in\R^3_-, \\
	\displaystyle\frac{1}{(2\pi)^2}\int_{-\infty}^{+\infty}\int_{-\infty}^{+\infty}T_{p,3}e^{\beta_{p,-}x_3}e^{i\boldsymbol{\zeta}\cdot(\mathbf{x}'-\mathbf{y}')}d\zeta_1d\zeta_2,&\mathbf{x}\in\R^3_-,\;\mathbf{y}\in\R^3_-, 
\end{array}
\right.
\end{align*}
where 
\begin{align*}
  &R_{p,1}=R_{p,2}=\frac{C_0}{D}i\left[\frac{(k_{s,+}^{-2}\beta_{s,+}+k_{s,-}^{-2}\beta_{s,-})\beta_{s,-}+(k_{s,+}^{-2}-k_{s,-}^{-2})|\boldsymbol{\zeta}|^2}{\beta_{s,+}+\beta_{s,-}}e^{-\beta_{s,+}y_3}\right. \\
  &\qquad\qquad\qquad\qquad\quad\left.-\frac{(k_{s,+}^{-2}-k_{s,-}^{-2})k_{p,+}^{-2}|\boldsymbol{\zeta}|^2}{k_{p,+}^{-2}\beta_{p,+}+k_{p,-}^{-2}\beta_{p,-}}e^{-\beta_{p,+}y_3}\right], \\
  &T_{p,1}=T_{p,2}=\frac{C_0}{D}i\left[\frac{-(k_{s,+}^{-2}\beta_{s,+}+k_{s,-}^{-2}\beta_{s,-})\beta_{s,+}+(k_{s,+}^{-2}-k_{s,-}^{-2})|\boldsymbol{\zeta}|^2}{\beta_{s,+}+\beta_{s,-}}e^{\beta_{s,-}y_3}\right. \\
  &\qquad\qquad\qquad\qquad\quad\left.-\frac{(k_{s,+}^{-2}-k_{s,-}^{-2})k_{p,-}^{-2}|\boldsymbol{\zeta}|^2}{k_{p,+}^{-2}\beta_{p,+}+k_{p,-}^{-2}\beta_{p,-}}e^{\beta_{p,-}y_3}\right], \\
  &R_{p,3}=\frac{C_0|\boldsymbol{\zeta}|^2}{D}\left[k_{s,+}^{-2}e^{-\beta_{s,+}y_3}-\frac{(k_{s,+}^{-2}-k_{s,-}^{-2})k_{p,+}^{-2}\beta_{p,+}}{k_{p,+}^{-2}\beta_{p,+}+k_{p,-}^{-2}\beta_{p,-}}e^{-\beta_{p,+}y_3}\right], \\
  &T_{p,3}=\frac{C_0|\boldsymbol{\zeta}|^2}{D}\left[k_{s,-}^{-2}e^{\beta_{s,-}y_3}+\frac{(k_{s,+}^{-2}-k_{s,-}^{-2})k_{p,-}^{-2}\beta_{p,-}}{k_{p,+}^{-2}\beta_{p,+}+k_{p,-}^{-2}\beta_{p,-}}e^{\beta_{p,-}y_3}\right]. 
\end{align*}
Note that $C_0$ and $D$ here are the same as in the two-dimensional case. 

For $\mathbf{U}_{s,j}$, we obtain that  
\begin{align*}
&U_{s,1}^{(1)}(\mathbf{x},\mathbf{y})=\\
&\left\{
\begin{array}{ll}
	\displaystyle\frac{1}{(2\pi)^2}\int_{-\infty}^{+\infty}\int_{-\infty}^{+\infty}R_{s,1}^{(1)}e^{-\beta_{s,+}x_3}\zeta_1\zeta_2e^{i\boldsymbol{\zeta}\cdot(\mathbf{x}'-\mathbf{y}')}d\zeta_1d\zeta_2,&\mathbf{x}\in\R^3_+,\;\mathbf{y}\in\R^3_+,\\
	\displaystyle\frac{1}{(2\pi)^2}\int_{-\infty}^{+\infty}\int_{-\infty}^{+\infty}R_{s,1}^{(1)}e^{\beta_{s,-}x_3}\zeta_1\zeta_2e^{i\boldsymbol{\zeta}\cdot(\mathbf{x}'-\mathbf{y}')}d\zeta_1d\zeta_2,&\mathbf{x}\in\R^3_-,\;\mathbf{y}\in\R^3_+, \\
	\displaystyle\frac{1}{(2\pi)^2}\int_{-\infty}^{+\infty}\int_{-\infty}^{+\infty}T_{s,1}^{(1)}e^{-\beta_{s,+}x_3}\zeta_1\zeta_2e^{i\boldsymbol{\zeta}\cdot(\mathbf{x}'-\mathbf{y}')}d\zeta_1d\zeta_2,&\mathbf{x}\in\R^3_+,\;\mathbf{y}\in\R^3_-, \\
	\displaystyle\frac{1}{(2\pi)^2}\int_{-\infty}^{+\infty}\int_{-\infty}^{+\infty}T_{s,1}^{(1)}e^{\beta_{s,-}x_3}\zeta_1\zeta_2e^{i\boldsymbol{\zeta}\cdot(\mathbf{x}'-\mathbf{y}')}d\zeta_1d\zeta_2,&\mathbf{x}\in\R^3_-,\;\mathbf{y}\in\R^3_-, 
\end{array}
\right.
\end{align*}
and 
\begin{align*}
&U_{s,1}^{(2)}(\mathbf{x},\mathbf{y})=\\
&\left\{
\begin{array}{ll}
	\displaystyle\frac{1}{(2\pi)^2}\int_{-\infty}^{+\infty}\int_{-\infty}^{+\infty}R_{s,1}^{(2)}e^{-\beta_{s,+}x_3}\zeta_1^2e^{i\boldsymbol{\zeta}\cdot(\mathbf{x}'-\mathbf{y}')}d\zeta_1d\zeta_2,&\mathbf{x}\in\R^3_+,\;\mathbf{y}\in\R^3_+,\\
	\displaystyle\frac{1}{(2\pi)^2}\int_{-\infty}^{+\infty}\int_{-\infty}^{+\infty}R_{s,1}^{(2)}e^{\beta_{s,-}x_3}\zeta_1^2e^{i\boldsymbol{\zeta}\cdot(\mathbf{x}'-\mathbf{y}')}d\zeta_1d\zeta_2,&\mathbf{x}\in\R^3_-,\;\mathbf{y}\in\R^3_+, \\
	\displaystyle\frac{1}{(2\pi)^2}\int_{-\infty}^{+\infty}\int_{-\infty}^{+\infty}T_{s,1}^{(2)}e^{-\beta_{s,+}x_3}\zeta_1^2e^{i\boldsymbol{\zeta}\cdot(\mathbf{x}'-\mathbf{y}')}d\zeta_1d\zeta_2,&\mathbf{x}\in\R^3_+,\;\mathbf{y}\in\R^3_-, \\
	\displaystyle\frac{1}{(2\pi)^2}\int_{-\infty}^{+\infty}\int_{-\infty}^{+\infty}T_{s,1}^{(2)}e^{\beta_{s,-}x_3}\zeta_1^2e^{i\boldsymbol{\zeta}\cdot(\mathbf{x}'-\mathbf{y}')}d\zeta_1d\zeta_2,&\mathbf{x}\in\R^3_-,\;\mathbf{y}\in\R^3_-, 
\end{array}
\right.
\end{align*}
and $U_{s,1}^{(3)}(\mathbf{x},\mathbf{y})=0$, where 
\ben
  \begin{aligned}
  	R_{s,1}^{(1)}=-R_{s,1}^{(2)}=&\;\frac{C_0}{D}\left[(k_{p,+}^{-2}-k_{p,-}^{-2})\frac{(k_{s,+}^{-2}\beta_{s,+}+k_{s,-}^{-2}\beta_{s,-})\beta_{s,-}+(k_{s,+}^{-2}-k_{s,-}^{-2})|\boldsymbol{\zeta}|^2}{(k_{s,+}^{-2}\beta_{s,+}+k_{s,-}^{-2}\beta_{s,-})(\beta_{s,+}+\beta_{s,-})}\right. \\
  &\qquad\left.e^{-\beta_{s,+}y_3}-k_{p,+}^{-2}e^{-\beta_{p,+}y_3}\right], \\
  T_{s,1}^{(1)}=-T_{s,1}^{(2)}=&\;\frac{C_0}{D}\left[-(k_{p,+}^{-2}-k_{p,-}^{-2})\frac{(k_{s,+}^{-2}\beta_{s,+}+k_{s,-}^{-2}\beta_{s,-})\beta_{s,+}-(k_{s,+}^{-2}-k_{s,-}^{-2})|\boldsymbol{\zeta}|^2}{(k_{s,+}^{-2}\beta_{s,+}+k_{s,-}^{-2}\beta_{s,-})(\beta_{s,+}+\beta_{s,-})}\right. \\
  &\qquad\left.e^{\beta_{s,-}y_3}-k_{p,-}^{-2}e^{\beta_{p,-}y_3}\right]. 
  \end{aligned}
\enn

Further, we have 
\begin{align*}
&U_{s,2}^{(1)}(\mathbf{x},\mathbf{y})=\\
&\left\{
\begin{array}{ll}
	\displaystyle\frac{1}{(2\pi)^2}\int_{-\infty}^{+\infty}\int_{-\infty}^{+\infty}R_{s,2}^{(1)}e^{-\beta_{s,+}x_3}\zeta_2^2e^{i\boldsymbol{\zeta}\cdot(\mathbf{x}'-\mathbf{y}')}d\zeta_1d\zeta_2,&\mathbf{x}\in\R^3_+,\;\mathbf{y}\in\R^3_+,\\
	\displaystyle\frac{1}{(2\pi)^2}\int_{-\infty}^{+\infty}\int_{-\infty}^{+\infty}R_{s,2}^{(1)}e^{\beta_{s,-}x_3}\zeta_2^2e^{i\boldsymbol{\zeta}\cdot(\mathbf{x}'-\mathbf{y}')}d\zeta_1d\zeta_2,&\mathbf{x}\in\R^3_-,\;\mathbf{y}\in\R^3_+, \\
	\displaystyle\frac{1}{(2\pi)^2}\int_{-\infty}^{+\infty}\int_{-\infty}^{+\infty}T_{s,2}^{(1)}e^{-\beta_{s,+}x_3}\zeta_2^2e^{i\boldsymbol{\zeta}\cdot(\mathbf{x}'-\mathbf{y}')}d\zeta_1d\zeta_2,&\mathbf{x}\in\R^3_+,\;\mathbf{y}\in\R^3_-, \\
	\displaystyle\frac{1}{(2\pi)^2}\int_{-\infty}^{+\infty}\int_{-\infty}^{+\infty}T_{s,2}^{(1)}e^{\beta_{s,-}x_3}\zeta_2^2e^{i\boldsymbol{\zeta}\cdot(\mathbf{x}'-\mathbf{y}')}d\zeta_1d\zeta_2,&\mathbf{x}\in\R^3_-,\;\mathbf{y}\in\R^3_-, 
\end{array}
\right.
\end{align*}
and 
\begin{align*}
&U_{s,2}^{(2)}(\mathbf{x},\mathbf{y})=\\
&\left\{
\begin{array}{ll}
	\displaystyle\frac{1}{(2\pi)^2}\int_{-\infty}^{+\infty}\int_{-\infty}^{+\infty}R_{s,2}^{(2)}e^{-\beta_{s,+}x_3}\zeta_1\zeta_2e^{i\boldsymbol{\zeta}\cdot(\mathbf{x}'-\mathbf{y}')}d\zeta_1d\zeta_2,&\mathbf{x}\in\R^3_+,\;\mathbf{y}\in\R^3_+,\\
	\displaystyle\frac{1}{(2\pi)^2}\int_{-\infty}^{+\infty}\int_{-\infty}^{+\infty}R_{s,2}^{(2)}e^{\beta_{s,-}x_3}\zeta_1\zeta_2e^{i\boldsymbol{\zeta}\cdot(\mathbf{x}'-\mathbf{y}')}d\zeta_1d\zeta_2,&\mathbf{x}\in\R^3_-,\;\mathbf{y}\in\R^3_+, \\
	\displaystyle\frac{1}{(2\pi)^2}\int_{-\infty}^{+\infty}\int_{-\infty}^{+\infty}T_{s,2}^{(2)}e^{-\beta_{s,+}x_3}\zeta_1\zeta_2e^{i\boldsymbol{\zeta}\cdot(\mathbf{x}'-\mathbf{y}')}d\zeta_1d\zeta_2,&\mathbf{x}\in\R^3_+,\;\mathbf{y}\in\R^3_-, \\
	\displaystyle\frac{1}{(2\pi)^2}\int_{-\infty}^{+\infty}\int_{-\infty}^{+\infty}T_{s,2}^{(2)}e^{\beta_{s,-}x_3}\zeta_1\zeta_2e^{i\boldsymbol{\zeta}\cdot(\mathbf{x}'-\mathbf{y}')}d\zeta_1d\zeta_2,&\mathbf{x}\in\R^3_-,\;\mathbf{y}\in\R^3_-, 
\end{array}
\right.
\end{align*}
and $U_{s,2}^{(3)}(\mathbf{x},\mathbf{y})=0$, where 
$
  R_{s,2}^{(l)}=R_{s,1}^{(l)},~T_{s,2}^{(l)}=T_{s,1}^{(l)} 
$
for $l=1,2$. 

Finally, it is yielded that 
\begin{align*}
&U_{s,3}^{(1)}(\mathbf{x},\mathbf{y})=\\
&\left\{
\begin{array}{ll}
	\displaystyle\frac{1}{(2\pi)^2}\int_{-\infty}^{+\infty}\int_{-\infty}^{+\infty}R_{s,3}^{(1)}e^{-\beta_{s,+}x_3}\zeta_2e^{i\boldsymbol{\zeta}\cdot(\mathbf{x}'-\mathbf{y}')}d\zeta_1d\zeta_2,&\mathbf{x}\in\R^3_+,\;\mathbf{y}\in\R^3_+,\\
	\displaystyle\frac{1}{(2\pi)^2}\int_{-\infty}^{+\infty}\int_{-\infty}^{+\infty}R_{s,3}^{(1)}e^{\beta_{s,-}x_3}\zeta_2e^{i\boldsymbol{\zeta}\cdot(\mathbf{x}'-\mathbf{y}')}d\zeta_1d\zeta_2,&\mathbf{x}\in\R^3_-,\;\mathbf{y}\in\R^3_+, \\
	\displaystyle\frac{1}{(2\pi)^2}\int_{-\infty}^{+\infty}\int_{-\infty}^{+\infty}T_{s,3}^{(1)}e^{-\beta_{s,+}x_3}\zeta_2e^{i\boldsymbol{\zeta}\cdot(\mathbf{x}'-\mathbf{y}')}d\zeta_1d\zeta_2,&\mathbf{x}\in\R^3_+,\;\mathbf{y}\in\R^3_-, \\
	\displaystyle\frac{1}{(2\pi)^2}\int_{-\infty}^{+\infty}\int_{-\infty}^{+\infty}T_{s,3}^{(1)}e^{\beta_{s,-}x_3}\zeta_2e^{i\boldsymbol{\zeta}\cdot(\mathbf{x}'-\mathbf{y}')}d\zeta_1d\zeta_2,&\mathbf{x}\in\R^3_-,\;\mathbf{y}\in\R^3_-, 
\end{array}
\right.
\end{align*}
and 
\begin{align*}
&U_{s,3}^{(2)}(\mathbf{x},\mathbf{y})=\\
&\left\{
\begin{array}{ll}
	\displaystyle\frac{1}{(2\pi)^2}\int_{-\infty}^{+\infty}\int_{-\infty}^{+\infty}R_{s,3}^{(2)}e^{-\beta_{s,+}x_3}\zeta_1e^{i\boldsymbol{\zeta}\cdot(\mathbf{x}'-\mathbf{y}')}d\zeta_1d\zeta_2,&\mathbf{x}\in\R^3_+,\;\mathbf{y}\in\R^3_+,\\
	\displaystyle\frac{1}{(2\pi)^2}\int_{-\infty}^{+\infty}\int_{-\infty}^{+\infty}R_{s,3}^{(2)}e^{\beta_{s,-}x_3}\zeta_1e^{i\boldsymbol{\zeta}\cdot(\mathbf{x}'-\mathbf{y}')}d\zeta_1d\zeta_2,&\mathbf{x}\in\R^3_-,\;\mathbf{y}\in\R^3_+, \\
	\displaystyle\frac{1}{(2\pi)^2}\int_{-\infty}^{+\infty}\int_{-\infty}^{+\infty}T_{s,3}^{(2)}e^{-\beta_{s,+}x_3}\zeta_1e^{i\boldsymbol{\zeta}\cdot(\mathbf{x}'-\mathbf{y}')}d\zeta_1d\zeta_2,&\mathbf{x}\in\R^3_+,\;\mathbf{y}\in\R^3_-, \\
	\displaystyle\frac{1}{(2\pi)^2}\int_{-\infty}^{+\infty}\int_{-\infty}^{+\infty}T_{s,3}^{(2)}e^{\beta_{s,-}x_3}\zeta_1e^{i\boldsymbol{\zeta}\cdot(\mathbf{x}'-\mathbf{y}')}d\zeta_1d\zeta_2,&\mathbf{x}\in\R^3_-,\;\mathbf{y}\in\R^3_-, 
\end{array}
\right.
\end{align*}
and $U_{s,3}^{(3)}(\mathbf{x},\mathbf{y})=0$, where 
\ben
  \begin{aligned}
  	R_{s,3}^{(1)}=-R_{s,3}^{(2)}=&\;\frac{C_0}{D}i\left[-\frac{(k_{p,+}^{-2}-k_{p,-}^{-2})k_{s,+}^{-2}|\boldsymbol{\zeta}|^2}{k_{s,+}^{-2}\beta_{s,+}+k_{s,-}^{-2}\beta_{s,-}}e^{-\beta_{s,+}y_3}+k_{p,+}^{-2}\beta_{p,+}e^{-\beta_{p,+}y_3}\right], \\
  T_{s,3}^{(1)}=-T_{s,3}^{(2)}=&\;\frac{C_0}{D}i\left[-\frac{(k_{p,+}^{-2}-k_{p,-}^{-2})k_{s,-}^{-2}|\boldsymbol{\zeta}|^2}{k_{s,+}^{-2}\beta_{s,+}+k_{s,-}^{-2}\beta_{s,-}}e^{\beta_{s,-}y_3}-k_{p,-}^{-2}\beta_{p,-}e^{\beta_{p,-}y_3}\right]. 
  \end{aligned}
\enn

Next we aim to verify that all the expressions above satisfy the radiation condition. To this end, we change the corresponding Fourier transform to the Hankel transform. Denote by $J_m$ and $H_m^{(1)}$ the Bessel and Hankel function of the first kind of order $m$. 
\begin{lemma}\label{lem3.5}
	For $\alpha\in[0,2\pi]$, the following hold 
	\begin{align*}
	  &\frac{1}{2\pi}\int_{0}^{2\pi}e^{it\cos(\gamma-\alpha)}d\gamma=J_0(t), \\
	  &\frac{1}{2\pi}\int_{0}^{2\pi}\cos\gamma e^{it\cos(\gamma-\alpha)}d\gamma=iJ_1(t)\cos\alpha, \\
	  &\frac{1}{2\pi}\int_{0}^{2\pi}\sin\gamma e^{it\cos(\gamma-\alpha)}d\gamma=iJ_1(t)\sin\alpha, \\
	  &\frac{1}{2\pi}\int_{0}^{2\pi}\sin\gamma\cos\gamma e^{it\cos(\gamma-\alpha)}d\gamma=-J_2(t)\sin\alpha\cos\alpha, \\
	  &\frac{1}{2\pi}\int_{0}^{2\pi}\cos^2\gamma e^{it\cos(\gamma-\alpha)}d\gamma=\frac{1}{2}\left(J_0(t)-J_2(t)\cos2\alpha\right), \\
	  &\frac{1}{2\pi}\int_{0}^{2\pi}\sin^2\gamma e^{it\cos(\gamma-\alpha)}d\gamma=\frac{1}{2}\left(J_0(t)+J_2(t)\cos2\alpha\right). 
	\end{align*}
\end{lemma}
\begin{proof}
	By basic calculation we see that 
	\ben
	  \begin{aligned}
	  	\frac{1}{2\pi}\int_{0}^{2\pi}e^{it\cos(\gamma-\alpha)}d\gamma=&\;\frac{1}{2\pi}\left(\int_{-\alpha}^{0}+\int_{0}^{2\pi-\alpha}\right)e^{it\cos\gamma}d\gamma \\
	  =&\;\frac{1}{2\pi}\int_{0}^{2\pi}e^{it\cos\gamma}d\gamma=J_0(t), 
	  \end{aligned}
	\enn
	where the last equality comes from the integral representation of the Bessel function \cite{GN95}. From the recurrence formula of the Bessel functions \cite{GN95}, it follows that 
	\begin{align*}
	  J_1(t)=-J_0'(t)=\frac{-i}{2\pi}\int_{0}^{2\pi}\cos\gamma e^{it\cos\gamma}d\gamma. 
	\end{align*}
	Further, it is easily deduced that 
	$$
	  \int_{0}^{2\pi}\sin\gamma e^{it\cos\gamma}d\gamma=\int_{-\pi}^{\pi}\sin\gamma e^{it\cos\gamma}d\gamma=0. 
	$$
	Therefore, 
	\ben
	  \begin{aligned}
	  	&\;\frac{1}{2\pi}\int_{0}^{2\pi}\cos\gamma e^{it\cos(\gamma-\alpha)}d\gamma \\
	  =&\;\frac{1}{2\pi}\int_{-\alpha}^{2\pi-\alpha}\cos(\gamma+\alpha)e^{it\cos\gamma}d\gamma \\
	  =&\;\frac{1}{2\pi}\left(\cos\alpha\int_{-\alpha}^{2\pi-\alpha}\cos\gamma e^{it\cos\gamma}d\gamma-\sin\alpha\int_{-\alpha}^{2\pi-\alpha}\sin\gamma e^{it\cos\gamma}d\gamma\right) \\
	  =&\;\frac{1}{2\pi}\left(\cos\alpha\int_{0}^{2\pi}\cos\gamma e^{it\cos\gamma}d\gamma-\sin\alpha\int_{0}^{2\pi}\sin\gamma e^{it\cos\gamma}d\gamma\right) \\
	  =&\;iJ_1(t)\cos\alpha, 
	  \end{aligned}
	\enn
	and 
	\ben
	  \begin{aligned}
	  	&\;\frac{1}{2\pi}\int_{0}^{2\pi}\sin\gamma e^{it\cos(\gamma-\alpha)}d\gamma \\
	  =&\;\frac{1}{2\pi}\int_{-\alpha}^{2\pi-\alpha}\sin(\gamma+\alpha)e^{it\cos\gamma}d\gamma \\
	  =&\;\frac{1}{2\pi}\left(\sin\alpha\int_{0}^{2\pi}\cos\gamma e^{it\cos\gamma}d\gamma+\cos\alpha\int_{0}^{2\pi}\sin\gamma e^{it\cos\gamma}d\gamma\right) \\
	  =&\;iJ_1(t)\sin\alpha. 
	  \end{aligned}
	\enn
	Finally, it can be verified that 
	$$
	  J_2(t)=J_0(t)-2J_1'(t)=\frac{-1}{2\pi}\int_{0}^{2\pi}\cos2\gamma e^{it\cos\gamma}d\gamma
	$$
	and
	$$
	  \int_{0}^{2\pi}\sin2\gamma e^{it\cos\gamma}d\gamma=0. 
	$$
	Thus, following the similar procedure we can obtain the left three equalities and finish the proof. 
\end{proof}
\begin{theorem}\label{thm3.6}
	For $f=f(|\boldsymbol{\zeta}|^2,x_3,y_3)$, under the change of variables 
	$$
	  x_1-y_1=\rho\cos\alpha,~x_2-y_2=\rho\sin\alpha,~\alpha\in[0,2\pi], 
	$$
	and
	$$
	  \zeta_1=\xi\cos\gamma,~\zeta_2=\xi\sin\gamma,~\gamma\in[0,2\pi], 
	$$
	we have 
	\begin{align*}
	  &\frac{1}{(2\pi)^2}\int_{-\infty}^{+\infty}\int_{-\infty}^{+\infty}f(|\boldsymbol{\zeta}|^2,x_3,y_3)e^{i\boldsymbol{\zeta}\cdot(\mathbf{x}'-\mathbf{y}')}d\zeta_1d\zeta_2=\frac{1}{4\pi}\int_{\mathcal{C}}f(\xi^2,x_3,y_3)H_0^{(1)}(\xi\rho)\xi d\xi, \\
	  &\frac{1}{(2\pi)^2}\int_{-\infty}^{+\infty}\int_{-\infty}^{+\infty}f(|\boldsymbol{\zeta}|^2,x_3,y_3)\zeta_1e^{i\boldsymbol{\zeta}\cdot(\mathbf{x}'-\mathbf{y}')}d\zeta_1d\zeta_2=\frac{i}{4\pi}\cos\alpha\int_{\mathcal{C}}f(\xi^2,x_3,y_3)H_1^{(1)}(\xi\rho)\xi^2d\xi, \\
	  &\frac{1}{(2\pi)^2}\int_{-\infty}^{+\infty}\int_{-\infty}^{+\infty}f(|\boldsymbol{\zeta}|^2,x_3,y_3)\zeta_2e^{i\boldsymbol{\zeta}\cdot(\mathbf{x}'-\mathbf{y}')}d\zeta_1d\zeta_2=\frac{i}{4\pi}\sin\alpha\int_{\mathcal{C}}f(\xi^2,x_3,y_3)H_1^{(1)}(\xi\rho)\xi^2d\xi, \\ 
	  &\frac{1}{(2\pi)^2}\int_{-\infty}^{+\infty}\int_{-\infty}^{+\infty}f(|\boldsymbol{\zeta}|^2,x_3,y_3)\zeta_1\zeta_2e^{i\boldsymbol{\zeta}\cdot(\mathbf{x}'-\mathbf{y}')}d\zeta_1d\zeta_2=\frac{-1}{8\pi}\sin2\alpha\int_{\mathcal{C}}f(\xi^2,x_3,y_3)H_2^{(1)}(\xi\rho)\xi^3d\xi, \\
	  &\frac{1}{(2\pi)^2}\int_{-\infty}^{+\infty}\int_{-\infty}^{+\infty}f(|\boldsymbol{\zeta}|^2,x_3,y_3)\zeta_1^2e^{i\boldsymbol{\zeta}\cdot(\mathbf{x}'-\mathbf{y}')}d\zeta_1d\zeta_2 \\
	  &\qquad\qquad\qquad\qquad\qquad\qquad=\frac{1}{8\pi}\int_{\mathcal{C}}f(\xi^2,x_3,y_3)(H_0^{(1)}(\xi\rho)-H_2^{(1)}(\xi\rho)\cos2\alpha)\xi^3d\xi, \\
	  &\frac{1}{(2\pi)^2}\int_{-\infty}^{+\infty}\int_{-\infty}^{+\infty}f(|\boldsymbol{\zeta}|^2,x_3,y_3)\zeta_2^2e^{i\boldsymbol{\zeta}\cdot(\mathbf{x}'-\mathbf{y}')}d\zeta_1d\zeta_2 \\
	  &\qquad\qquad\qquad\qquad\qquad\qquad=\frac{1}{8\pi}\int_{\mathcal{C}}f(\xi^2,x_3,y_3)(H_0^{(1)}(\xi\rho)+H_2^{(1)}(\xi\rho)\cos2\alpha)\xi^3d\xi, 
	\end{align*}
	where the integration path $\mathcal{C}$ is depicted in Figure {\rm 3.3}. 
\end{theorem}
\begin{proof}
	Directly utilize Lemma \ref{lem3.5} and the fact that \cite{GN95} 
	$$
	  J_m(z)=\frac{1}{2}(H_m^{(1)}(z)-(-1)^mH_m^{(1)}(-z)),~-\pi<\arg z\leq\pi, 
	$$
	for $m=0,1,2$. 
\end{proof}

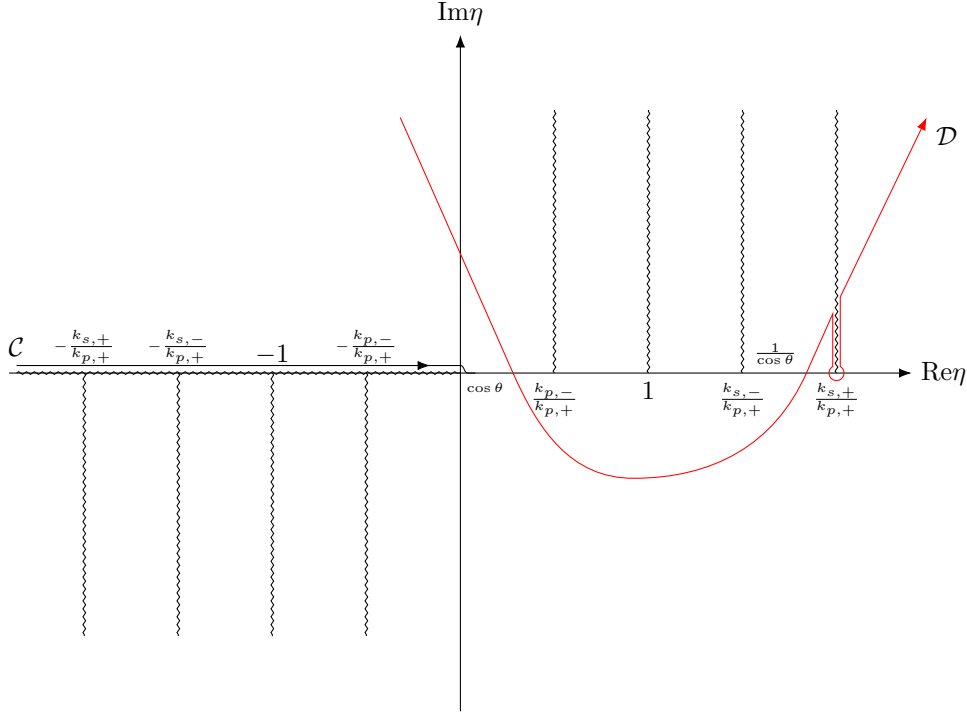
\begin{figure}\label{fig3.3}
\begin{adjustbox}{margin=0cm 0cm 0cm 0cm}
	\begin{tikzpicture}
		\tikzmath{\value=sqrt(3)/20;}
		\draw[-Latex] (-6,0) -- (6,0) node[right]{$\Rt\eta$}; 
		\draw[-Latex] (0,-4.5) -- (0,4.5) node[above]{$\I\eta$}; 
		\draw[decorate,decoration={zigzag, amplitude=0.5pt, segment length=3pt}] (1.25,0) node[font=\tiny, below]{$\frac{k_{p,-}}{k_{p,+}}$} -- (1.25,3.5); 
		\draw[decorate,decoration={zigzag, amplitude=0.5pt, segment length=3pt}] (-1.25,0) node[font=\tiny,above]{$-\frac{k_{p,-}}{k_{p,+}}$} -- (-1.25,-3.5); 
		\draw[decorate,decoration={zigzag, amplitude=0.5pt, segment length=3pt}] (2.5,0) node[below]{$1$} -- (2.5,3.5); 
		\draw[decorate,decoration={zigzag, amplitude=0.5pt, segment length=3pt}] (-2.5,0) node[above]{$-1$} -- (-2.5,-3.5); 
		\draw[decorate,decoration={zigzag, amplitude=0.5pt, segment length=3pt}] (3.75,0) node[font=\tiny,below]{$\frac{k_{s,-}}{k_{p,+}}$} -- (3.75,3.5); 
		\draw[decorate,decoration={zigzag, amplitude=0.5pt, segment length=3pt}] (-3.75,0) node[font=\tiny,above]{$-\frac{k_{s,-}}{k_{p,+}}$} -- (-3.75,-3.5); 
		\draw[decorate,decoration={zigzag, amplitude=0.5pt, segment length=3pt}] (5,0) node[font=\tiny,below]{$\frac{k_{s,+}}{k_{p,+}}$} -- (5,3.5); 
		\draw[decorate,decoration={zigzag, amplitude=0.5pt, segment length=3pt}] (-5,0) node[font=\tiny,above]{$-\frac{k_{s,+}}{k_{p,+}}$} -- (-5,-3.5); 
		\draw[red, smooth, -Latex] (-0.8,3.4) -- (0.7,0) node[font=\tiny, black, below left]{$\cos\theta$} to[out=-atan(34/15), in=180] (2.3,-1.4) to[out=0,in=180+atan(9/4)] (4.6,0) node[font=\tiny,black, above left]{$\frac{1}{\cos\theta}$} -- (4.95,0.7875) -- (4.95,\value) arc (120:420:0.1) -- (5.05,1.0125) -- (6.2,3.4) node[black, below right]{$\mathcal{D}$}; 
		\draw[decorate,decoration={zigzag, amplitude=0.5pt, segment length=3pt}] (-5.9,0) -- (0,0); 
		\draw[black,smooth] (-5.9,0.1) node[above]{$\mathcal{C}$} -- (0,0.1) to[out=0, in=180] (0.1,0) -- (0.2,0); 
		\draw[-Latex] (-0.5,0.1) -- (-0.4,0.1); 
	\end{tikzpicture}
\end{adjustbox}
\caption{Integration path $\mathcal{C}$ in the Hankel transform and the steepest descent path $\mathcal{D}$ utilized in the asymptotic approximation.}
\end{figure}

For $\mathbf{y}\in\R^3_+\cup\R^3_-$ and $\mathbf{x}=r\mathbf{\hat x}=r(\cos\varphi_x\cos\theta_x,\sin\varphi_x\cos\theta_x,\sin\theta_x)$ with $\theta_x\in(0,\pi/2]$ and $\varphi_x\in[0,2\pi]$, now letting 
$$
  x_1-y_1=R\cos\beta\cos\alpha,~x_2-y_2=R\cos\beta\sin\alpha,~x_3+|y_3|=R\sin\beta
$$
with $\rho=R\cos\beta$ and $\alpha\in[0,2\pi]$, $\beta\in[0,\pi/2)$. In view of Theorem \ref{thm3.6}, from the uniform asymptotic expansion of the Hankel function \cite{GN95}, for $m=0,1,2$, 
$$
  H_m^{(1)}(z)=\sqrt{\frac{2}{\pi z}}e^{i(z-\frac{\pi}{2}m-\frac{\pi}{4})}\left(1+O(|z|^{-1})\right)~\text{as}~|z|\rightarrow\infty, 
$$
we see that the integral expression for $\wid G_{p,j},\mathbf{\wid G}_{s,j},U_{p,j}$ and $\mathbf{U}_{s,j}$ can be handled similarily as in the two-dimensional case. Following the lines in \cite[Section 2.3.4]{CP17} (see details from (2.61) to (2.74) in \cite{CP17}), analyzing analogously as in the two-dimensional case, we can obtain the corresponding asymptotic expansions. 
\begin{theorem}\label{thm3.7}
  For $a=p,s$, $\mathbf{y}\in\R^3_+\cup\R^3_-$ and $\mathbf{x}=r\mathbf{\hat x}=r(\hat x_1,\hat x_2,\hat x_3)=r(\cos\varphi_x\cos\theta_x,\sin\varphi_x\cos\theta_x,\sin\theta_x)$ with $\theta_x\in(0,\pi/2]$ and $\varphi_x\in[0,2\pi]$, \\
  1. if $f\in\{A_{p,3}^+,B_{p,3}^+,A_{s,1}^{(2)+},B_{s,1}^{(2)+},A_{s,2}^{(1)^+},B_{s,2}^{(1)+},R_{p,3},T_{p,3}\}$, $f=f(|\boldsymbol{\zeta}|^2,y_3)$,  then for 
  $$
    F_{1,+}(\mathbf{x},\mathbf{y}):=\frac{1}{(2\pi)^2}\int_{-\infty}^{+\infty}\int_{-\infty}^{+\infty}f(|\boldsymbol{\zeta}|^2,y_3)e^{-\beta_{a,+}x_3}e^{i\boldsymbol{\zeta}\cdot(\mathbf{x}'-\mathbf{y}')}d\zeta_1d\zeta_2
  $$
  we have the asymptotic behaviors 
  \begin{align*}
    F_{1,+}(\mathbf{x},\mathbf{y})&=\frac{e^{ik_{a,+}r}}{r}F_{1,+}^\infty(\hat{\mathbf{x}},\mathbf{y})+O(r^{-\frac{5}{4}}), \\ 
    \nabla_{\mathbf{y}}F_{1,+}(\mathbf{x},\mathbf{y})&=\frac{e^{ik_{a,+}r}}{r}\nabla_{\mathbf{y}}F_{1,+}^\infty(\hat{\mathbf{x}},\mathbf{y})+O(r^{-\frac{5}{4}}), 
  \end{align*}
  where 
  $$
    F_{1,+}^\infty(\hat{\mathbf{x}},\mathbf{y}):=e^{-i\frac{\pi}{2}}\frac{k_{a,+}}{2\pi}\hat x_3f(k_{a,+}^2\cos^2\theta_x,y_3)e^{-ik_{a,+}\hat{\mathbf{x}}\cdot\mathbf{y}'}, 
  $$
  2. if $f\in\{A_{p,1}^+,B_{p,1}^+,A_{s,2}^{(3)+},B_{s,2}^{(3)+},A_{s,3}^{(2)+},B_{s,3}^{(2)+},R_{p,1},T_{p,1},R_{s,3}^{(2)},T_{s,3}^{(2)}\}$, $f=f(|\boldsymbol{\zeta}|^2,y_3)$,  then for 
  $$
    F_{2,+}(\mathbf{x},\mathbf{y}):=\frac{1}{(2\pi)^2}\int_{-\infty}^{+\infty}\int_{-\infty}^{+\infty}f(|\boldsymbol{\zeta}|^2,y_3)e^{-\beta_{a,+}x_3}\zeta_1e^{i\boldsymbol{\zeta}\cdot(\mathbf{x}'-\mathbf{y}')}d\zeta_1d\zeta_2
  $$
  we have the asymptotic behaviors 
  \begin{align*}
  F_{2,+}(\mathbf{x},\mathbf{y})&=\frac{e^{ik_{a,+}r}}{r}F_{2,+}^\infty(\hat{\mathbf{x}},\mathbf{y})+O(r^{-\frac{5}{4}}), \\ 
  \nabla_{\mathbf{y}}F_{2,+}(\mathbf{x},\mathbf{y})&=\frac{e^{ik_{a,+}r}}{r}\nabla_{\mathbf{y}}F_{2,+}^\infty(\hat{\mathbf{x}},\mathbf{y})+O(r^{-\frac{5}{4}}), 
  \end{align*}
  where 
  $$
    F_{2,+}^\infty(\hat{\mathbf{x}},\mathbf{y}):=e^{-i\frac{\pi}{2}}\frac{k_{a,+}^2}{2\pi}\hat x_1\hat x_3f(k_{a,+}^2\cos^2\theta_x,y_3)e^{-ik_{a,+}\hat{\mathbf{x}}\cdot\mathbf{y}'}, 
  $$
  3. if $f\in\{A_{p,2}^+,B_{p,2}^+,A_{s,1}^{(3)+},B_{s,1}^{(3)+},A_{s,3}^{(1)+},B_{s,3}^{(1)+},R_{p,2},T_{p,2},R_{s,3}^{(1)},T_{s,3}^{(1)}\}$, $f=f(|\boldsymbol{\zeta}|^2,y_3)$,  then for 
  $$
  F_{3,+}(\mathbf{x},\mathbf{y}):=\frac{1}{(2\pi)^2}\int_{-\infty}^{+\infty}\int_{-\infty}^{+\infty}f(|\boldsymbol{\zeta}|^2,y_3)e^{-\beta_{a,+}x_3}\zeta_2e^{i\boldsymbol{\zeta}\cdot(\mathbf{x}'-\mathbf{y}')}d\zeta_1d\zeta_2
  $$
  we have the asymptotic behaviors 
  \begin{align*}
  F_{3,+}(\mathbf{x},\mathbf{y})&=\frac{e^{ik_{a,+}r}}{r}F_{3,+}^\infty(\hat{\mathbf{x}},\mathbf{y})+O(r^{-\frac{5}{4}}), \\ 
  \nabla_{\mathbf{y}}F_{3,+}(\mathbf{x},\mathbf{y})&=\frac{e^{ik_{a,+}r}}{r}\nabla_{\mathbf{y}}F_{3,+}^\infty(\hat{\mathbf{x}},\mathbf{y})+O(r^{-\frac{5}{4}}), 
  \end{align*}
  where 
  $$
  F_{3,+}^\infty(\hat{\mathbf{x}},\mathbf{y}):=e^{-i\frac{\pi}{2}}\frac{k_{a,+}^2}{2\pi}\hat x_2\hat x_3f(k_{a,+}^2\cos^2\theta_x,y_3)e^{-ik_{a,+}\hat{\mathbf{x}}\cdot\mathbf{y}'}, 
  $$
  4. if $f\in\{A_{s,1}^{(1)},B_{s,1}^{(1)},A_{s,2}^{(2)},B_{s,2}^{(2)},R_{s,1}^{(1)},T_{s,1}^{(1)},R_{s,2}^{(2)},T_{s,2}^{(2)}\}$, $f=f(|\boldsymbol{\zeta}|^2,y_3)$,  then for 
  $$
  F_{4,+}(\mathbf{x},\mathbf{y}):=\frac{1}{(2\pi)^2}\int_{-\infty}^{+\infty}\int_{-\infty}^{+\infty}f(|\boldsymbol{\zeta}|^2,y_3)e^{-\beta_{a,+}x_3}\zeta_1\zeta_2e^{i\boldsymbol{\zeta}\cdot(\mathbf{x}'-\mathbf{y}')}d\zeta_1d\zeta_2
  $$
  we have the asymptotic behaviors 
  \begin{align*}
  F_{4,+}(\mathbf{x},\mathbf{y})&=\frac{e^{ik_{a,+}r}}{r}F_{4,+}^\infty(\hat{\mathbf{x}},\mathbf{y})+O(r^{-\frac{5}{4}}), \\ 
  \nabla_{\mathbf{y}}F_{4,+}(\mathbf{x},\mathbf{y})&=\frac{e^{ik_{a,+}r}}{r}\nabla_{\mathbf{y}}F_{4,+}^\infty(\hat{\mathbf{x}},\mathbf{y})+O(r^{-\frac{5}{4}}), 
  \end{align*}
  where 
  $$
  F_{4,+}^\infty(\hat{\mathbf{x}},\mathbf{y}):=e^{-i\frac{\pi}{2}}\frac{k_{a,+}^3}{2\pi}\hat x_1\hat x_2\hat x_3f(k_{a,+}^2\cos^2\theta_x,y_3)e^{-ik_{a,+}\hat{\mathbf{x}}\cdot\mathbf{y}'}, 
  $$
  5. if $f\in\{\hat A_{s,1}^{(2)},\hat B_{s,1}^{(2)},R_{s,1}^{(2)},T_{s,1}^{(2)}\}$, $f=f(|\boldsymbol{\zeta}|^2,y_3)$,  then for 
  $$
  F_{5,+}(\mathbf{x},\mathbf{y}):=\frac{1}{(2\pi)^2}\int_{-\infty}^{+\infty}\int_{-\infty}^{+\infty}f(|\boldsymbol{\zeta}|^2,y_3)e^{-\beta_{a,+}x_3}\zeta_1^2e^{i\boldsymbol{\zeta}\cdot(\mathbf{x}'-\mathbf{y}')}d\zeta_1d\zeta_2
  $$
  we have the asymptotic behaviors 
  \begin{align*}
  F_{5,+}(\mathbf{x},\mathbf{y})&=\frac{e^{ik_{a,+}r}}{r}F_{5,+}^\infty(\hat{\mathbf{x}},\mathbf{y})+O(r^{-\frac{5}{4}}), \\ 
  \nabla_{\mathbf{y}}F_{5,+}(\mathbf{x},\mathbf{y})&=\frac{e^{ik_{a,+}r}}{r}\nabla_{\mathbf{y}}F_{5,+}^\infty(\hat{\mathbf{x}},\mathbf{y})+O(r^{-\frac{5}{4}}), 
  \end{align*}
  where 
  $$
  F_{5,+}^\infty(\hat{\mathbf{x}},\mathbf{y}):=e^{-i\frac{\pi}{2}}\frac{k_{a,+}^3}{2\pi}\hat x_1^2\hat x_3f(k_{a,+}^2\cos^2\theta_x,y_3)e^{-ik_{a,+}\hat{\mathbf{x}}\cdot\mathbf{y}'}, 
  $$
  6. if $f\in\{\hat A_{s,2}^{(1)},\hat B_{s,2}^{(1)},R_{s,2}^{(1)},T_{s,2}^{(1)}\}$, $f=f(|\boldsymbol{\zeta}|^2,y_3)$, then for 
  $$
  F_{6,+}(\mathbf{x},\mathbf{y}):=\frac{1}{(2\pi)^2}\int_{-\infty}^{+\infty}\int_{-\infty}^{+\infty}f(|\boldsymbol{\zeta}|^2,y_3)e^{-\beta_{a,+}x_3}\zeta_2^2e^{i\boldsymbol{\zeta}\cdot(\mathbf{x}'-\mathbf{y}')}d\zeta_1d\zeta_2
  $$
  we have the asymptotic behaviors 
  \begin{align*}
  F_{6,+}(\mathbf{x},\mathbf{y})&=\frac{e^{ik_{a,+}r}}{r}F_{6,+}^\infty(\hat{\mathbf{x}},\mathbf{y})+O(r^{-\frac{5}{4}}), \\ 
  \nabla_{\mathbf{y}}F_{6,+}(\mathbf{x},\mathbf{y})&=\frac{e^{ik_{a,+}r}}{r}\nabla_{\mathbf{y}}F_{6,+}^\infty(\hat{\mathbf{x}},\mathbf{y})+O(r^{-\frac{5}{4}}), 
  \end{align*}
  where 
  $$
  F_{6,+}^\infty(\hat{\mathbf{x}},\mathbf{y}):=e^{-i\frac{\pi}{2}}\frac{k_{a,+}^3}{2\pi}\hat x_2^2\hat x_3f(k_{a,+}^2\cos^2\theta_x,y_3)e^{-ik_{a,+}\hat{\mathbf{x}}\cdot\mathbf{y}'}. 
  $$
\end{theorem}

If $\theta_x\in[-\pi/2,0)$, it follows similarily that: 
\begin{theorem}\label{thm3.8}
	For $a=p,s$, $\mathbf{y}\in\R^3_+\cup\R^3_-$ and $\mathbf{x}=r\mathbf{\hat x}=r(\hat x_1,\hat x_2,\hat x_3)=r(\cos\varphi_x\cos\theta_x,\sin\varphi_x\cos\theta_x,\sin\theta_x)$ with $\theta_x\in[-\pi/2,0)$ and $\varphi_x\in[0,2\pi]$, \\ 
	1. if $f\in\{A_{p,3}^-,B_{p,3}^-,A_{s,1}^{(2)-},B_{s,1}^{(2)-},A_{s,2}^{(1)^-},B_{s,2}^{(1)-},R_{p,3},T_{p,3}\}$, $f=f(|\boldsymbol{\zeta}|^2,y_3)$,  then for 
	$$
	F_{1,-}(\mathbf{x},\mathbf{y}):=\frac{1}{(2\pi)^2}\int_{-\infty}^{+\infty}\int_{-\infty}^{+\infty}f(|\boldsymbol{\zeta}|^2,y_3)e^{\beta_{a,-}x_3}e^{i\boldsymbol{\zeta}\cdot(\mathbf{x}'-\mathbf{y}')}d\zeta_1d\zeta_2
	$$
	we have the asymptotic behaviors 
	\begin{align*}
	F_{1,-}(\mathbf{x},\mathbf{y})&=\frac{e^{ik_{a,-}r}}{r}F_{1,-}^\infty(\hat{\mathbf{x}},\mathbf{y})+O(r^{-\frac{5}{4}}), \\ 
	\nabla_{\mathbf{y}}F_{1,-}(\mathbf{x},\mathbf{y})&=\frac{e^{ik_{a,-}r}}{r}\nabla_{\mathbf{y}}F_{1,-}^\infty(\hat{\mathbf{x}},\mathbf{y})+O(r^{-\frac{5}{4}}), 
	\end{align*}
	where 
	$$
	F_{1,-}^\infty(\hat{\mathbf{x}},\mathbf{y}):=e^{i\frac{\pi}{2}}\frac{k_{a,-}}{2\pi}\hat x_3f(k_{a,-}^2\cos^2\theta_x,y_3)e^{-ik_{a,-}\hat{\mathbf{x}}\cdot\mathbf{y}'}, 
	$$
	2. if $f\in\{A_{p,1}^-,B_{p,1}^-,A_{s,2}^{(3)-},B_{s,2}^{(3)-},A_{s,3}^{(2)-},B_{s,3}^{(2)-},R_{p,1},T_{p,1},R_{s,3}^{(2)},T_{s,3}^{(2)}\}$, $f=f(|\boldsymbol{\zeta}|^2,y_3)$,  then for 
	$$
	F_{2,-}(\mathbf{x},\mathbf{y}):=\frac{1}{(2\pi)^2}\int_{-\infty}^{+\infty}\int_{-\infty}^{+\infty}f(|\boldsymbol{\zeta}|^2,y_3)e^{\beta_{a,-}x_3}\zeta_1e^{i\boldsymbol{\zeta}\cdot(\mathbf{x}'-\mathbf{y}')}d\zeta_1d\zeta_2
	$$
	we have the asymptotic behaviors 
	\begin{align*}
	F_{2,-}(\mathbf{x},\mathbf{y})&=\frac{e^{ik_{a,-}r}}{r}F_{2,-}^\infty(\hat{\mathbf{x}},\mathbf{y})+O(r^{-\frac{5}{4}}), \\ 
	\nabla_{\mathbf{y}}F_{2,-}(\mathbf{x},\mathbf{y})&=\frac{e^{ik_{a,-}r}}{r}\nabla_{\mathbf{y}}F_{2,-}^\infty(\hat{\mathbf{x}},\mathbf{y})+O(r^{-\frac{5}{4}}), 
	\end{align*}
	where 
	$$
	F_{2,-}^\infty(\hat{\mathbf{x}},\mathbf{y}):=e^{i\frac{\pi}{2}}\frac{k_{a,-}^2}{2\pi}\hat x_1\hat x_3f(k_{a,-}^2\cos^2\theta_x,y_3)e^{-ik_{a,-}\hat{\mathbf{x}}\cdot\mathbf{y}'}, 
	$$
	3. if $f\in\{A_{p,2}^-,B_{p,2}^-,A_{s,1}^{(3)-},B_{s,1}^{(3)-},A_{s,3}^{(1)-},B_{s,3}^{(1)-},R_{p,2},T_{p,2},R_{s,3}^{(1)},T_{s,3}^{(1)}\}$, $f=f(|\boldsymbol{\zeta}|^2,y_3)$,  then for 
	$$
	F_{3,-}(\mathbf{x},\mathbf{y}):=\frac{1}{(2\pi)^2}\int_{-\infty}^{+\infty}\int_{-\infty}^{+\infty}f(|\boldsymbol{\zeta}|^2,y_3)e^{\beta_{a,-}x_3}\zeta_2e^{i\boldsymbol{\zeta}\cdot(\mathbf{x}'-\mathbf{y}')}d\zeta_1d\zeta_2
	$$
	we have the asymptotic behaviors 
	\begin{align*}
	F_{3,-}(\mathbf{x},\mathbf{y})&=\frac{e^{ik_{a,-}r}}{r}F_{3,-}^\infty(\hat{\mathbf{x}},\mathbf{y})+O(r^{-\frac{5}{4}}), \\ 
	\nabla_{\mathbf{y}}F_{3,-}(\mathbf{x},\mathbf{y})&=\frac{e^{ik_{a,-}r}}{r}\nabla_{\mathbf{y}}F_{3,-}^\infty(\hat{\mathbf{x}},\mathbf{y})+O(r^{-\frac{5}{4}}), 
	\end{align*}
	where 
	$$
	F_{3,-}^\infty(\hat{\mathbf{x}},\mathbf{y}):=e^{i\frac{\pi}{2}}\frac{k_{a,-}^2}{2\pi}\hat x_2\hat x_3f(k_{a,-}^2\cos^2\theta_x,y_3)e^{-ik_{a,-}\hat{\mathbf{x}}\cdot\mathbf{y}'}, 
	$$
	4. if $f\in\{A_{s,1}^{(1)},B_{s,1}^{(1)},A_{s,2}^{(2)},B_{s,2}^{(2)},R_{s,1}^{(1)},T_{s,1}^{(1)},R_{s,2}^{(2)},T_{s,2}^{(2)}\}$, $f=f(|\boldsymbol{\zeta}|^2,y_3)$,  then for 
	$$
	F_{4,-}(\mathbf{x},\mathbf{y}):=\frac{1}{(2\pi)^2}\int_{-\infty}^{+\infty}\int_{-\infty}^{+\infty}f(|\boldsymbol{\zeta}|^2,y_3)e^{\beta_{a,-}x_3}\zeta_1\zeta_2e^{i\boldsymbol{\zeta}\cdot(\mathbf{x}'-\mathbf{y}')}d\zeta_1d\zeta_2
	$$
	we have the asymptotic behaviors 
	\begin{align*}
	F_{4,-}(\mathbf{x},\mathbf{y})&=\frac{e^{ik_{a,-}r}}{r}F_{4,-}^\infty(\hat{\mathbf{x}},\mathbf{y})+O(r^{-\frac{5}{4}}), \\ 
	\nabla_{\mathbf{y}}F_{4,-}(\mathbf{x},\mathbf{y})&=\frac{e^{ik_{a,-}r}}{r}\nabla_{\mathbf{y}}F_{4,-}^\infty(\hat{\mathbf{x}},\mathbf{y})+O(r^{-\frac{5}{4}}), 
	\end{align*}
	where 
	$$
	F_{4,-}^\infty(\hat{\mathbf{x}},\mathbf{y}):=e^{i\frac{\pi}{2}}\frac{k_{a,-}^3}{2\pi}\hat x_1\hat x_2\hat x_3f(k_{a,-}^2\cos^2\theta_x,y_3)e^{-ik_{a,-}\hat{\mathbf{x}}\cdot\mathbf{y}'}, 
	$$
	5. if $f\in\{\hat A_{s,1}^{(2)},\hat B_{s,1}^{(2)},R_{s,1}^{(2)},T_{s,1}^{(2)}\}$, $f=f(|\boldsymbol{\zeta}|^2,y_3)$,  then for 
	$$
	F_{5,-}(\mathbf{x},\mathbf{y}):=\frac{1}{(2\pi)^2}\int_{-\infty}^{+\infty}\int_{-\infty}^{+\infty}f(|\boldsymbol{\zeta}|^2,y_3)e^{\beta_{a,-}x_3}\zeta_1^2e^{i\boldsymbol{\zeta}\cdot(\mathbf{x}'-\mathbf{y}')}d\zeta_1d\zeta_2
	$$
	we have the asymptotic behaviors 
	\begin{align*}
	F_{5,-}(\mathbf{x},\mathbf{y})&=\frac{e^{ik_{a,-}r}}{r}F_{5,-}^\infty(\hat{\mathbf{x}},\mathbf{y})+O(r^{-\frac{5}{4}}), \\ 
	\nabla_{\mathbf{y}}F_{5,-}(\mathbf{x},\mathbf{y})&=\frac{e^{ik_{a,-}r}}{r}\nabla_{\mathbf{y}}F_{5,-}^\infty(\hat{\mathbf{x}},\mathbf{y})+O(r^{-\frac{5}{4}}), 
	\end{align*}
	where 
	$$
	F_{5,-}^\infty(\hat{\mathbf{x}},\mathbf{y}):=e^{i\frac{\pi}{2}}\frac{k_{a,-}^3}{2\pi}\hat x_1^2\hat x_3f(k_{a,-}^2\cos^2\theta_x,y_3)e^{-ik_{a,-}\hat{\mathbf{x}}\cdot\mathbf{y}'}, 
	$$
	6. if $f\in\{\hat A_{s,2}^{(1)},\hat B_{s,2}^{(1)},R_{s,2}^{(1)},T_{s,2}^{(1)}\}$, $f=f(|\boldsymbol{\zeta}|^2,y_3)$, then for 
	$$
	F_{6,-}(\mathbf{x},\mathbf{y}):=\frac{1}{(2\pi)^2}\int_{-\infty}^{+\infty}\int_{-\infty}^{+\infty}f(|\boldsymbol{\zeta}|^2,y_3)e^{\beta_{a,-}x_3}\zeta_2^2e^{i\boldsymbol{\zeta}\cdot(\mathbf{x}'-\mathbf{y}')}d\zeta_1d\zeta_2
	$$
	we have the asymptotic behaviors 
	\begin{align*}
	F_{6,-}(\mathbf{x},\mathbf{y})&=\frac{e^{ik_{a,-}r}}{r}F_{6,-}^\infty(\hat{\mathbf{x}},\mathbf{y})+O(r^{-\frac{5}{4}}), \\ 
	\nabla_{\mathbf{y}}F_{6,-}(\mathbf{x},\mathbf{y})&=\frac{e^{ik_{a,-}r}}{r}\nabla_{\mathbf{y}}F_{6,-}^\infty(\hat{\mathbf{x}},\mathbf{y})+O(r^{-\frac{5}{4}}), 
	\end{align*}
	where 
	$$
	F_{6,-}^\infty(\hat{\mathbf{x}},\mathbf{y}):=e^{i\frac{\pi}{2}}\frac{k_{a,-}^3}{2\pi}\hat x_2^2\hat x_3f(k_{a,-}^2\cos^2\theta_x,y_3)e^{-ik_{a,-}\hat{\mathbf{x}}\cdot\mathbf{y}'}. 
	$$
\end{theorem}

From Theorems \ref{thm3.7} and \ref{thm3.8}, we see that all the $\wid G_{p,j},\mathbf{\wid G}_{s,j},U_{p,j}$ and $\mathbf{U}_{s,j}$ satisfy the radiation condition \eqref{1.3b}. Therefore, we obtain the existence of the Green's tensor. 
\begin{theorem}\label{thm3.9}
	The two-layered Green's tensor $\mathbf{G}(\mathbf{x},\mathbf{y})$ defined by \eqref{3.9} exists and is unique. Moreover, it has the explicit form $\mathbf{G}(\mathbf{x},\mathbf{y})=(\mathbf{G}_1(\mathbf{x},\mathbf{y}),\mathbf{G}_2(\mathbf{x},\mathbf{y}),$ $\mathbf{G}_3(\mathbf{x},\mathbf{y}))$ with the column vector $\mathbf{G}_j=-k_p^{-2}\grad_{\mathbf{x}}G_{p,j}+k_s^{-2}\curl_{\mathbf{x}}\mathbf{G}_{s,j}$, $j=1,2,3$, where $G_{p,j}=\wid G_{p,j}+U_{p,j}$ and $\mathbf{G}_{s,j}=\mathbf{\wid G}_{s,j}+\mathbf{U}_{s,j}$. 
\end{theorem}
\begin{remark}\label{rem3.11}
	Still, we have $\mathbf{G}(\mathbf{x},\mathbf{y})=\mathbf{G}(\mathbf{y},\mathbf{x})^\top$ for $\mathbf{x},\mathbf{y}\in\R^3_\pm$. 
\end{remark}

\section{Existence of solution}\label{sec4}
\setcounter{equation}{0}
In this section, we prove the existence of solutions to problem \eqref{1.1}--\eqref{1.3} when $a_0=1$. Following the arguments in \cite{GXY18,GGT18}, we propose a variational formulation in $B_R$ coupled with a Dirichlet-to-Neumann map derived from the integral representation of the scattered field and transmitted field in $\R^d\setminus\ov B_R$. The variational formulation is then shown to be of Fredholm type, and thus the existence of solutions follows from the uniqueness, thereby achieving the first well-posedness result for elastic scattering by rough interfaces. 

In this section, by Remark \ref{rem2.2}, without loss of genearlity, we always set $\tilde{\mu}=\mu(\lambda+\mu)/(\lambda+3\mu)$ and $\tilde{\lambda}=(\lambda+\mu)(\lambda+2\mu)/(\lambda+3\mu)$. Denote by $\boldsymbol{\Pi}_\pm(\mathbf{x},\mathbf{y})$ the free space Green's tensor for the Navier equation $(\Delta^*+\rho_\pm\om^2)\mathbf{u}=0$, which is given by 
$$
  \boldsymbol{\Pi}_\pm(\mathbf{x},\mathbf{y})=\frac{1}{\mu}\Phi_{k_{s,\pm}}(\mathbf{x},\mathbf{y})\mathbf{I}+\frac{1}{\rho_\pm\om^2}\nabla_{\mathbf{x}}\nabla_{\mathbf{x}}^\top(\Phi_{k_{s,\pm}}(\mathbf{x},\mathbf{y})-\Phi_{k_{p,\pm}}(\mathbf{x},\mathbf{y})). 
$$
We refer to \cite{TA01,PH98} for some basic properties of $\boldsymbol{\Pi}_\pm$. From the definition and the explicit expression of the two-layered Green's tensor $\mathbf{G}$, it is noted that $\mathbf{G}(\mathbf{x},\mathbf{y})-\boldsymbol{\Pi}_\pm(\mathbf{x},\mathbf{y})\in [C^1(\R^d\times\R^d)]^{d\times d}$. 

We shall consider the point source incident wave $\mathbf{u}^{in}(\mathbf{x},\mathbf{z},\mathbf{a})=\boldsymbol{\Pi}_+(\mathbf{x},\mathbf{z})\mathbf{a}$ with $\mathbf{z}\in D^+$ and $\mathbf{a}\in\C^d$. In the following, we choose $R>0$ sufficiently large such that $\Gamma$ is flat outside $B_{R/2}$ and $\mathbf{z}\notin\pa B_R$. First we introduce the reflected wave $\mathbf{u}^{re}_\pm(\mathbf{x},\mathbf{z},\mathbf{a})$, which is the scattered field corresponding to the unperturbed problem \eqref{1.1}--\eqref{1.3}, i.e., $\Gamma=\Gamma_0$. Further, we define the reference wave $\mathbf{u}^{0}(\mathbf{x},\mathbf{z},\mathbf{a})$ by 
\begin{align*}
\mathbf{u}^{0}(\mathbf{x},\mathbf{z},\mathbf{a}):=
\left\{
\begin{array}{ll}
	\mathbf{u}^{in}(\mathbf{x},\mathbf{z},\mathbf{a})+\mathbf{u}^{re}_+(\mathbf{x},\mathbf{z},\mathbf{a}),\;&\mathbf{x}\in\R^d_+,\\ [1mm]
	\mathbf{u}^{re}_-(\mathbf{x},\mathbf{z},\mathbf{a}),\;&\mathbf{x}\in\R^d_-. 
\end{array}
\right.
\end{align*}
In view of the uniqueness of the direct problem, it is easily deduced that $\mathbf{u}^{0}(\mathbf{x},\mathbf{z},\mathbf{a})=\mathbf{G}(\mathbf{x},\mathbf{z})\mathbf{a}$ if $\mathbf{z}\in D^+\cap\R^d_+$, and $\mathbf{u}^{0}(\mathbf{x},\mathbf{z},\mathbf{a})=0$ if $\mathbf{z}\in D^+\setminus\R^d_+$. 

Now we split the scattered wave $\mathbf{u}_\pm$ into two parts in $\R^d\setminus B_R$. Define $\mathbf{\wid u}(\mathbf{x},\mathbf{z},\mathbf{a}):=\mathbf{u}_\pm(\mathbf{x},\mathbf{z},\mathbf{a})-\mathbf{u}^{re}_\pm(\mathbf{x},\mathbf{z},\mathbf{a})$ in $\ov\R^d_\pm\setminus B_R$. It is seen that $[\mathbf{\wid u}(\cdot,\mathbf{z},\mathbf{a})]=[\mathbf{P}_{\tilde{\mu},\tilde{\lambda}}\mathbf{\wid u}(\cdot,\mathbf{z},\mathbf{a})]=0$ on $\Gamma\setminus\ov B_R$. By applying the Betti formula and Theorems \ref{thm2.4a} and \ref{thm2.5a}, similarily as in \cite{GXY18,PH98,VD79}, it is easy to derive the Green's representation formula for $\mathbf{\wid u}$: 
\begin{align}\label{4.1}
  \mathbf{\wid u}(\mathbf{x},\mathbf{z},\mathbf{a})=\int_{\pa B_R}\left(\boldsymbol{\Pi}^{(2)}_{\tilde{\mu},\tilde{\lambda}}(\mathbf{x},\mathbf{y})\mathbf{\wid u}(\mathbf{y},\mathbf{z},\mathbf{a})-\mathbf{G}(\mathbf{x},\mathbf{y})\mathbf{P}_{\tilde{\mu},\tilde{\lambda}}^{(\mathbf{y})}\mathbf{\wid u}(\mathbf{y},\mathbf{z},\mathbf{a})\right)ds(\mathbf{y}),~~~
\end{align}
where $\mathbf{x}\in\R^d\setminus\ov B_R$ and  $(\boldsymbol{\Pi}^{(2)}_{\tilde{\mu},\tilde{\lambda}})_{jk}(\mathbf{x},\mathbf{y}):=(\mathbf{P}_{\tilde{\mu},\tilde{\lambda}}^{(\mathbf{y})}(\mathbf{G}_{j\cdot}(\mathbf{x},\mathbf{y}))^\top)_k$ with $\mathbf{G}=(\mathbf{G}_{jk})$. Taking the limit $\mathbf{x}\rightarrow\pa B_R$ in \eqref{4.1} and setting $\mathbf{p}=\mathbf{P}_{\tilde{\mu},\tilde{\lambda}}\mathbf{\wid u}|_{\pa B_R}$, we obtain that 
\begin{align}\label{4.2}
  (\frac{1}{2}\mathcal{I}-\mathcal{K})(\mathbf{\wid u}|_{\pa B_R})+\mathcal{S}\mathbf{p}=0~~~{\text on}~\pa B_R. 
\end{align}
Here $\mathcal{I}$ is the identity operator, $\mathcal{K}$ and $\mathcal{S}$ are the boundary integral operators over $\pa B_R$ defined by 
\begin{align*}
 &( \mathcal{K}\mathbf{g})(\mathbf{x}):=\int_{\pa B_R}\boldsymbol{\Pi}_{\tilde{\mu},\tilde{\lambda}}^{(2)}(\mathbf{x},\mathbf{y})\mathbf{g}(\mathbf{y})ds(\mathbf{y}), \\ 
 &(\mathcal{S}\mathbf{g})(\mathbf{x}):=\int_{\pa B_R}\mathbf{G}(\mathbf{x},\mathbf{y})\mathbf{g}(\mathbf{y})ds(\mathbf{y}). 
\end{align*}
We note that the mapping properties and the jump relations of these operators remain valid due to the smoothness of $\mathbf{G}-\boldsymbol{\Pi}_\pm$ and our special choice for $\tilde{\mu}$ and $\tilde{\lambda}$ (see details in \cite{GXY18,WM10,VD79}). Let 
\begin{align*}
\mathbf{\hat u}(\mathbf{x},\mathbf{z},\mathbf{a}):=
\left\{
\begin{array}{ll}
	\mathbf{u}_+(\mathbf{x},\mathbf{z},\mathbf{a}),\;&\mathbf{x}\in B_R\cap D^+,\\ [1mm]
	\mathbf{u}_-(\mathbf{x},\mathbf{z},\mathbf{a})-\mathbf{u}^{in}(\mathbf{x},\mathbf{z},\mathbf{a}),\;&\mathbf{x}\in B_R\cap D^-. 
\end{array}
\right.
\end{align*}
Utilizing the Betti formula for $\mathbf{\hat u}$ in $B_R$, we then obtain the equivalent variational formualtion of problem \eqref{1.1}--\eqref{1.3}:  find $(\mathbf{\hat u},\mathbf{p})\in [H^1(B_R)]^d\times [H^{-1/2}(\pa B_R)]^d:=\mathbf{X}$ such that 
\begin{align}\label{4.3}
\mathbb{B}((\mathbf{\hat u},\mathbf{p}),(\boldsymbol{\varphi},\boldsymbol{\chi}))=
\left(
\begin{array}{l}
	b_1((\mathbf{\hat u},\mathbf{p}),(\boldsymbol{\varphi},\boldsymbol{\chi})) \\
	b_2((\mathbf{\hat u},\mathbf{p}),(\boldsymbol{\varphi},\boldsymbol{\chi})) 
\end{array}
\right)
=
\left(
\begin{array}{l}
	L_1(\boldsymbol{\varphi},\boldsymbol{\chi}) \\
	L_2(\boldsymbol{\varphi},\boldsymbol{\chi}) 
\end{array}
\right)
\end{align}
for all $(\boldsymbol{\varphi},\boldsymbol{\chi})\in\mathbf{X}$, where 
\begin{align*}
  &b_1((\mathbf{\hat u},\mathbf{p}),(\boldsymbol{\varphi},\boldsymbol{\chi})):=\int_{B_R}\left(\mathcal{E}_{\tilde{\mu},\tilde{\lambda}}(\mathbf{\hat u},\ov{\boldsymbol{\varphi}})-\rho\om^2\mathbf{\hat u}\cdot\ov{\boldsymbol{\varphi}}\right)d\mathbf{x}-\int_{\pa B_R}\mathbf{p}\cdot\ov{\boldsymbol{\varphi}}ds, \\
  &b_2((\mathbf{\hat u},\mathbf{p}),(\boldsymbol{\varphi},\boldsymbol{\chi})):=\int_{\pa B_R}\left((\frac{1}{2}\mathcal{I}-\mathcal{K})(\mathbf{\hat u}|_{\pa B_R})+\mathcal{S}\mathbf{p}\right)\cdot\ov{\boldsymbol{\chi}}ds, \\
  &L_1(\boldsymbol{\varphi},\boldsymbol{\chi}):=-\int_{B_R\cap D^-}\left(\mathcal{E}_{\tilde{\mu},\tilde{\lambda}}(\mathbf{u}^{in},\ov{\boldsymbol{\varphi}})-\rho_-\om^2\mathbf{u}^{in}\cdot\ov{\boldsymbol{\varphi}}\right)d\mathbf{x}+\int_{\pa(B_R\cap D^-)}\mathbf{P}_{\tilde{\mu},\tilde{\lambda}}\mathbf{u}^{in}\cdot\ov{\boldsymbol{\varphi}}ds \\
  &\qquad\qquad\qquad+\int_{\pa B_R^+}\mathbf{P}_{\tilde{\mu},\tilde{\lambda}}\mathbf{u}^{re}_+\cdot\ov{\boldsymbol{\varphi}}ds+\int_{\pa B_R^-}\mathbf{P}_{\tilde{\mu},\tilde{\lambda}}\mathbf{u}^{re}_-\cdot\ov{\boldsymbol{\varphi}}ds \\
  &L_2(\boldsymbol{\varphi},\boldsymbol{\chi}):=\int_{\pa B_R}\left((\frac{1}{2}\mathcal{I}-\mathcal{K})(\mathbf{u}^{re}_+|_{\pa B_R^+}+\mathbf{u}^{re}_-|_{\pa B_R^-}-\mathbf{u}^{in}|_{\pa B_R^-})\right)\cdot\ov{\boldsymbol{\chi}}ds, 
\end{align*}
and 
\begin{align*}
  \mathcal{E}_{\tilde{\mu},\tilde{\lambda}}(\mathbf{\hat u},\ov{\boldsymbol{\varphi}})=
  \left\{
  \begin{array}{ll}
  	(\mu+\tilde{\mu})\nabla\mathbf{\hat u}\cdot\na\ov{\boldsymbol{\varphi}}+\tilde{\lambda}\dive\mathbf{\hat u}\dive\ov{\boldsymbol{\varphi}}-\tilde{\mu}\dive^\perp\mathbf{\hat u}\dive^\perp\ov{\boldsymbol{\varphi}},~~~&{\rm if}~d=2,\\
  	(\mu+\tilde{\mu})\nabla\mathbf{\hat u}\cdot\na\ov{\boldsymbol{\varphi}}+\tilde{\lambda}\dive\mathbf{\hat u}\dive\ov{\boldsymbol{\varphi}}-\tilde{\mu}\curl\mathbf{\hat u}\cdot\curl\ov{\boldsymbol{\varphi}},~~~&{\rm if}~d=3. 
  \end{array}
  \right.
\end{align*}

Denote by $(\cdot,\cdot)$ the duality between $[H^1(B_R)]^d$ and $[H^{-1}(B_R)]^d$, and by $\langle\cdot,\cdot\rangle$ the duality between $[H^{1/2}(\pa B_R)]^d$ and $[H^{-1/2}(\pa B_R)]^d$. From the Riesz representation theorem, there exist linear bounded operators 
  \begin{align*}
  	T_1,J_1:&\;[H^1(B_R)]^d\rightarrow[H^{-1}(B_R)]^d, \\
  T_2:&\;[H^{-1/2}(\pa B_R)]^d\rightarrow[H^{-1}(B_R)]^d, \\
  T_3:&\;[H^{-1/2}(\pa B_R)]^d\rightarrow[H^{1/2}(\pa B_R)]^d, \\
  J_2:&\;[H^1(B_R)]^d\rightarrow[H^{1/2}(\pa B_R)]^d, 
  \end{align*}
such that for $(\mathbf{\hat u},\mathbf{p}),(\boldsymbol{\varphi},\boldsymbol{\chi})\in\mathbf{X}$, 
  \begin{align*}
  	(T_1\mathbf{\hat u},\boldsymbol{\varphi})&:=\frac{1}{2}\int_{B_R}\left(\mathcal{E}_{\tilde{\mu},\tilde{\lambda}}(\mathbf{\hat u},\ov{\boldsymbol{\varphi}})+\rho\om^2\mathbf{\hat u}\cdot\ov{\boldsymbol{\varphi}}\right)d\mathbf{x}, \\
  (J_1\mathbf{\hat u},\boldsymbol{\varphi})&:=-\int_{B_R}\rho\om^2\mathbf{\hat u}\cdot\ov{\boldsymbol{\varphi}}d\mathbf{x}, \\
  \langle T_2\mathbf{p},\boldsymbol{\varphi}\rangle&:=\frac{1}{2}\int_{\pa B_R}\mathbf{p}\cdot\ov{\boldsymbol{\varphi}}ds, \\
  \langle T_3\mathbf{p},\boldsymbol{\chi}\rangle&:=\int_{\pa B_R}\mathcal{S}\mathbf{p}\cdot\ov{\boldsymbol{\chi}}ds, \\
  \langle J_2\mathbf{\hat u},\boldsymbol{\chi}\rangle&:=-\int_{\pa B_R}\mathcal{K}(\mathbf{\hat u}|_{\pa B_R})\cdot\ov{\boldsymbol{\chi}}ds. 
  \end{align*}
Therefore, we may rewrite $\mathbb{B}:\mathbf{X}\times\mathbf{X}\rightarrow\C^2$ as 
$$
  \mathbb{B}((\mathbf{\hat u},\mathbf{p}),(\boldsymbol{\varphi},\boldsymbol{\chi}))=\langle\mathbb{B}_1((\mathbf{\hat u},\mathbf{p}),(\boldsymbol{\varphi},\boldsymbol{\chi}))\rangle+\langle\mathbb{B}_2((\mathbf{\hat u},\mathbf{p}),(\boldsymbol{\varphi},\boldsymbol{\chi}))\rangle, 
$$
where $\langle\cdot,\cdot\rangle$ denotes the duality bwtween $\mathbf{X}$ and $\mathbf{X}'$, and the operators $\mathbb{B}_j:\mathbf{X}\rightarrow\mathbf{X}'$, $j=1,2$ are defined as 
\begin{align*}
\mathbb{B}_1:=
\left(
\begin{array}{cc}
	T_1 & -T_2 \\
	T_2^* & T_3 
\end{array}
\right)
,~~~\mathbb{B}_2:=
\left(
\begin{array}{cc}
	J_1 & 0 \\
	J_2 & 0 
\end{array}
\right). 
\end{align*}
To show the existence of solutions, it then suffices to prove that $\mathbb{B}$ is of Fredholm type since the uniqueness follows from Theorem \ref{thm2.4}. By \cite[Lemma 3.2]{GXY18}, $T_1$ is coercive over $[H^1(B_R)]^d$. Combining the regularity of $\mathbf{G}-\boldsymbol{\Pi}_\pm$ and \cite[Theorem 7.6]{WM10}, it can be derived that $T_3$ is a strongly elliptic operator over $[H^{-1/2}(\pa B_R)]^d$, which implies that the real part of $\mathbb{B}_1$, given by 
\begin{align*}
\Rt\mathbb{B}_1:=\frac{\mathbb{B}_1+\mathbb{B}_1^*}{2}=
\left(
\begin{array}{cc}
	T_1 &  \\
	  & T_3 
\end{array}
\right), 
\end{align*}
is strongly elliptic over $\mathbf{X}$. Moreover, by the Sobolev embedding, $J_1$ is compact. Since from our choice for $\tilde{\mu}$ and $\tilde{\lambda}$, the operator $\mathcal{K}:[H^{1/2}(\pa B_R)]^d\rightarrow[H^{1/2}(\pa B_R)]^d$ is compact, which indicates that $J_2$ and thus $\mathbb{B}_2$ is compact. Hence, the operator $\mathbb{B}$ is Fredholm with index zero. In summary, we obtain the following result: 
\begin{theorem}\label{thm4.1}
	For any $\mathbf{z}\in D^+$ and $\mathbf{a}\in\C^d$, problem \eqref{1.1}--\eqref{1.3} with $a_0=1$ admits a unique solution $(\mathbf{u}_+,\mathbf{u}_-)\in H^1(B_R\cap D^+)\times H^1(B_R\cap D^-)$ such that $\mathbf{u}_+=\mathbf{\hat u}$ in $B_R\cap D^+$ and $\mathbf{u}_-=\mathbf{\hat u}+\mathbf{u}^{in}$ in $B_R\cap D^-$. 
\end{theorem}

\section*{Acknowledgements}
This work was supported by the NNSF of China with Grants 12326607 and 1257010856. 
Our manuscript has no associated data, and there is no conflict of interest between authors. 

\end{document}